\documentclass[reqno,centertags,11pt]{amsart}

\usepackage{amssymb,amsmath,amsfonts,amssymb}
\usepackage{hyperref}
\usepackage{enumerate}
\usepackage{tikz}

\usepackage{mathtools}
\mathtoolsset{showonlyrefs=true}

-\marginparwidth \oddsidemargin -9mm \evensidemargin \oddsidemargin

\usepackage{latexsym}
\advance\hoffset by 5mm

\def\@abssec#1{\vspace{.05in}\footnotesize \parindent .2in
{\bf #1. }\ignorespaces}

\newtheorem{theorem}{Theorem}[section]

\newtheorem{lemma}[theorem]{Lemma}

\newtheorem{remark}[theorem]{Remark}

\newcommand{\R}{\ensuremath{\mathbb{R}}}
\newcommand{\Z}{\ensuremath{\mathbb{Z}}}
\newcommand{\T}{\ensuremath{\mathbb{T}}}

\newcommand{\N}{\ensuremath{\mathbb{N}}}
\newcommand{\LL}{\ensuremath{\mathcal{L}}}

\newcommand{\dd}{\mathrm{d}}

\def\Rz{{R_0}}
\def\rz{{r_0}}

\allowdisplaybreaks \numberwithin{equation}{section}

\begin{document}

\title[1D Euler-alignment system with misalignment]{Global regularity for a 1D Euler-alignment system with misalignment}
\author{Qianyun Miao}
\address{School of Mathematics and Statistics, Beijing Institute of Technology, Beijing 100081, P. R. China}
\email{qianyunm@bit.edu.cn}
\author{Changhui Tan}
\address{Department of Mathematics, University of South Carolina,
  Columbia SC 29208, USA}
\email{tan@math.sc.edu}
\author{Liutang Xue}
\address{Laboratory of Mathematics and Complex Systems (MOE), School of Mathematical Sciences, Beijing Normal University, Beijing 100875, P.R. China}
\email{xuelt@bnu.edu.cn}
\date{\today}
\maketitle

\begin{abstract}
We study one-dimensional Eulerian dynamics with nonlocal
alignment interactions, featuring strong short-range alignment, and
long-range misalignment. Compared with the well-studied Euler-alignment
system, the presence of the misalignment brings
different behaviors of the solutions, including the possible creation
of vacuum at infinite time, which destabilizes the solutions.
We show that with a strongly singular
short-range alignment interaction, the solution is globally regular,
despite the effect of misalignment.
\end{abstract}

\section{Introduction}
We are interested in the following one-dimensional pressureless
Euler-alignment system
\begin{align}
  \partial_t \rho + \partial_x ( \rho u) =&~0, \label{EA-rho}\\
  \partial_t u + u\,\partial_x u
  =&\int_\R\phi(x-y)(u(y)-u(x))\rho(y)\dd y,\label{EA-u}
\end{align}
with initial data
\begin{equation}\label{EA-init}
  \big.(\rho,u)\big|_{t=0}(x)=(\rho_0,u_0)(x).
\end{equation}

This system,  first derived in \cite{HT08}, can be viewed as the
macroscopic representation of the agent-based Cucker-Smale model \cite{CS07},
describing the emergent phenomenon of animal flocks.

Here, $\rho$ and $u$ represent the density and velocity, respectively.
The right-hand side of \eqref{EA-u} is the nonlocal alignment
force, where $\phi$ is called the \emph{influence function}.
When $\phi>0$, the velocity $u(x)$ intends to align with $u(y)$ as
time evolves.

Although the global well-posedness theory for the Euler-alignment
system in multi-dimensions is still incomplete (one can see
\cite{DMPW,HeT17,Shv19,TT14} for interesting partial results),
the theory on the 1D Euler-alignment system \eqref{EA-rho}-\eqref{EA-u} has been
well-established in the last decade, under the assumption that
influence function $\phi$ is non-negative, symmetric, and decreasing
in $\R^+$.

The behavior of $\phi$ near the origin plays an important role in the
global regularity of the system.
If $\phi$ is bounded, whether the
solution is globally regular depends on the choice of initial data.
In \cite{CCTT}, a sharp critical threshold on the initial data is
derived, which distinguishes global smooth solutions and finite time
singularity formations.
If $\phi$ is weakly singular, namely unbounded but integrable at the
origin, a different critical threshold has been obtained in
\cite{T20}.
If $\phi$ is strongly singular, namely non-integrable at the origin,
the strong short-range alignment is known to bring dissipation which
prevents finite time singularity formations, for all smooth periodic
initial data which stays away from vacuum ($\rho_0>0$).
Global regularity is shown in \cite{DKRT}, and independently in
\cite{ST17, ST18}.

The non-negativity assumption on $\phi$ is also crucial for the
stability, as well as the long time behavior of the system. Indeed,
one can calculate the dynamics of energy fluctuation
\begin{equation}\label{EnergyFluc}
  \frac{d}{dt}\iint_{\R^2}|u(x)-u(y)|^2\rho(x)\rho(y)\dd x\dd y
  =-\iint_{\R^2}\phi(x-y)|u(x)-u(y)|^2\rho(x)\rho(y)\dd x\dd y.
\end{equation}
If $\phi$ has a positive lower bound, it is easy to see that the
energy fluctuation decays exponentially in time. It leads to a
velocity alignment as time approaches infinity.
In \cite{TT14}, such fast alignment with an exponentially decay rate
has been shown for any $\phi$
which delays sufficiently slow at infinity, such that
$\int_0^\infty\phi(r)dr=+\infty$.
Finally, if $\phi\geq0$ and degenerate (namely compactly supported),
velocity alignment can be shown only for periodic initial data away
from vacuum \cite{DS19}, with a sub-exponential rate of convergence.
\smallskip

In this paper, we focus on a different type of influence function,
which is not necessarily non-negative.
When $\phi(x-y)<0$, the velocity $u(x)$ intends to misalign with
$u(y)$. Such \emph{misalignment} behavior could bring instability to
the system. Indeed, it is easy to see from \eqref{EnergyFluc} that the
energy fluctuation no longer decays in time.
A natural question is, how does the misalignment affect the global
well-posedness and long time behavior of the system.

A typical choice of the influence function of our concern is
\begin{equation}\label{phi-albe}
  \phi_{\alpha,\beta}(x)=\frac{c_\alpha}{|x|^{1+\alpha}} - \mu
  \frac{c_\beta}{|x|^{1+\beta}},
\end{equation}
where the parameter $0<\beta<\alpha<2$, the coefficient $\mu>0$, and
$c_\alpha$, $c_\beta$ are positive constant, defined in \eqref{Lam-alp}.
This influence function has two main features:
\begin{itemize}
  \item Strong alignment in the short range: $\phi_{\alpha,\beta}(x)$
    behaves like $|x|^{-1-\alpha}$ near the origin. More precisely,
    \[\frac{c_\alpha}{2|x|^{1+\alpha}}<\phi_{\alpha,\beta}(x)<\frac{c_\alpha}{|x|^{1+\alpha}},
    \quad\forall~0<|x|<\left(\frac{c_\alpha}{2\mu c_\beta}\right)^{\frac{1}{\alpha-\beta}}.\]
  \item Misalignment in the long range: $\phi_{\alpha,\beta}(x)$ becomes
    negative if $|x|$ is large enough. More precisely,
    \[\phi_{\alpha,\beta}(x)<0,\quad\forall~
    |x|>\left(\frac{c_\alpha}{\mu c_\beta}\right)^{\frac{1}{\alpha-\beta}}.\]
\end{itemize}

The system \eqref{EA-rho}-\eqref{EA-u} with influence function
\eqref{phi-albe} is closely related to the following Burgers type equation
\begin{equation}\label{nonlKS}
  \partial_t u + u\,\partial_x u =
  -\Lambda^\alpha u + \mu \Lambda^\beta u,\quad u|_{t=0}=u_0,
\end{equation}
where the fractional differential operator $\Lambda^\alpha=(-\partial_x^2)^{\frac{\alpha}{2}}$
has the expression formula
\begin{equation}\label{Lam-alp}
  \Lambda^\alpha f(x) = c_\alpha\, \mathrm{p.v.} \int_{\R} \frac{f(x) - f(y)}{|x-y|^{1+\alpha}} \dd y,\quad c_\alpha=\frac{2^\alpha\Gamma(\frac{1+\alpha}{2})}{\sqrt{\pi}|\Gamma(-\frac{\alpha}{2})|}.
\end{equation}

Equation \eqref{nonlKS} can be obtained by formally enforcing
$\rho(x,t)\equiv 1$ in the velocity dynamics \eqref{EA-u} associated with $\phi(x)=\phi_{\alpha,\beta}(x)$.
When $\mu=0$, \eqref{nonlKS} is known as the fractal Burgers
equation. It was studied in \cite{KNS} and global regularity is
obtained if and only if $\alpha\geq1$.

When $\mu>0$, the equation \eqref{nonlKS} can be viewed as a nonlocal
analog of the notable Kuramoto-Sivashinsky equation (which corresponds
to $\alpha=4,\beta=2$ in \eqref{nonlKS}).
The linear pseudo-differential term $ \Lambda^\alpha u - \mu
\Lambda^\beta u$ gives long-wave instability and short-wave
stability. The case where $\alpha>1$ and $\beta<\alpha$ was first
introduced and studied by Granero-Belinch\'on and Hunter in
\cite{GBH}. They proved the global existence, uniqueness and instant
analyticity of solutions and also the existence of a compact attractor
for the equation \eqref{nonlKS}.
We remark that by applying the same process as in \cite{MX19}, one can
show the global well-posedness for the critical the case $\alpha=1$ with
$\beta<1$. Also, finite time blowup can be shown in the case
$0<\alpha,\beta<1$.

For our system \eqref{EA-rho}-\eqref{EA-u}, the constant density
profile $\rho(x,t)\equiv1$ does not preserve in time. For $\mu=0$, a remarkable
discovery in \cite{DKRT} is that, with a density-dependent fractional
dissipation, the global behavior of the solution differs from the
fractal Burgers equation. In particular, global regularity can be obtained for $\alpha\in(0,1)$.

The main goal of this paper is on the global well-posedness of the
Eulerian system \eqref{EA-rho}-\eqref{EA-u}, with the influence function
$\phi$ containing misalignment.
We will focus on periodic initial data $(\rho_0, u_0)$ where $x\in\T$,
and $\rho_0(x)>0$ away from vacuum.
Without loss of generality, we can set the period to be 1, and let
$\T=[-\frac{1}{2}, \frac{1}{2}]$.

As a suitable generalization of example \eqref{phi-albe}, we will consider the influence function $\phi(x)=\phi(-x)$
belonging to $C^4(\R\setminus \{0\})$ which satisfies the following assumptions.

\begin{itemize}
  \item[(A1)] \emph{Strong alignment in the short range:} there exist constants $\alpha\in (0,2)$, $a_0>0$ and $c_1\geq 1$ such that
  \begin{equation}\label{phi-assum1}
    \frac{1}{c_1} \frac{1}{|x|^{1+\alpha}}\leq \phi(x) \leq \frac{c_1}{|x|^{1+\alpha}},\quad \forall~ 0< |x| \leq a_0,
  \end{equation}
\begin{equation}\label{phi-assum1.2}
  \Big|\frac{\dd^j\phi(x)}{\dd x^j}\Big| \leq \frac{c_1}{|x|^{1+j+\alpha}},\;\;j=1,2,3,4,\quad \forall~ 0< |x| \leq a_0,
\end{equation}
  \begin{equation}\label{phi-assum1.3}
    \textrm{the mapping $r\mapsto \phi(r)$ is non-increasing in $r$ on $(0,a_0]$.}
  \end{equation}
\item[(A2)] \emph{Possible misalignment in the long range:} there exists a constant $c_2>0$ such that
\begin{equation}\label{phi-assum2}
  \int_{|x|\geq a_0} |x|^j\Big|\frac{\dd^j \phi(x)}{\dd x^j}\Big| \dd x \leq c_2,\;\;j=0,1,2,3,4.
\end{equation}
\end{itemize}
Such a function is indeed the kernel function of the following L\'evy operator
\begin{equation}\label{Lop-exp}
  \mathcal{L}f(x) = \mathrm{p.v.} \int_\R \phi(x-y) \big(f(x) - f(y) \big) \dd y,
\end{equation}
which corresponds to the infinitesimal generator of stable L\'evy
process (see \cite{Jacob}).

Under the periodic setup,  the alignment term can be expressed as
\[\int_\T\phi^S(x-y)(u(y)-u(x))\rho(y)dy\]
with the \emph{periodic influence function}
\begin{equation}\label{phi-S}
  \phi^S(x):=\sum_{k\in\Z}\phi(x+k),\quad\forall\,x\in\T.
\end{equation}

When $\phi$ satisfies assumptions (A1) and (A2), we assume $a_0 \leq \frac{1}{2}$ with no loss of generality, and noting that 
$\sum_{k\neq 0} |\phi(x+k)| \leq 3 c_2$ for every $x\in \T$ and $\sum_{k\in \Z} |\phi(x+k)| \leq c_2(1+ a_0^{-1})$ for every $|x|\in [a_0,\frac{1}{2}]$,
$\phi^S$ has the following similar properties.
\begin{itemize}
 \item[(A1$^S$)] \emph{Strong alignment in the short range:}
\begin{equation}\label{phi-s-assum1}
  \frac{1}{2c_1} \frac{1}{|x|^{1+\alpha}} \leq \phi^S(x) \leq \frac{2
    c_1}{|x|^{1+\alpha}},   \quad \forall~|x|\in  (0, \rz],
  \quad \rz =\min\left\{a_0, ~\Big(\frac{1}{6 c_1 c_2}\Big)^{\frac{1}{1+\alpha}}\right\}.
\end{equation}
 \item[(A2$^S$)] \emph{Possible misalignment in the long range:}
   \begin{equation}\label{phi-s-assum2}
  |\phi^S(x)| \leq c_3, \quad \forall~ |x| \in [\rz, 1/2],\quad
  c_3=c_1r_0^{-(1+\alpha)}  +c_2\big(1+a_0^{-1}\big).
\end{equation}
\end{itemize}

Condition \eqref{phi-s-assum2} allows $\phi^S$ to be negative in the
long range. This corresponds to the misalignment
effect. Figure~\ref{fig:phi} illustrates a typical periodic influence
function satisfying (A1$^S$) and (A2$^S$) with misalignment.

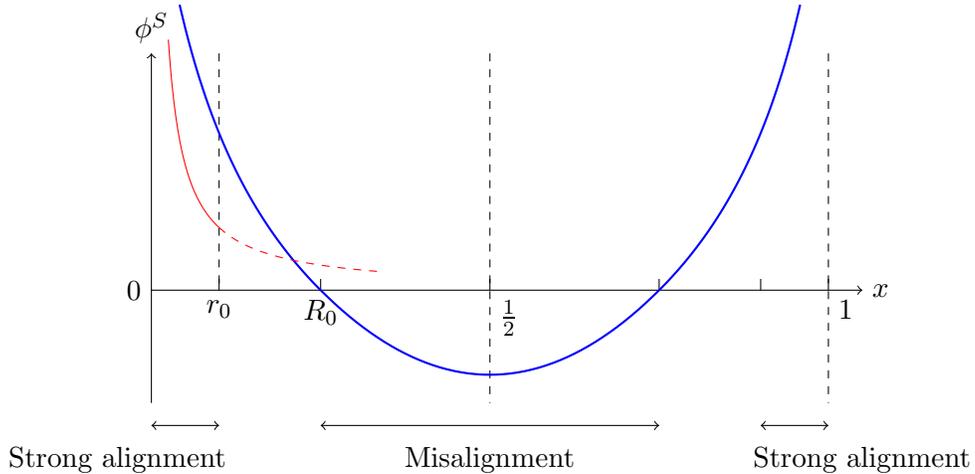
\begin{figure}[ht]
\begin{center}
\begin{tikzpicture}[scale=1.5]

    \coordinate (y) at (0,2.1);
    \coordinate (x) at (6.3,0);
    \draw[<->] (y) node[above] {$\phi^S$} -- (0, 0) --  (x) node[right] {$x$};
    \draw  (0,-1) -- (0, 0) node[left] {$0$};

    \draw (3,0.1) -- (3,0) node[below right] {$\frac{1}{2}$};
    \draw (6,0.1) -- (6,0) node[below right] {$1$};
    \draw (1.5,0.1) -- (1.5,0) node[below] {$\Rz$};
    \draw (.6,0.1) -- (.6,0) node[below] {$\rz$};
    \draw (4.5,0.1) -- (4.5,0); 
    \draw (5.4,0.1) -- (5.4,0); 

    \draw[dashed] (.6,2.1) -- (.6,0);
    \draw[dashed] (3,2.1) -- (3,-1);
    \draw[dashed] (6,2.1) -- (6,-1);

    \draw[thick, domain=0.75:2,smooth,variable=\x,blue]  plot ({\x-.5},{4*((2/\x)^.5-1)});
    \draw[thick, domain=0.75:2,smooth,variable=\x,blue]  plot ({6.5-\x},{4*((2/\x)^.5-1)});
    \draw[thick, domain=1.5:4.5,smooth,variable=\x,blue]  plot ({\x},{(\x-3)^2/3-3/4});
    \draw[domain=0.15:.6,smooth,variable=\x,red]  plot ({\x},{\x^(-1)/3});
    \draw[dashed, domain=.6:2,smooth,variable=\x,red]  plot
    ({\x},{\x^(-1)/3});

    \draw[<->] (0,-1.2) -- (.6, -1.2);
    \node at (-.3,-1.5) {Strong alignment};
    \draw[<->] (1.5,-1.2) -- (4.5, -1.2);
    \node at (3,-1.5) {Misalignment};
    \draw[<->] (5.4,-1.2) -- (6, -1.2);
    \node at (6.3,-1.5) {Strong alignment};
   \end{tikzpicture}
 \end{center}
 \caption{The illustration of the periodic influence function}\label{fig:phi}
\end{figure}

Now, let us state our main result.
\begin{theorem}[Global regularity]\label{thm:GR-EAS}
  Let the symmetric influence function $\phi\in C^4(\R\setminus\{0\})$
  be under assumptions (A1) and (A2) with $0<\alpha<2$.
  Let $s>\frac{3}{2}$ if $\alpha\in (0,1]$ and $s>\frac{5}{2}$ if
  $\alpha\in (1,2)$.
  Assume that the initial data satisfy
  \[\rho_0 \in H^s(\T), \quad \min_\T \rho_0>0,\quad
    u_0 \in H^{s+1 -\alpha}(\T), \quad \text{and}\quad
    G_0 := \partial_x u_0 - \LL \rho_0 \in H^{s-\frac{\alpha}{2}}(\T).\]
  Then for any $T>0$, the Euler-alignment system \eqref{EA-rho}-\eqref{EA-u} with associated periodic initial data
  $(\rho_0,u_0)$ generates a unique global smooth solution $(\rho,u)$ on the time interval $[0,T]$.
\end{theorem}

As a direct corollary, the theorem says that with $\phi(x)=\phi_{\alpha,\beta}(x)$ given by \eqref{phi-albe},
global regularity of the Euler-alignment system \eqref{EA-rho}-\eqref{EA-u}
can be obtained for the full range $0<\beta<\alpha<2$, $\mu>0$. In particular, the behavior differs
from equation \eqref{nonlKS} when $\alpha\in(0,1)$, where blowup can occur.
This is the same phenomenon as the $\mu=0$ case.

We shall emphasize, however, the presence of misalignment makes a big
difference in the regularity estimates, as well as the long time
behaviors of the solutions.

When misalignment effect is relatively weak (e.g. $\mu$ is small in \eqref{phi-albe}), then $\phi^S(x)>0$ for any $x\in\T$. In
this case, there is overall no misalignment. Global regularity and fast
alignment then follow. See related discussions in \cite{KT18}.
In particular, two important bounds can be derived. First, the density
has a uniform-in-time lower bound (see Remark \ref{rmk:lowbdd}), namely, there exists a positive
constant $\rho_m>0$, such that
\[\rho(x,t)\geq \rho_m,\quad \forall~x\in\T~\text{and}~t\geq0.\]
Second, the density oscillation
$\|\partial_x\rho(\cdot,t)\|_{L^\infty}$ is bounded uniformly in time.

When misalignment effect is strong enough (e.g. $\mu$ is big in \eqref{phi-albe}), then $\phi^S$ is not necessarily positive everywhere and the typical case is illustrated in Figure \ref{fig:phi}. With the long-range
misalignment, the density no longer has a uniform-in-time positive
lower bound. Indeed, as verified by numerical experiments, the lower
bound on density
can decay to zero as time approaches infinity.
The presence of vacuum is known to lead to destabilization of the
system, and the singularity formations \cite{T19}. Lack of the uniform
lower bound on density creates additional
difficulties towards the global well-posedness theory.

To prove Theorem \ref{thm:GR-EAS},
we first obtain lower/upper bound estimates on density $\rho$, stated in Lemmas
\ref{lem:lowbdd} and \ref{lem:uppbdd}. It guarantees that the density is uniformly-in-time bounded and also stays positive in any finite
time, although it could go to zero as time approaches infinity, with a
exponential decay rate.
Next, with the lower/upper bound estimates, we establish the local
well-posedness theory, using energy and commutator estimates.
Since we consider a large class of general influence functions $\phi$,
the crucial commutator estimates needs to be extended to general
L\'evy operators $\LL$ that are related to $\phi$.
Moreover, a sufficient condition that ensures the global regularity is shown,
which extends the result in \cite{DKRT, KT18} to a more general
setting. The sufficient condition, described in \eqref{eq:bc}, is
related to the boundedness of the density oscillation
$\|\partial_x\rho(\cdot,t)\|_{L^\infty}$ for the case $\alpha\in (0,1]$ and
$\|\partial^2_x\rho(\cdot,t)\|_{L^\infty}$ for the case $\alpha\in(1,2)$.
Finally, we prove that these density oscillations can be bounded in any
finite time, using the novel method on modulus of continuity, invented
in \cite{KNV} and with applications to the Euler-alignment system in
\cite{DKRT}.
We adapt it to the Euler-alignment system \eqref{EA-rho}-\eqref{EA-u}
with general influence function $\phi$.
There are two major difficulties. First, the case $\alpha\in[1,2)$
does not simply follow the same procedures as the $\alpha\in(0,1)$
case. See Remark~\ref{rmk:Omega} as well as
Section~\ref{subsec:rho-lip} for related discussions.
Second, with the presence of the misalignment, there is a lack of uniform
lower bound on the density, and thus
$\|\partial_x\rho(\cdot,t)\|_{L^\infty}$ and
$\|\partial_x^2\rho(\cdot,t)\|_{L^\infty}$ can grow in time.
We manage to get a bound of $\|\partial_x\rho(\cdot,t)\|_{L^\infty}$
with double exponential growth in time and a bound of
$\|\partial_x^2\rho(\cdot,t)\|_{L^\infty}$ with triple exponential
growth in time.
These bounds ensure the global regularity anyway.
However, the solutions could be very unstable as time approaches infinity.

The rest of the paper is organized as follows.
In Section \ref{sec:lem}, we state and show some important lemmas, including the
critical lower/upper bound estimates on density and some properties of L\'evy operator $\LL$.
In Section \ref{sec:lwp}, we establish the local well-posedness theory,
as well as the blowup criteria.
In Section \ref{sec:gwp}, we show global regularity of the considered system, and finish the
proof of Theorem \ref{thm:GR-EAS}. In Section \ref{sec:MOCes}, we present the detailed proof of
auxiliary lemmas related to modulus of continuity, which play crucial roles in the global regularity part.
Section \ref{sec:append} is appendix section which deals with the commutator estimates that are useful in the local well-posedness.

\section{Auxiliary lemmas}\label{sec:lem}

\subsection{Reformulation of the Euler-alignment system}
The alignment force in \eqref{EA-u} is known to have a commutator structure.
By using expression formula \eqref{Lop-exp} of L\'evy operator $\LL$, it can be written as
\begin{equation*}
  \int_\R\phi(x-y)(u(y)-u(x))\rho(y)\dd y = - \big(\LL (\rho u) - u \LL(\rho) \big)
  = - [ \LL, u ]\rho.
\end{equation*}
Note that in the case $\phi=\phi_{\alpha,\beta}$ given by \eqref{phi-albe}, the corresponding operator $\LL=\Lambda^\alpha - \mu\Lambda^\beta$.

To capture the commutator structure, we follow the idea of \cite{CCTT}.
Apply the operator $\LL$ to the $\rho$-equation \eqref{EA-rho} and get
\begin{equation*}
  \partial_t \LL \rho = -\partial_x \LL (\rho\, u) = - \partial_x ([\LL,u ]\rho) - \partial_x (u\,(\LL\rho)).
\end{equation*}
Apply $\partial_x$ to the $u$-equation \eqref{EA-u} and get
\begin{equation*}
  \partial_t (\partial_x u) + \partial_x(u\,\partial_x u) = - \partial_x ([\LL,u]\rho).
\end{equation*}
Combining these two equations together will yield a nice cancelation on the
term $\partial_x ([\LL,\rho]u)$. Define
\begin{equation}\label{defG}
  G=\partial_xu - \mathcal{L}\rho.
\end{equation}
We get
\begin{equation}\label{EA-G}
  \partial_t G + \partial_x (G\,u) =0,
\end{equation}

The Euler-alignment system \eqref{EA-rho}-\eqref{EA-u} can be
reformulated as the following system for $\rho$ and $G$:
\begin{equation}\label{EAS-ref}
\begin{cases}
  \partial_t \rho + \partial_x (\rho\, u)=0, \\
  \partial_t G + \partial_x (G\, u)=0, \\
  \partial_x u = G + \LL\rho.
\end{cases}
\end{equation}

For smooth solutions $(\rho, G)$, we can reconstruct the velocity $u$
from \eqref{EAS-ref} as follows.

First, by integrating equation \eqref{EA-rho} in $x$, we get the conservation of mass
\begin{equation}\label{b-rho0}
  \int_\T \rho(x,t)\dd x = \int_\T \rho_0(x) \dd x=:\bar{\rho}_0,
\end{equation}
where we denote $\bar{\rho}_0$ as the average density in $\T$.

Since $G$ also satisfies the continuity equation \eqref{EA-G}, we have
\[  \int_\T G(x,t) \dd x = \int_\T G_0(x) \dd x =
  \int_\T\partial_xu_0(x)\dd
  x+\int_\T\int_\T\phi^S(x-y)(\rho_0(x)-\rho_0(y))\dd x\dd y=0.\]
We also set
\begin{equation}\label{theta}
  \theta(x,t)= \rho(x,t) -\bar{\rho}_0,
\end{equation}
so that $\int_\T \theta(x,t) \dd x =0$. Thus we deduce that the primitive functions of $\theta(x,t)$ and $G(x,t)$ are periodic.
Denote by $(\varphi,\psi)$ the mean-free primitive functions of $(\theta,G)$:
\begin{equation}\label{the-varph}
  \theta(x,t)= \partial_x \varphi(x,t),\quad \int_\T \varphi(x,t)\dd x =0,
\end{equation}
and
\begin{equation}\label{G-varpsi}
  G(x,t) = \partial_x \psi(x,t),\quad \int_\T \psi(x,t) \dd x =0.
\end{equation}
Hence, from the relation \eqref{defG}, we see that
\begin{equation}\label{u-exp}
  u(x,t) = \psi(x,t) + \LL \varphi(x,t) + I_0(t).
\end{equation}

In order to determine $I_0(t)$, we make use of the conservation of
momentum. Indeed, from the system \eqref{EA-rho}-\eqref{EA-u}, we
have the dynamics of the momentum
\begin{equation}\label{eq:momentum}
  \partial_t(\rho u) + \partial_x(\rho u^2) = \rho(x)\int_\T \phi^S(x-y) (u(y)-u(x)) \rho(y) \dd y.
\end{equation}
Integrating of \eqref{eq:momentum} on $\T$ and using the fact that $\phi^S$ is an even function on $\T$, it yields
\begin{equation*}
  \frac{\dd}{\dd t}\int_\T \rho u\,\dd x = \int_\T \int_\T \phi^S(x-y) (u(y)-u(x)) \rho(x) \rho(y) \dd x \dd y =0,
\end{equation*}
thus we get
\begin{equation*}
  \int_\T \rho(x,t) u(x,t) \dd x = \int_\T \rho_0(x) u_0(x) \dd x.
\end{equation*}
Such conservation can be used to determine $I_0(t)$ in \eqref{u-exp}:
\begin{equation*}
  I_0(t) = \frac{1}{\bar{\rho}_0 } \left(\int_\T \rho_0(x) u_0(x)
    \dd x - \int_\T \rho(x,t) \psi(x,t) \dd x - \int_\T  \rho(x,t)  \LL\varphi(x,t)\dd x\right).
\end{equation*}
From \eqref{theta}-\eqref{the-varph} and the property of L\'evy operator $\LL$ (e.g. see \eqref{LKf}), we infer that
\begin{equation*}
  \int_\T \rho(x,t) \LL \varphi(x,t) \dd x = \bar{\rho}_0\int_\T \LL\varphi(x,t)\dd x + \int_\T \partial_x\varphi(x,t) \LL\varphi(x,t) \dd x =0,
\end{equation*}
thus
\begin{equation}\label{I0t}
  I_0(t) = \frac{1}{\bar{\rho}_0 } \left(\int_\T \rho_0(x) u_0(x) \dd x - \int_\T \rho(x,t) \psi(x,t) \dd x \right).
\end{equation}
In particular, if $G(x,t)\equiv 0$ then we have $\psi(x,t) \equiv 0$ and $I_0(t)$ is just a time-independent constant.

\subsection{Bounds on the density}\label{subsec:bdd-rho}
We first derive the crucial lower bound on $\rho$, which guarantees no creation of vacuum at finite time.
\begin{lemma}\label{lem:lowbdd}
  Assume the influence function $\phi(x)=\phi(-x)\in
  C^4(\R\setminus\{0\})$ satisfies assumptions (A1) and (A2) with $\alpha\in (0,2)$.
  Let $(\rho,u)$ be a smooth solution to the Euler-alignment system \eqref{EA-rho}-\eqref{EA-init} for $0\leq t\leq T$,
with smooth periodic initial data $(\rho_0,u_0)$ satisfying $\min_\T\rho_0(x) >0 $.
Then, there exists a positive constant
$M_0>0$, depending only on $c_3$ and the initial data, such that
\begin{equation}\label{eq:lowbdd}
  \rho(x,t) \geq M_0 e^{-c_3\bar{\rho}_0 t},\quad \forall~x\in \T, ~0\leq t\leq T.
\end{equation}
\end{lemma}

\begin{proof}
  We first observe that the quantity $F=G/\rho$ satisfies the following transport equation
\begin{equation}\label{Feq}
  \partial_t F + u\,\partial_x F=0,
\end{equation}
and it yields that
\begin{equation}\label{Flinf-es}
  \|F(t)\|_{L^\infty(\T)} \leq \|F_0\|_{L^\infty(\T)} = \Big\|\frac{\partial_x u_0 -\LL \rho_0 }{\rho_0}\Big\|_{L^\infty(\T)} <\infty.
\end{equation}

Note also that $\rho$ satisfies
\begin{equation}\label{rho-eq2}
  \partial_t \rho + u\,\partial_x \rho = -\rho\partial_x u = - \rho \LL\rho - \rho^2 F.
\end{equation}
Assume $T_*\leq T$ is the maximal time that $\min_{x\in \T}\rho(x,t)$ remains strictly positive.
The positiveness of $T_*$ is ensured by $\min_{x\in \T}\rho_0>0$
and the smoothness of $\rho$. For every $0\leq t\leq T_*$, we assume that $\underline{x}\in \T$ is a point that $\theta(x,t)$ attains its minimum
($\underline{x}$ maybe is dependent on $t$ and is not necessarily unique).
By virtue of formula \eqref{Lop-exp} and \eqref{phi-S}, we see that
\begin{equation*}
  - \mathcal{L}\rho(\underline{x},t) = \mathrm{p.v.}\int_\R \phi(\underline{x}-y) (\rho(y,t) - \rho(\underline{x},t)) \dd y
  = \mathrm{p.v.}\int_\T\phi^S(y) (\rho(y + \underline{x},t) - \rho(\underline{x},t)) \dd y,
\end{equation*}
where $\phi^S$ satisfies estimates \eqref{phi-s-assum1}-\eqref{phi-s-assum2}.
Since $-c_3<0$ is a lower bound of $\phi^S$ on $\T$, we have
\begin{equation}\label{Lrho-es1}
\begin{split}
  - \mathcal{L}\rho(\underline{x},t) \geq  -c_3 \int_\T \big( \rho(y+ \underline{x},t) - \rho(\underline{x},t) \big)\dd y
  = - c_3 \big( \bar{\rho}_0 - \rho(\underline{x},t)\big).
\end{split}
\end{equation}
Combining \eqref{rho-eq2} with \eqref{Flinf-es} and \eqref{Lrho-es1}, we obtain
\begin{equation*}
  \partial_t \rho(\underline{x},t) \geq -c_3 \bar{\rho}_0 \,\rho(\underline{x},t)
  -  \|F_0\|_{L^\infty}  \rho(\underline{x},t)^2.
\end{equation*}
Direct calculation then yields
\[\min_{x\in \T}\rho(x,t)\geq\frac{c_3\bar{\rho}_0}{(c_3\bar{\rho}_0 (\min_\T
    \rho_0)^{-1}+\|F_0\|_{L^\infty})e^{c_3\bar{\rho}_0t}-\|F_0\|_{L^\infty}}
  \geq\frac{c_3\bar{\rho}_0}{c_3\bar{\rho}_0 (\min_\T
    \rho_0)^{-1}+\|F_0\|_{L^\infty}}e^{-c_3\bar{\rho}_0t},\]
for any $0\leq t\leq T_*$.
Moreover, the above formula implies that $T_*=T$. So \eqref{eq:lowbdd}
holds as long as the solution stays smooth.
\end{proof}

\begin{remark}\label{rmk:lowbdd}
If the periodic influence function $\phi^S$ has a non-negative lower
bound on $\T$, that is,
\[
  \phi^S(x) \geq \phi_m,\quad \forall~ x\in \T,\quad \textrm{with some
    constant $\phi_m\geq0$},
\]
a similar estimate as \eqref{Lrho-es1} implies
\[- \mathcal{L}\rho(\underline{x},t) \geq
  \phi_m \big( \bar{\rho}_0 - \rho(\underline{x},t)\big).\]
Consequently, we have
\[\partial_t\rho(\underline{x},t)\geq
  \phi_m\bar{\rho}_0\,\rho(\underline{x},t)
  - \big(\|F_0\|_{L^\infty} + \phi_m \bar{\rho}_0 \big)
  \rho(\underline{x},t)^2,\]
where the right hand side stays positive if
$\rho(\underline{x},t)<\frac{\phi_m\bar{\rho}_0}{\phi_m\bar{\rho}_0+\|F_0\|_{L^\infty}}$.
This leads to a uniform-in-time lower bound on $\rho$
\begin{equation*}
  \min_{\T\times [0,T^*]}\rho(x,t) \geq \min\Big\{\min_\T \rho_0, ~\frac{\phi_m \bar{\rho}_0}{\|F_0\|_{L^\infty} + \phi_m \bar{\rho}_0} \Big\}.
\end{equation*}

Compared with Lemma~\ref{lem:lowbdd}, we observe a major difference
between systems with or without misalignment. Lack of uniform-in-time
lower bound on the density brings additional difficulties to the
local and global well-posedness theory.
\end{remark}

Next we show a uniform upper bound of density $\rho$.
\begin{lemma}\label{lem:uppbdd}
Let the assumptions of Lemma \ref{lem:lowbdd} be satisfied.
Then, there exists a positive constant $M_1>0$ dependent on $\alpha$, $\rz$, $c_1$, $c_3$, and $(\rho_0,u_0)$ but independent of $T$ such that
\begin{equation}\label{eq:uppbdd}
  \rho(x,t) \leq M_1, \quad \forall\, x\in \T,~ 0\leq t\leq T.
\end{equation}
\end{lemma}

\begin{proof}
  Assume that for every $0\leq t \leq T$, smooth solution $\theta(x,t)$ attains its maximum at some point $\overline{x}\in\T$
($\overline{x}$ maybe depends on $t$ and not necessarily is unique). We also have \eqref{rho-eq2} as the equation of $\rho$,
and we first intend to derive an upper bound of $\LL \rho(\overline{x},t)$, which has the following formula
\begin{equation*}
  \mathcal{L}\rho(\overline{x},t) = \mathrm{p.v.}\int_\T\phi^S(z) (\rho(\overline{x},t) - \rho(\overline{x}+z,t) ) \dd y.
\end{equation*}
The estimates \eqref{phi-s-assum1}-\eqref{phi-s-assum2} of $\phi^S$ ensure that
\begin{align}\label{eq:Lx-es}
  \,\mathcal{L}\rho(\overline{x},t)
  \geq & \,\mathrm{p.v.} \int_{|z|\leq r_0}  \frac{c_1}{2|z|^{1+\alpha}}(\rho(\overline{x},t) - \rho(\overline{x}+z,t) ) \dd y
  + \int_{r_0\leq |y|\leq \frac{1}{2}} (-c_3)  (\rho(\overline{x},t) - \rho(\overline{x}+z,t) )\dd y \nonumber \\
  \geq & \,\mathrm{p.v.} \int_{|z|\leq r_0}  \frac{c_1}{2|z|^{1+\alpha}}(\theta(\overline{x},t) - \theta(\overline{x}+z,t) ) \dd y
  - c_3 (1-2r_0) \rho(\overline{x},t) .
\end{align}
In order to estimate the integral on the right hand side of \eqref{eq:Lx-es}, we use the idea of nonlinear maximum principle originated in \cite{ConV}. Set $\varpi\in C^\infty(\R)$ be a test function such that
\begin{equation}\label{chi-prop}
  \textrm{$0\leq \varpi\leq 1$, \quad $\varpi\equiv 0$ on
    $[-1/2,1/2]$,\quad $\varpi\equiv 1$ on $\R\setminus [-1,1]$,
    \quad and \quad $\|\varpi'\|_{L^\infty(\R)}\leq4$.}
\end{equation}
Denote $\varpi_r(x)= \varpi(\frac{x}{r})$ for every $r>0$.
Let $r\in (0,\frac{r_0}{2})$ be a constant to be chosen later.
In view of \eqref{the-varph} and the fact that
\begin{equation}\label{vphi-Linf-es}
  \|\varphi(t)\|_{L^\infty(\R)}=\|\varphi(t)\|_{L^\infty(\T)}\leq \|\theta(t)\|_{L^1(\T)}
  \leq \|\rho(t)\|_{L^1(\T)} + \bar{\rho}_0=2\bar{\rho}_0,
\end{equation}
we use the integration by parts to infer that
\begin{align*}
  \mathcal{L}\rho(\overline{x},t) \geq & \, \mathrm{p.v.} \int_\R  \frac{c_1}{2|z|^{1+\alpha}} \varpi_r(z)(1-\varpi_{r_0}(z))
  \big(\theta(\overline{x},t) - \theta(\overline{x}+z,t) \big) \dd z - c_3 \rho(\overline{x},t)  \\
  \geq & \,\theta(\overline{x},t) \int_{r\leq |z|\leq \frac{r_0}{2} } \frac{c_1}{2|z|^{1+\alpha}}  \dd z - \int_\R \frac{c_1}{2 |z|^{1+\alpha}} \varpi_r(z) (1-\varpi_{r_0}(z)) \partial_z \varphi(\overline{x}+z,t) \dd z   - c_3 \rho(\overline{x},t) \\
  \geq &\, \frac{c_1}{\alpha}(\rho(\overline{x},t) -\bar{\rho}_0) \Big(r^{-\alpha} - \Big(\frac{r_0}{2}\Big)^{-\alpha} \Big)-
         \frac{c_1}{2}\|\varphi(t)\|_{L^\infty(\R)}  \int_\R
         \Big|\partial_z\Big(\frac{\varpi_r(z)(1-\varpi_{r_0}(z)}{|z|^{1+\alpha}}\Big)\Big|\dd z -c_3\rho(\overline{x},t) \\
  \geq & \,\frac{c_1}{2\alpha} \Big(r^{-\alpha} - \Big(\frac{r_0}{2}\Big)^{-\alpha} \Big) \rho(\overline{x},t)
  -\frac{80c_1}{\alpha}\bar{\rho}_0r^{-(1+\alpha)}-c_3 \rho(\overline{x},t),
\end{align*}
where in the last inequality we assume
$\rho(\overline{x},t)\geq2\bar{\rho}_0$ so that
$\rho(\overline{x},t)-\bar{\rho}_0\geq\frac12\rho(\overline{x},t)$,
and also
\[
  \int_\R\Big|\partial_z\Big(\frac{\varpi_r(z)(1-\varpi_{r_0}(z)}{|z|^{1+\alpha}}\Big)\Big|\dd
  z
  \leq
2\left[\left(\frac4r+\frac4{r_0}\right)\cdot\frac{1}{\alpha}\left(\frac{r}{2}\right)^{-\alpha}+\left(\frac{r}{2}\right)^{-(1+\alpha)}\right]\leq \frac{80}{\alpha}r^{-(1+\alpha)}.\]

Now, let us pick $r$ satisfying $\frac{c_1}{4\alpha} \rho(\overline{x},t) r^{-\alpha} = \frac{80c_1}{\alpha}\bar{\rho}_0r^{-(1+\alpha)}$, that is
\[r = \frac{320\bar{\rho}_0}{\rho(\overline{x},t)},\]
and we may also assume that $\rho(\overline{x},t)> \frac{640 \bar{\rho}_0}{r_0}$ so that $r\in (0, \frac{r_0}{2})$,
then, it follows that 
\begin{equation}\label{Lrho-lbd}
  \mathcal{L}\rho(\overline{x},t) \geq  \frac{ c_1}{5\cdot 10^5 \alpha}  \bar{\rho}_0^{-\alpha} \rho(\overline{x},t)^{1+\alpha}
  - \Big( c_3 +  \frac{2 c_1}{\alpha} r_0^{-\alpha}\Big)\rho(\overline{x},t).
\end{equation}

Now from the equation \eqref{rho-eq2}, by using \eqref{Flinf-es} and \eqref{Lrho-lbd}, we directly have
\[
  \partial_t \rho(\overline{x},t) \leq  - \rho(\overline{x},t) \LL \rho(\overline{x},t) + \|F(t)\|_{L^\infty} \rho(\overline{x},t)^2
  \leq  -  \frac{c_1}{5\cdot 10^5 \alpha} \bar{\rho}_0^{-\alpha} \rho(\overline{x},t)^{2+\alpha} +  \Big(c_3 + \frac{2 c_1}{\alpha} r_0^{-\alpha}+ \|F_0\|_{L^\infty}\Big) \rho(\overline{x},t)^2.
\]
If we additionally assume that $\rho(\overline{x},t)$ is large enough so that
\begin{equation}\label{rho-barx-cd}
  \rho(\overline{x},t) \geq \Big(10^6c_1^{-1}\big(c_3\alpha + 2 c_1 r_0^{-\alpha} +\|F_0\|_{L^\infty}\alpha \big)\Big)^{\frac1\alpha}\bar{\rho}_0,
\end{equation}
we get
\begin{align*}
  \partial_t \rho(\overline{x},t)\leq  -  \frac{c_1}{10^6 \,\alpha} \bar{\rho}_0^{-\alpha} \rho(\overline{x},t)^{2+\alpha} <0.
\end{align*}
Therefore, noting that the condition \eqref{rho-barx-cd} implies $\rho(\overline{x},t)\geq \max\{2, 1000\,r_0^{-1}\} \bar{\rho}_0$, we conclude the desired uniform-in-time upper bound
\begin{equation*}
  \rho(\overline{x},t) \leq \max\left\{\max_\T \rho_0,\,\,
  \bar{\rho}_0\cdot \Big(10^6c_1^{-1}\big(c_3\alpha + 2 c_1 r_0^{-\alpha} +\|F_0\|_{L^\infty}\alpha \big)\Big)^{\frac1\alpha}\right\}.
\end{equation*}
\end{proof}

As a direct consequence of Lemmas \ref{lem:lowbdd} and \ref{lem:uppbdd}, we see that
\begin{equation}\label{I0t-bdd}
  |I_0(t)|\leq C,\quad \forall t\in[0,T],
\end{equation}
with $C$ depending only on the influence function $\phi$ and the
initial data $(\rho_0,u_0)$.
Indeed, in light of relation \eqref{G-varpsi} and estimates \eqref{Flinf-es}, \eqref{eq:uppbdd}, we get
\begin{equation}\label{psi-Linf-es}
  \|\psi(t)\|_{L^\infty(\T)}\leq \|G(t)\|_{L^\infty} \leq \|F(t)\|_{L^\infty} \|\rho(t)\|_{L^\infty} \leq M_1 \|F_0\|_{L^\infty} ,
\end{equation}
thus from \eqref{I0t} and \eqref{b-rho0} it yields
\begin{equation*}
  |I_0(t)| \leq \frac{1}{\bar{\rho}_0} \Big(\|u_0\|_{L^\infty} \int_\T\rho_0(x)\dd x + \|\psi(t)\|_{L^\infty} \int_\T \rho(x,t)\dd x\Big)\leq \|u_0\|_{L^\infty} + M_1 \|F_0\|_{L^\infty}.
\end{equation*}

\subsection{Some properties of L\'evy operator $\LL$}
Throughout this subsection, we assume that $\LL$ is the L\'evy operator defined by \eqref{Lop-exp}
with kernel function $\phi(x)=\phi(-x)\in C^4(\R\setminus\{0\})$ satisfying assumptions (A1)(A2) with $\alpha\in (0,2)$.

By taking the Fourier transform on $\mathcal{L}$, we get
\begin{equation}\label{Lsymb}
  \widehat{\LL \, f} (\zeta) = A(\zeta)  \widehat{f}(\zeta),\quad \forall \zeta\in\R,
\end{equation}
where the symbol $A(\zeta)$ is given by the L\'evy-Khintchine formula (see \cite[Eq. 3.217]{Jacob})
\begin{equation}\label{LKf}
  A(\zeta) := \int_{\R\setminus \{0\}}\left( 1- \cos(\zeta\, x)\right) \phi(x)\dd x.
\end{equation}
The next lemma concerns the pointwise lower/upper bound estimates of the symbol. 
\begin{lemma}\label{lem:symb}
The symbol $A(\zeta)$ given by \eqref{LKf} of the considered L\'evy operator $\LL$ satisfies that
\begin{equation}\label{A-est}
  A(\zeta)\geq  C'^{-1} |\zeta|^{\alpha} - C'/2,\quad \forall \zeta\in\R,
\end{equation}
and
\begin{equation}\label{A-est2}
  A(\zeta)\leq  C |\zeta|^{\alpha} + C,\quad \forall \zeta\in\R,
\end{equation}
where $\alpha\in (0,2)$ and $C$, $C'$ are positive constants depending only on $\alpha$ and $a_0,c_1,c_2$.
\end{lemma}

\begin{remark}\label{rmk:symb}
  From estimate \eqref{A-est}, it is clear that $C'+ A(\zeta) $ is strictly positive.
We thus can define the operator $\sqrt{C'\mathrm{Id} + \LL}$ as the following multiplier operator
\begin{equation}\label{def:sqrtL}
  \mathcal{F}\big(\sqrt{C'\mathrm{Id} + \LL\,} f\big)(\zeta) = \sqrt{C' + A(\zeta)} \widehat{f}(\zeta),\quad \forall\zeta\in\R.
\end{equation}
\end{remark}

\begin{proof}[Proof of Lemma \ref{lem:symb}]
  Recalling that for every $\alpha\in (0,2)$ we have (e.g., see \cite[Eq. (3.219)]{Jacob})
\begin{equation}\label{eq:fact1}
  |\zeta|^\alpha = c_{\alpha}\;   \int_{\R\setminus\{0\}}\left( 1-\cos(x\, \zeta)\right) \frac{1}{|x|^{1+\alpha}} \dd x,\quad \forall \zeta\in\R,
\end{equation}
and by virtue of the conditions \eqref{phi-assum1} and \eqref{phi-assum2}, we obtain
\begin{align*}
  A(\zeta) 
  & \geq  c_1^{-1}\;  \int_{0<|x|\leq a_0} \left( 1-\cos(x\, \zeta)\right) \frac{1}{|x|^{1+\alpha}} \dd x
  - \int_{|x|\geq a_0} \big(1-\cos(x\,\zeta)\big) |\phi(x)| \dd x \\
  & \geq  c_1^{-1}\;  \int_{|x|>0} \left( 1-\cos(x\, \zeta)\right) \frac{1}{|x|^{1+\alpha}} \dd x
  - c_1^{-1} \int_{|x|\geq a_0} \big(1-\cos(x\,\zeta)\big) \frac{1}{|x|^{1+\alpha}} \dd x - \int_{|x|\geq a_0} |\phi(x)|\dd x \\
  & \geq c_1^{-1}  c_\alpha^{-1} |\zeta|^\alpha  - \frac{2 }{\alpha} c_1^{-1} a_0^{-\alpha} - c_2 ,
\end{align*}
and
\begin{align*}
  A(\zeta) 
  & \leq  c_1\;  \int_{0<|x|\leq a_0} \left( 1-\cos(x\, \zeta)\right) \frac{1}{|x|^{1+\alpha}} \dd x
  + \int_{|x|\geq a_0} \big(1-\cos(x\,\zeta)\big) |\phi(x)| \dd x \\
  & \leq  c_1\;  \int_{|x|>0} \left( 1-\cos(x\, \zeta)\right) \frac{1}{|x|^{1+\alpha}} \dd x
  + 2 \int_{|x|\geq a_0} |\phi(x)| \dd x  \leq c_1  c_\alpha^{-1} |\zeta|^\alpha  + 2 c_2 ,
\end{align*}
as desired.
\end{proof}

The differentiability property of $\phi(x)$ in assumptions (A1)(A2) is mainly used to show the following property of symbol $A(\zeta)$.
\begin{lemma}\label{lem:A-diff}
The symbol $A(\zeta)$ given by \eqref{LKf} of the considered L\'evy operator $\LL$ satisfies that for $n=1,2,3,4$,
\begin{equation}\label{ndAzeta-es0}
  \Big|\frac{\dd^n A(\zeta)}{\dd \zeta^n} \Big| \leq
  \begin{cases}
    C |\zeta|^{\alpha-n},\quad & \textrm{for   }\;|\zeta|\geq \max\{a_0^{-1},1\}, \\
    C |\zeta|^{-n},\quad & \textrm{for   }\;|\zeta|\leq \max\{a_0^{-1},1\},
  \end{cases}
\end{equation}
where $C>0$ is a constant depending only on coefficients $\alpha,a_0 ,c_1,c_2$ in $\LL$.
\end{lemma}

\begin{remark}\label{rmk-lem:Adif}
  Based on Lemmas \ref{lem:symb} and \ref{lem:A-diff}, we find that for $n=1,2,3,4$,
\begin{equation}\label{nd-sqrAz-es}
  \Big|\frac{\dd^n\sqrt{C' + A(\zeta)}}{\dd \zeta^n}\Big| \leq
  \begin{cases}
    C |\zeta|^{\frac{\alpha}{2}-n},\quad & \textrm{for   }\;|\zeta| \geq \max\{a_0^{-1}, 1\}, \\
    C |\zeta|^{-n},\quad & \textrm{for   }\;|\zeta| \leq \max\{a_0^{-1}, 1\},
  \end{cases}
\end{equation}
where $C>0$ is a constant depending only on coefficients $\alpha,C',a_0 ,c_1,c_2$.
\end{remark}

\begin{proof}[Proof of Lemma \ref{lem:A-diff}]
  Let $\varpi(x)\in C^\infty$ be a test function satisfying \eqref{chi-prop}, and set $\varpi_r(x)= \varpi(\frac{x}{r})$ with $r > 0$.
From \eqref{LKf} and the integration by parts, we infer that
\begin{align}\label{A'zeta-es2}
  |A'(\zeta)| & = \Big|\int_{\R\setminus \{0\}} \partial_\zeta(1-\cos(x\, \zeta)) \,\varpi_r(x) \phi(x) \dd x + \int_{\R\setminus \{0\}} \partial_\zeta(1-\cos(x\,\zeta)) \,\big(1-\varpi_r(x)\big) \phi(x) \dd x \Big| \nonumber \\
  & \lesssim \frac{1}{|\zeta|} \int_{\R\setminus\{0\}}(1-\cos(x\, \zeta)) \big|\partial_x \big(x\,\varpi_r(x) \phi(x)\big)\big| \dd x + \int_{0<|x|\leq r} |x| |\sin (x\,\zeta)|  \big(1-\varpi_r(x)\big) |\phi(x)| \,\dd x \nonumber \\
  & \lesssim \frac{1}{|\zeta|} \int_{\R\setminus \{0\}} \Big( \varpi_r(x) |\phi(x)| + \frac{1}{r} \Big|\varpi'(\frac{x}{r})\Big| \, |x|\,|\phi(x)|
  + \varpi_r(x) |x|\, |\phi'(x)|\Big)\dd x   \nonumber \\
  & \mbox{}\quad +\int_{0<|x|\leq r, |x\,\zeta|\leq 1} |x|^2 |\zeta| |\phi(x)|\dd x + \int_{0<|x|\leq r, |x\,\zeta|\geq 1} |x| |\phi(x)|\dd x.
\end{align}
If the spectrum $|\zeta|$ is large enough, that is, $|\zeta|\geq \max\{a_0^{-1},1\}$, we let $r\leq \min\{a_0, |\zeta|^{-1}, 1\} = |\zeta|^{-1}$ and thus
\begin{align*}
  |A'(\zeta)| \lesssim &\frac{1}{|\zeta| r} \int_{\frac{r}{2}\leq |x| \leq r} \frac{c_1}{|x|^\alpha}\dd x
  +  \frac{1}{|\zeta|} \int_{\frac{r}{2}\leq |x|\leq a_0}  \frac{c_1}{|x|^{1+\alpha}} \dd x
  +  \frac{1}{|\zeta|} \int_{|x|\geq a_0}  \big(|\phi(x)| + |x| |\phi'(x)| \big) \dd x \\
  & + |\zeta| \int_{0<|x|\leq r, |x\,\zeta|\leq 1} |x|^2|\phi(x)|\dd x \\
  \lesssim & \frac{1}{|\zeta| r^\alpha} + \frac{1}{|\zeta|} + |\zeta| r^{2-\alpha} \lesssim \frac{1}{|\zeta| r^\alpha}  + |\zeta| r^{2-\alpha}.
\end{align*}
By choosing $r$ to be $ \frac{1}{2 |\zeta|}$, we conclude that $|A'(\zeta)| \leq C |\zeta|^{\alpha-1}$. If $|\zeta|$ is such that $|\zeta|\leq \max\{a_0^{-1},1\}$,
we set $r= \min\{a_0, 1\}$ (which satisfies $r\leq |\zeta|^{-1}$), and from \eqref{A'zeta-es2} we directly have
\begin{equation}\label{A'zeta-es3}
  |A'(\zeta)| \lesssim \frac{1}{|\zeta|}\int_{\frac{r}{2}\leq |x| \leq a_0} \frac{c_1}{|x|^{1+\alpha}} \dd x
  + \frac{1}{|\zeta|}\int_{|x|\geq a_0} \big(|\phi(x)| + |x| |\phi'(x)| \big) \dd x
  + \int_{0<|x|\leq r} \frac{c_1}{|x|^{\alpha -1}} \dd x \lesssim |\zeta|^{-1}.
\end{equation}
Hence \eqref{ndAzeta-es0} with $n=1$ follows.

Concerning higher-order derivatives $A^{(n)}(\zeta)$, $n=2,3,4$, by using conditions \eqref{phi-assum1}-\eqref{phi-assum1.2} and \eqref{phi-assum2}, we obtain
\begin{align*}
  & |A^{(n)}(\zeta)| = \Big|\int_{\R\setminus \{0\}} \partial_\zeta^n(1-\cos(x\, \zeta)) \,\varpi_r(x) \phi(x) \dd x + \int_{\R\setminus \{0\}} \partial_\zeta^n(1-\cos(x\,\zeta)) \,\big(1-\varpi_r(x)\big) \phi(x) \dd x \Big| \\
  & \lesssim \frac{1}{|\zeta|^n} \int_{\R\setminus\{0\}} (1-\cos(x \,\zeta)) \big|\partial_x^n \big( x^n\varpi_r(x) \phi(x)\big) \big| \dd x
  + \int_{0<|x|\leq r} |x|^n \big(1-\varpi_r(x)\big) |\cos (x\,\zeta + \frac{n\pi}{2})|  |\phi(x)| \,\dd x \\
  & \lesssim \frac{1}{|\zeta|^n} \int_{|x|\geq \frac{r}{2}} \Big( |\phi(x)| + |x| |\phi'(x)| +\cdots + |x|^n |\phi^{(n)}(x)| \Big)\dd x +
  \int_{0<|x| \leq r} |x|^n |\phi(x)| \dd x .
\end{align*}
If $\zeta$ is such that $|\zeta|\geq \max\{a_0^{-1},1\}$, we set $r = 2|\zeta|^{-1}$ (which satisfies $r \leq \min\{a_0,1\}$), and then
\begin{align*}
  |A^{(n)}(\zeta)| & \lesssim \frac{1}{|\zeta|^n} \int_{\frac{r}{2}\leq|x|\leq a_0} \frac{c_1}{|x|^{1+\alpha}} \dd x
  + \frac{1}{|\zeta|^n} \sum_{j=0}^n \int_{|x|\geq a_0} |x|^j |\phi^{(j)}(x)| \dd x + \int_{0<|x|\leq r} c_1|x|^{n-1-\alpha} \dd x \\
  & \lesssim \frac{1}{|\zeta|^n r^\alpha} + \frac{1}{|\zeta|^n} + r^{n-\alpha} \lesssim |\zeta|^{\alpha-n} .
\end{align*}
If $|\zeta|\leq \max\{a_0^{-1},1\}$, we also let $r= \min\{a_0,1\}$, and it yields $A^{(n)}(\zeta) \lesssim |\zeta|^{-n}$ similarly as deriving \eqref{A'zeta-es3}.
Hence the desired estimate \eqref{ndAzeta-es0} follows by combining the above two estimates.
\end{proof}

As an application of Lemma \ref{lem:A-diff}, we have the $L^\infty$-boundedness property of the L\'evy operator $\LL$.
\begin{lemma}\label{lem:Lop-Linf}
  There exists a constant $C>0$ depending only on $\alpha$ such that the considered L\'evy operator $\LL$ satisfies
\begin{equation}\label{eq:Lop-Linf}
  \|\LL f\|_{L^\infty} \leq C \|f\|_{B^\alpha_{\infty,1}},
\end{equation}
where $B^\alpha_{\infty,1}$ denotes the Besov space (see \eqref{Besov-spr} below for definition).
\end{lemma}

\begin{proof}
  We here adopt the notations of Littelewood-Paley theory introduced in the appendix. Denoting by $k_0 := [a_0^{-1}]+1$, and using estimate \eqref{ndAzeta-es0},
the result of \cite[Lemma 2.2]{BCD11} directly implies that for every $k\geq k_0$ and for every $p\in [1,\infty]$,
\begin{align*}
  \|\Delta_k \LL f\|_{L^p} \leq C 2^{k \alpha} \|\Delta_k f\|_{L^p},
\end{align*}
with $C>0$ a constant depending on the coefficients in $\LL$. For the operator $\chi(2^{-k_0}D) \LL$, its kernel function
$\widetilde{h}_{k_0}(x) = C_0 \int_\R e^{i x\,\zeta} \chi(2^{-k_0}\zeta) A(\zeta) \dd \zeta$ indeed satisfies that $\|\widetilde{h}_{k_0}\|_{L^1} \leq C$
(due to Lemma \ref{lem:A-diff} and from an easy computation as showing \eqref{eq:claim}),
so that we have that for every $p\in [1,\infty]$,
\begin{align*}
  \| \chi(2^{-k_0}D)\LL f \|_{L^p} \leq C \| f\|_{L^p} .
\end{align*}
Thus the desired estimate \eqref{eq:Lop-Linf} follows from the high-low frequency decomposition:
\begin{align*}
  \|\LL f\|_{L^\infty} & \leq \|\chi(2^{-k_0}D) f\|_{L^\infty} + \sum_{k\geq k_0} \|\Delta_k \LL f\|_{L^\infty} \\
  & \leq C \|f\|_{L^\infty} + C \sum_{k\geq k_0} 2^{k\alpha} \|\Delta_k f\|_{L^\infty} \leq C \|f\|_{B^\alpha_{\infty,1}}.
\end{align*}
\end{proof}

\section{Local well-posedness}\label{sec:lwp}

In this section, we establish the local well-posedness result for the smooth solution to the Euler-alignment system \eqref{EA-rho}-\eqref{EA-init}.
\begin{theorem}\label{thm:lwp}
   Assume the influence function $\phi(x)=\phi(-x)\in
   C^4(\R\setminus\{0\})$ satisfies assumptions (A1) and (A2) with $\alpha\in (0,2)$.
Let $s>\frac{3}{2}$ if $\alpha\in (0,1]$ and let $s>\frac{5}{2}$ if $\alpha\in (1,2)$. Suppose that the initial data $(\rho_0,u_0)$ satisfies
\begin{equation*}
  \rho_0\in H^s(\T),\quad \min_\T \rho_0>0,\quad G_0:= \partial_x u_0 -\LL \rho_0 \in H^{s-\frac{\alpha}{2}}(\T).
\end{equation*}
Then there exists a time $T_0>0$ depending only on $\phi$ and $(\rho_0,u_0)$ such that the system
\eqref{EAS-ref} admits a unique strong solution $(\rho(x,t),u(x,t))$ on $[0,T_0]$, which satisfies
\begin{equation*}
  \rho\in C([0,T_0]; H^s(\T))\cap L^2([0,T_0]; H^{s+\frac{\alpha}{2}}(\T)),\quad u\in C([0,T_0]; H^{s+1-\alpha}(\T)).
\end{equation*}

Moreover, let $T^*>0$ be the maximal existence time of the above strong solution $(\rho,u)$, then if $T^*<\infty$, we necessarily have
\begin{equation}\label{eq:bc}
\begin{cases}
  \int_0^{T^*}\|\partial_x \rho(t)\|_{L^\infty(\T)}^2 \dd t =\infty,& \quad \textrm{for   }\;\alpha\in (0,1],\\
  \int_0^{T^*} \|\partial_x^2 \rho(t)\|_{L^\infty(\T)}^2 \dd t =\infty,&\quad \textrm{for   }\;\alpha\in (1,2).
\end{cases}
\end{equation}
\end{theorem}

\begin{proof}
  The proof of Theorem \ref{thm:lwp} uses the same procedure as that of
  \cite[Theorem 3.1]{DKRT}, taking into account the misalignment
  effect. We deal with a general class of L\'evy
  operator $\LL$ with the larger scope of $\alpha$ belonging to
  $(0,2)$, which adds difficulties in the analysis.
  We here mainly sketch the proof on the a priori estimates and the blowup
  criteria \eqref{eq:bc}.

We begin with the equivalent system \eqref{EAS-ref}, and we intend to obtain a priori estimate of the following quantity
\begin{equation}\label{Yt}
  Y(t):=  \|\rho(t)\|_{H^s(\T)}^2 + \|G(t)\|_{H^{s-\frac{\alpha}{2}}}^2
\end{equation}
on the small time interval $[0,T_0]$ with $T_0>0$ a constant depending
only on the influence function and the initial data.
By applying the differential operator $\Lambda^s$ to the equation of $\rho$ in \eqref{EAS-ref}, multiplying both sides with $\Lambda^s \rho$ and then integrating in $x$, it follows that
\begin{align}\label{rho-Hs-spl}
  \frac{1}{2} \frac{\dd}{\dd t}\|\rho(t)\|_{\dot H^s}^2 = & - \int_\T \Lambda^s \rho\cdot \Lambda^s \partial_x(\rho\,u) \dd x \nonumber \\
  = & - \int_\T \Lambda^s \rho\cdot (\Lambda^s \partial_x u \, \rho ) \,\dd x - \int_\T \Lambda^s \rho \cdot (u\,\partial_x\Lambda^s \rho) \dd x - \int_\T \Lambda^s \rho \cdot [\Lambda^s\partial_x,u,\rho] \dd x \nonumber \\
  =: & I + II +III,
\end{align}
where in term $III$ we used the notation $[L,f,g]= L(f\,g)- f (L g) - g (L f)$ for some operator $L$.

For term $I$, by using relation $\partial_x u= G + \LL \rho$, we have the following splitting
\begin{align}\label{I-decom}
  I = & - \int_\T (\Lambda^s \LL \rho)\cdot (\rho\,\Lambda^s\rho) \dd x - \int_\T (\Lambda^{s-\frac{\alpha}{2}} G)\cdot \Lambda^{\frac{\alpha}{2}}(\rho\,\Lambda^s\rho) \dd x \nonumber \\
  = & - \int_\T (\Lambda^s \sqrt{C'\mathrm{Id} +\LL\,} \rho)\cdot \sqrt{C' \mathrm{Id} +\LL\,}(\rho\,\Lambda^s\rho) \dd x
  + C' \int_\T |\Lambda^s \rho|^2 \,\rho \,\dd x - \int_\T (\Lambda^{s-\frac{\alpha}{2}} G)\cdot \Lambda^{\frac{\alpha}{2}}(\rho\,\Lambda^s\rho) \dd x \nonumber \\
  = & - \int_\T |\sqrt{C'\mathrm{Id} +\LL\,}\Lambda^s \rho|^2 \rho\, \dd x
   - \int_\T \sqrt{C'\mathrm{Id}+\LL\,}\Lambda^s\rho\cdot \big([\sqrt{C'\mathrm{Id} +\LL\,},\rho]\Lambda^s \rho\big)\, \dd x + C' \int_\T |\Lambda^s \rho|^2 \,\rho \,\dd x \nonumber \\
  & -\int_\T (\Lambda^{s-\frac{\alpha}{2}}G )\cdot \big((\Lambda^{s+\frac{\alpha}{2}}\rho)\, \rho\big)\,\dd x
  - \int_\T (\Lambda^{s-\frac{\alpha}{2}}G )\cdot \big( [\Lambda^{\frac{\alpha}{2}},\rho]\Lambda^s\rho \big)\,\dd x \nonumber \\
  := &\, I_1 + I_2 + I_3 + I_4 + I_5 ,
\end{align}
where the operator $\sqrt{C' \mathrm{Id} + \LL}$ is defined via formula \eqref{def:sqrtL}.
For $I_1$, by denoting $\rho_{\min,t}$ as $\min\limits_{\T\times [0,t]}\rho(x,s)$
(which satisfies $\rho_{\min,t}\geq M_0 e^{-c_3\bar{\rho}_0 t}$ from Lemma \ref{lem:lowbdd}), and using Plancherel's theorem and estimate \eqref{A-est}, we find
\begin{align}\label{I1-es1}
  I_1 \leq - \rho_{\min,t} \int_\T |\sqrt{C' + \LL \rho\,}\Lambda^s \rho|^2 \dd x
  \leq - C'^{-1} \rho_{\min,t} \int_\T |\Lambda^{s + \frac{\alpha}{2}} \rho|^2 \dd x.
\end{align}
By virtue of symbol upper-bound \eqref{A-est2}, commutator estimate \eqref{eq:comm-es0} (with $\epsilon = \frac{2-\alpha}{4}>0$) and Young's inequality,
the second term can be estimated as follows:
\begin{align}\label{I2-es1}
  |I_2| & \leq \|\sqrt{C'\mathrm{Id} + \LL\,}\Lambda^s \rho\|_{L^2} \|[\sqrt{C'\mathrm{Id} + \LL\,},\rho]\Lambda^s \rho\|_{L^2} \nonumber \\
  & \leq C \big(\|\rho\|_{\dot H^{s+\frac{\alpha}{2}}} + \|\rho\|_{\dot H^s} \big) \|\rho \|_{\dot H^s} \|\rho\|_{C^{\frac{2+\alpha}{4}}} \nonumber \\
  & \leq \frac{\rho_{\min,t}}{8C'} \|\rho\|_{\dot H^{s+\frac{\alpha}{2}}} +
  C (1+ \rho_{\min,t}^{-1}) \big(1+ \|\rho\|_{C^{\frac{2+\alpha}{4}}}^2 \big) \|\rho\|_{\dot H^s}^2.
\end{align}
The estimation of $I_3$ is taking advantage of Lemma \ref{lem:uppbdd}:
\begin{align}\label{I3-es}
  |I_3| \leq C' \|\rho(t)\|_{L^\infty} \|\rho\|_{\dot H^s}^2 \leq C' M_1 \|\rho\|_{\dot H^s}^2.
\end{align}
By using H\"older's inequality and commutator estimate \eqref{eq:comm-es2},
we similarly get that
\begin{align}\label{I4-I5-es}
  |I_4| + |I_5| & \leq \|G\|_{\dot H^{s-\frac{\alpha}{2}}} \|\rho\|_{\dot H^{s+\frac{\alpha}{2}}} \|\rho\|_{L^\infty}
  + C \|G\|_{\dot H^{s-\frac{\alpha}{2}}} \|\rho\|_{\dot H^s} \|\rho\|_{C^{\frac{2+\alpha}{4}}} \nonumber \\
  & \leq \frac{\rho_{\min,t}}{8C'} \|\rho\|_{\dot H^{s+\frac{\alpha}{2}}}^2 +
  C (1+ \rho_{\min,t}^{-1}) \big(1+ \|\rho\|_{C^{\frac{2+\alpha}{4}}}^2 \big) \big( \|\rho\|_{\dot H^s}^2 + \|G\|_{\dot H^{s-\frac{\alpha}{2}}}^2\big) .
\end{align}

Next, the term $II$ can be estimated from the integration by parts:
\begin{equation*}
  |II| = \frac{1}{2} \left|\int_\T (\Lambda^s\rho)^2\cdot \partial_x u\,\dd x \right| \leq \frac{1}{2} \|\partial_x u\|_{L^\infty} \|\rho\|_{\dot H^s}^2.
\end{equation*}
In view of estimates \eqref{Flinf-es}, \eqref{eq:uppbdd} and relation $\partial_x u= \LL \rho + G$, we see that
\begin{equation}\label{G-Linf-es1}
  \|G(t)\|_{L^\infty} \leq \|F(t)\|_{L^\infty} \|\rho(t)\|_{L^\infty} \leq \|F_0\|_{L^\infty} M_1 \leq C ,
\end{equation}
and
\begin{equation}\label{u-Lip-es}
\begin{split}
  \|\partial_x u(t)\|_{L^\infty} \leq \|\LL \rho(t)\|_{L^\infty}  + \|G(t)\|_{L^\infty}
  \leq  C (1 + \|\rho(t)\|_{B^\alpha_{\infty,1}}), 
\end{split}
\end{equation}
thus we also get
\begin{equation}\label{II-es}
  |II| \leq C (1 + \|\rho(t)\|_{B^\alpha_{\infty,1}}) \|\rho\|_{\dot H^s}^2.
\end{equation}
Taking advantage of the commutator estimate \eqref{eq:comm-es} below, the term $III$ can be estimated as
\begin{equation}\label{III-es1}
  |III| \leq \|\rho\|_{\dot H^s} \|[\Lambda^s\partial_x,u,\rho]\|_{L^2} \leq C \|\rho\|_{\dot H^s} \big(\|\partial_x u\|_{L^\infty} \|\rho\|_{\dot H^s} + \|\partial_x \rho\|_{L^\infty} \|u\|_{\dot H^s} \big).
\end{equation}
We need to bound the term $\|u\|_{\dot H^s(\T)}$: from \eqref{A-est2},
\begin{align}\label{uHs-es}
  \|u(t)\|_{\dot H^s(\T)} \leq \|\partial_x u(t)\|_{\dot H^{s-1}(\T)}
  & \leq  C\big(\|\rho(t)\|_{\dot H^{s+\alpha-1}} + \|\rho(t)\|_{\dot H^{s-1}} + \|G(t)\|_{\dot H^{s-1}}\big) \nonumber\\
  & \leq  C \big(\|\rho(t)\|_{H^s} + \|\rho(t)\|_{\dot H^{s+ \frac{\alpha}{2}}} \big) + C \|G(t)\|_{H^{s-\frac{\alpha}{2}}},
\end{align}
where $C>0$ depends on $\alpha,a_0,c_1,c_2$. Thus inserting estimates \eqref{u-Lip-es}, \eqref{uHs-es} into \eqref{III-es1} and using Young's inequality lead to
\begin{align}\label{III-es}
  |III| & \leq C \|\partial_x u\|_{L^\infty} \|\rho\|_{H^s}^2
  + C \|\rho\|_{\dot H^s} \|\partial_x \rho\|_{L^\infty}\Big( \|\rho\|_{H^s} + \|\rho\|_{\dot H^{s+ \frac{\alpha}{2}}} +  \|G\|_{H^{s-\frac{\alpha}{2}}}\Big) \nonumber \\
  & \leq \frac{\rho_{\min,t}}{8C'} \|\rho\|_{\dot H^{s+\frac{\alpha}{2}}}^2 + C (1+\rho_{\min,t}^{-1})  \big(1 + \|\partial_x \rho\|_{L^\infty}^2
  + \|\rho\|_{B^\alpha_{\infty,1}} \big) \big( \|\rho\|_{H^s}^2  + \|G\|_{H^{s-\frac{\alpha}{2}}}^2\big).
\end{align}
By taking the scalar product of $\rho$-equation with $\rho$ itself, we infer that
\begin{align}\label{rho-L2-es}
  \frac{1}{2}\frac{\dd}{\dd t}\|\rho(t)\|_{L^2}^2 = \int_\T \partial_x (u\,\rho)\cdot\rho \dd x = \frac{1}{2} \int_\T |\rho|^2\,\partial_x u\, \dd x
  \leq \frac{1}{2}\|\partial_x u\|_{L^\infty} \|\rho\|_{L^2}^2.
\end{align}
Since
\begin{equation}\label{Lam-sig-rho-Linf}
  \|\rho\|_{C^\sigma} \leq C_0 \|\rho\|_{B^\sigma_{\infty,1}}
  \leq C_\sigma (\|\rho\|_{L^\infty} + \|\partial_x \rho\|_{L^\infty}) \leq C(1 + \|\partial_x \rho\|_{L^\infty}),\quad \forall \sigma\in (0,1),
\end{equation}
we gather the above estimates on $I$, $II$, $III$ and \eqref{rho-L2-es} to deduce that
\begin{align}\label{rho-Hs-es}
  & \frac{1}{2}\frac{\dd}{\dd t}\|\rho(t)\|_{H^s}^2 + \frac{\rho_{\min,t}}{2C'} \|\rho\|_{\dot H^{s+\frac{\alpha}{2}}}^2 \nonumber \\
  \leq & \, C (1+\rho_{\min,t}^{-1})  \big(1 + \|\partial_x \rho\|_{L^\infty}^2
  + \|\rho\|_{B^\alpha_{\infty,1}} \big)\big( \|\rho\|_{H^s}^2 +\|G\|_{H^{s-\frac{\alpha}{2}}}^2\big),
\end{align}
with $C>0$ depending on $\alpha,a_0,c_1,c_2$ and initial data $(\rho_0,u_0)$.

Now we consider the estimation of $G$, and from the equation of $G$ in system \eqref{EAS-ref}, we get
\begin{align}
  & \frac{1}{2}\frac{\dd}{\dd t} \|G(t)\|_{\dot H^{s-\frac{\alpha}{2}}}^2 = - \int_\T (\Lambda^{s-\frac{\alpha}{2}} G) \cdot\big(\Lambda^{s-\frac{\alpha}{2}}\partial_x (u\,G)\big) \dd x \nonumber \\
  &= - \int_\T (\Lambda^{s-\frac{\alpha}{2}} G)\cdot \big(u\, (\Lambda^{s-\frac{\alpha}{2}}\partial_x G)\big) \dd x - \int_\T (\Lambda^{s-\frac{\alpha}{2}}G)\cdot \big([\Lambda^{s-\frac{\alpha}{2}}\partial_x, u]G\big)\dd x
  :=  IV + V,
\end{align}
where in the second line we have used the notation $[L,f]g=L(f\,g)- f L(g)$ for some operator $L$.
The term $IV$ can be treated as $II$ through integration by parts:
\begin{equation}
  |IV| = \frac{1}{2}\left| \int_\T |\Lambda^{s-\frac{\alpha}{2}}G|^2 \cdot\partial_x u \dd x\right| \leq \frac{1}{2} \|\partial_x u\|_{L^\infty} \|G\|_{\dot H^{s-\frac{\alpha}{2}}}^2.
\end{equation}
By applying estimates \eqref{G-Linf-es1}, \eqref{u-Lip-es}, \eqref{u-L2-es} and commutator estimate \eqref{eq:comm-es3},
we deduce that
\begin{align}
  |V| & \leq C \|G\|_{\dot H^{s-\frac{\alpha}{2}}} \big(\|\partial_x u\|_{L^\infty} \|G\|_{\dot H^{s-\frac{\alpha}{2}}} + \|G\|_{L^\infty} \|\partial_x u\|_{\dot H^{s-\frac{\alpha}{2}}} \big) \nonumber \\
  & \leq C \|\partial_x u\|_{L^\infty} \|G\|_{H^{s-\frac{\alpha}{2}}}^2 + C \|G\|_{\dot H^{s-\frac{\alpha}{2}}} \|G\|_{L^\infty} \big(\|\rho\|_{\dot H^{s+\frac{\alpha}{2}}} + \|\rho\|_{\dot H^{s-\frac{\alpha}{2}}} + \|G\|_{\dot H^{s-\frac{\alpha}{2}}} \big) \nonumber \\
  & \leq \frac{\rho_{\min,t}}{8C'} \|\rho\|_{\dot H^{s+\frac{\alpha}{2}}}^2 + C(1 + \rho_{\min,t}^{-1})
  \big(1 + \|\rho\|_{B^\alpha_{\infty,1}}\big)\big( \|G\|_{H^{s-\frac{\alpha}{2}}}^2 + \|\rho\|_{H^s}^2\big).
\end{align}
In a similar way as showing \eqref{rho-L2-es}, we also get
\begin{equation}\label{G-L2-es}
  \frac{1}{2}\frac{\dd}{\dd t}\|G(t)\|_{L^2}^2 = \int_\T \partial_x (u\,G)\cdot G \dd x = \frac{1}{2} \int_\T |G|^2\,\partial_x u\, \dd x
  \leq \frac{1}{2}\|\partial_x u\|_{L^\infty} \|G\|_{L^2}^2.
\end{equation}

Then collecting \eqref{rho-Hs-es} and the above estimates on $G$ yields
\begin{equation}\label{Yt-es}
\begin{split}
  \frac{1}{2}\frac{\dd}{\dd t}Y(t) + \frac{\rho_{\min,t}}{4C'} \|\rho\|_{\dot H^{s+\frac{\alpha}{2}}}^2
  \leq \, C  (1+\rho_{\min,t}^{-1})  \big(1 + \|\partial_x \rho\|_{L^\infty}^2
  + \|\rho\|_{B^\alpha_{\infty,1}} \big) Y(t),
\end{split}
\end{equation}
where $Y(t)$ is defined in \eqref{Yt}.
The Sobolev embedding $H^s(\T)\hookrightarrow B^1_{\infty,1}(\T)\hookrightarrow W^{1,\infty}(\T)$ for every $s>\frac{3}{2}$ as well as $H^s(\T)\hookrightarrow W^{2,\infty}(\T)$ for every $s>\frac{5}{2}$,
and the following estimate
\begin{equation}\label{Lam-alp-rho-Linf}
  \|\rho(t)\|_{B^\alpha_{\infty,1}} \leq C_\alpha (\|\rho(t)\|_{L^\infty} + \|\partial_x^2 \rho(t)\|_{L^\infty}) \leq C (1 + \|\partial_x^2 \rho(t)\|_{L^\infty}),\quad \forall \alpha\in (1,2),
\end{equation}
yield
\begin{equation}
  \frac{\dd}{\dd t} Y(t) + \frac{\rho_{\min,t}}{2C'} \|\rho\|_{\dot H^{s+\frac{\alpha}{2}}}^2  \leq C  (1+\rho_{\min,t}^{-1}) Y(t)^2,
\end{equation}
which implies that there exists a time $T_0>0$ depending only on $\alpha$, coefficients in $\LL$, $\min \rho_0$ and
$Y(0)=\|\rho_0\|_{H^s}^2 + \|G_0\|_{H^{s-\frac{\alpha}{2}}}^2$ so that $Y(t)$ is uniformly bounded on time interval $[0,T_0]$.
By a standard process, one can show that
\begin{equation}
  \rho\in C([0,T_0]; H^s(\T))\cap L^2([0,T_0]; \dot H^{s+\frac{\alpha}{2}}),\quad G\in C([0,T_0]; H^{s-\frac{\alpha}{2}}(\T)),
\end{equation}
and in combination with the following $L^2$-estimate of $u$ (from formulas \eqref{the-varph}, \eqref{u-exp} and estimates \eqref{I0t-bdd}, \eqref{psi-Linf-es})
\begin{align}\label{u-L2-es}
  \|u(t)\|_{L^2(\T)} & \leq \|\psi(t)\|_{L^2(\T)} + \|\LL \varphi(t)\|_{L^2(\T)} + |I_0(t)| \nonumber \\
  & \leq C_0 \|\psi(t)\|_{L^\infty(\T)} + \|\LL\Lambda^{-2}\partial_x \theta(t)\|_{L^2(\T)}  + C \nonumber\\
  & \leq C + C \|\theta(t)\|_{H^{\max\{0, \alpha-1\}}(\T)}  \leq C + C  \|\rho(t)\|_{ H^{\max\{0, \alpha-1\}}(\T)},
\end{align}
it also ensures $u\in C([0,T_0]; H^{s+1-\alpha}(\T))$ with
\begin{equation}
\begin{split}
  \|u\|_{L^\infty([0,T_0]; H^{s+1-\alpha})}^2 & \leq C_0 \|u\|_{L^\infty([0,T_0]; L^2)}^2 + C_0\|\partial_x u\|_{L^\infty([0,T_0]; \dot H^{s-\alpha})}^2 \\
  & \leq C (1+ \|\rho\|_{L^\infty([0,T_0]; H^s)}^2 + \|G\|_{L^\infty([0,T_0]; H^{s-\alpha})}^2 ) <\infty.
\end{split}
\end{equation}

Next we prove the blowup criteria \eqref{eq:bc}. Let $T^*>0$ be the maximal existence time of the smooth solution $(\rho,u)$ constructed as above,
the local well-posedness result firstly guarantees the natural blowup criteria: if $T^*<\infty$,
then necessarily
\begin{equation}
  \sup_{t\in [0,T^*[}\big(\|\rho(t)\|_{H^s(\T)} + \|G(t)\|_{H^{s-\frac{\alpha}{2}}(\T)}\big)=\infty .
\end{equation}
On the other hand, taking advantage of Gr\"onwall's inequality, estimate \eqref{Yt-es} together with inequalities \eqref{Lam-sig-rho-Linf}, \eqref{Lam-alp-rho-Linf}
implies that for every $0<T<T^*$,
\begin{equation*}
  \sup_{t\in [0,T]}Y(t) \leq
  \begin{cases}
    Y(0) \exp\Big\{ C (1+\rho_{\min,T}^{-1})  \int_0^T\big(1 + \|\partial_x \rho\|_{L^\infty}^2 \big) \dd t \Big\}, &\;\; \textrm{for}\;\; \alpha \in (0,1),\\
    Y(0) \exp\Big\{ C (1+\rho_{\min,T}^{-1})  \int_0^T\big(1 + \|\partial_x^2 \rho\|_{L^\infty}^2 \big) \dd t \Big\}, &\;\; \textrm{for}\;\; \alpha \in (1,2),
  \end{cases}
\end{equation*}
which ensures the blowup criteria \eqref{eq:bc} for the cases $\alpha\in (0,1)\cup (1,2)$.
For the case $\alpha=1$, we use the Beale-Kato-Majda's refinement: by arguing as \cite[Eq. (15)]{BKM84}, one can show that
\begin{equation}\label{Lam1-rho-Linf-es}
\begin{split}
  \| \rho(t)\|_{B^1_{\infty,1}} & \leq C_0 \|\rho(t)\|_{L^\infty} + C \|\partial_x \rho(t)\|_{L^\infty} \log(e + \|\rho(t)\|_{H^s}) + C \\
  & \leq C + C \|\partial_x \rho(t)\|_{L^\infty} \log (e +\|\rho(t)\|_{H^s}^2 ) ;
\end{split}
\end{equation}
so that inserting \eqref{Lam1-rho-Linf-es} into estimate \eqref{Yt-es} leads to
\begin{equation}
  \frac{\dd}{\dd t}Y(t) \leq  \, C (1+\rho_{\min,t}^{-1}) \big(1 + \|\partial_x \rho\|_{L^\infty}^2 \big) \log(e + Y(t)) Y(t),
\end{equation}
and also
\begin{equation}
  \sup_{t\in [0,T]}Y(t) \leq (e +Y(0))^{\exp\left\{C (1+\rho_{\min,T}^{-1})  \int_0^T\big(1 + \|\partial_x \rho\|_{L^\infty}^2 \big)\dd t \right\}},
\end{equation}
which proves the blowup criteria \eqref{eq:bc} at $\alpha=1$ case.
\end{proof}

\section{Global well-posedness}\label{sec:gwp}
In this section, we show our main global well-posedness result,
Theorem~\ref{thm:GR-EAS}.

According to the blowup criterion \eqref{eq:bc} in Theorem
\ref{thm:lwp}, we intend to show the boundedness of
$\|\partial_x\rho\|_{L^\infty(\T\times [0,T])}$ and
$\|\partial_x^2 \rho\|_{L^\infty(\T\times [0,T])}$, for cases
$\alpha\in(0,1]$ and $\alpha\in(1,2)$ respectively, for any given
finite time $T>0$.

Let us fix a time $T$ for the rest of the section. To control
$\partial_x\rho$ (and $\partial_x^2\rho$), we adopt the novel method
on modulus of continuity, which is originated in \cite{KNV,KNS} (see also \cite{Kis} for further development).
The general strategy is to prove that the evolution of the considered equation preserves a stationary (or time-dependent) modulus of continuity,
which furthermore implies the desired Lipschitz regularity.

\subsection{The modules of continuity}
A function $\omega:(0,\infty) \rightarrow (0,\infty)$ is called a \emph{modulus of continuity} (abbr. MOC) if $\omega$ is continuous on $(0,\infty)$, nondecreasing, concave, and
piecewise $C^2$ with one-sided derivatives defined at every point in $(0,\infty)$.
We say a function $f$ obeys the modulus of continuity $\omega$ if
\[|f(x)-f(y)| < \omega(|x-y|) \quad\forall~x\neq y\in \T.\]

We start with the following function
\begin{equation}\label{MOC0}
  \omega_1(\xi) =
  \begin{cases}
    \delta \left(\xi - \frac{1}{4}\xi^{1+\frac{\alpha}{2}}\right), & \quad \textrm{for  }\;0<\xi \leq 1, \\
    \frac{3\delta}{4} + \gamma \log\xi, & \quad \textrm{for  }\; \xi>1,
  \end{cases}
\end{equation}
where $\delta$ and $\gamma$ are small parameters to be chosen later.
It is easy to check that $\omega_1$ is a MOC. In particular, concavity
can be guaranteed if we pick $\gamma<\frac{\delta}{2}$.

In order to make sure the initial data $\rho_0$ obey a MOC $\omega$,
we shall construct $\omega$ by the scaling
$\omega(\xi)=\omega_1(\xi/\lambda)$,
where $\lambda$ is a small parameter called the scaling factor.

\begin{lemma}\label{lem:scaling}
  Let $\omega_1$ be defined in \eqref{MOC0} with $\delta$ and $\gamma$
  given. Then, for any function $f\in W^{1,\infty}(\R)$, there exists a small
  scaling factor $\lambda>0$ such that $f$ obeys the MOC
  $\omega(\xi)=\omega_1(\xi/\lambda)$.
\end{lemma}
\begin{proof}
Owing to $|f(x)-f(y)|\leq 2\|f\|_{L^\infty}$ and
$|f(x)-f(y)|\leq \|f'\|_{L^\infty} |x-y|$, it only needs to show that
$\min\{2\|f\|_{L^\infty},\|f'\|_{L^\infty} |x-y|\}< \omega(|x-y|)$.
Then from the concavity property on $\omega$, and by denoting $a_1: =
\frac{2\|f\|_{L^\infty}}{\|f'\|_{L^\infty}}$, it
suffices to show that
\[  2 \|f\|_{L^\infty} <\omega(a_1) =\omega_1(a_1/\lambda).\]
Pick a small $\lambda<a_1$.
We see that $\omega_1(a_1/\lambda)> \gamma \log( a_1/\lambda)$.
Thus by further choosing $\lambda$ small enough, that is,
\begin{equation}\label{dkg-cd0}
  \lambda\leq  a_1 e^{-2\gamma^{-1}\|f\|_{L^\infty}}
  = \frac{2\|f\|_{L^\infty}}{\|f'\|_{L^\infty}} e^{-2\gamma^{-1}\|f\|_{L^\infty}},
\end{equation}
we conclude that such an MOC $\omega(\xi)$ is obeyed by the function $f$.
\end{proof}

We summarize our choice of the MOC
\begin{equation}\label{MOC1}
  \omega(\xi) =
  \begin{cases}
    \delta \lambda^{-1} \xi - \frac{1}{4}\delta \lambda^{-1-\frac{\alpha}{2}} \xi^{1+\frac{\alpha}{2}}, & \quad \textrm{for  }\;0<\xi \leq \lambda, \\
    \frac{3\delta}{4} + \gamma \log \frac{\xi}{\lambda}, & \quad \textrm{for  }\; \xi>\lambda,
  \end{cases}
\end{equation}
with three small parameters $\delta, \gamma$ and $\lambda$ to be
determined later. Both $\delta$ and $\gamma$ are independent of the
scaling parameter $\lambda$.

We would like to emphasize the behavior of $\omega$ near the origin:
\begin{equation}\label{ome-cond}
  \omega(0+)=0,\quad
  \omega'(0+)= \delta \lambda^{-1}<\infty,
  \quad \omega''(0+)=-\infty.
\end{equation}
Since $\|f'\|_{L^\infty}\leq\omega'(0+)$ for any $f$ that obeys
$\omega$, \eqref{ome-cond} implies Lipschitz continuity $f$, with
\[\|f'\|_{L^\infty}\leq\delta\lambda^{-1}<\infty.\]
Moreover, the last part of \eqref{ome-cond} will be used in
Lemmas~\ref{lem:bd-scena} and \ref{lem:bd-scena2}.

\subsection{Uniform Lipschitz regularity of $\rho(t)$ on
  $[0,T]$}\label{subsec:rho-lip}
It suffices to show $\rho(t)$ obeys the MOC $\omega$ in \eqref{MOC1} for all
$t\in[0,T]$, as
\begin{equation}\label{eq:rho-lip}
  \sup_{t\in [0,T]}\|\partial_x\rho(\cdot,t)\|_{L^\infty(\T)} \leq \omega'(0+) = \delta \lambda^{-1}.
\end{equation}
From \eqref{dkg-cd0}, we can ensure that $\rho_0$ obeys $\omega$ by
picking a sufficiently small $\lambda$
\begin{equation}\label{dkg-cd2}
  \lambda\leq\frac{2\|\rho_0\|_{L^\infty}}{\|\rho_0'\|_{L^\infty}}e^{-2\gamma^{-1}\|\rho_0\|_{L^\infty}}.
\end{equation}
We need to prove the preservation of the MOC $\omega$ in time.
Let us argue by contradiction.
Assume that $t_1\in (0,T]$ is the first time that the modulus of continuity $\omega(\xi)$ is violated by $\rho(x,t)$.
We state the following lemma describing the only possible breakthrough scenario.
The proof is identical to that of \cite{KNV}, provided that $\omega$
satisfies \eqref{ome-cond}.
\begin{lemma}\label{lem:bd-scena}
  Assume that $\rho(x,t)$ is a smooth function on $\T\times [0,T]$ and $\rho_0(x)$ obeys the MOC $\omega(\xi)$ given by \eqref{MOC1}.
Suppose that $t_1\in (0,T]$ is the first time that such an $\omega(\xi)$ is lost by $\rho$, then
\begin{equation}\label{eq:scena0}
    |\rho(\tilde{x},t_1)-\rho(\tilde{y},t_1)|\leq \omega(|\tilde{x}-\tilde{y}|),\quad\forall~ \tilde{x},\tilde{y}\in \T,
\end{equation}
and there exist two distinct points $x\neq y\in \T$ satisfying
\begin{equation}\label{eq:scena}
\begin{split}
  \rho(x,t_1) - \rho(y,t_1) = \omega(\xi), \quad \textrm{with   }\,\xi =|x-y|.
\end{split}
\end{equation}
\end{lemma}

Moreover, since the range of $\rho$ lies in $[0,M_1]$ by Lemma~\ref{lem:uppbdd},
the equality \eqref{eq:scena} and the positivity of $\rho$ imply $\omega(\xi)\leq M_1$.
Therefore, breakthrough could only happen in the region
\begin{equation}\label{xi-scope}
 0<\xi\leq\Xi:=\omega^{-1}(M_1) = \lambda e^{\gamma^{-1}(M_1-\frac{3}{4}\delta)}.
\end{equation}
We can pick a small enough $\lambda$
\begin{equation}\label{del-cond1}
  \lambda \leq \frac{r_0}{4} e^{- M_1 \gamma^{-1}}.
\end{equation}
to guarantee that the breakthrough can only happen in the short range,
with
\[\Xi\leq \frac{r_0}{4}.\]
where $r_0$ is defined in \eqref{phi-s-assum1}.

Next, we intend to prove that for the points $x\neq y\in\T$ satisfying
equality \eqref{eq:scena} with $\xi=|x-y|$ in the range \eqref{xi-scope}, it holds
\begin{equation}\label{eq:targ}
  \partial_t (\rho(x,t)-\rho(y,t))|_{t=t_1} <0.
\end{equation}
If so, the breakthrough scenario won't happen, and it will yield a
contradiction and in turn guarantee the preservation of the MOC.
For simplicity, we drop the $t_1$-dependence in the sequel.

To verify \eqref{eq:targ}, we use the equation of $\rho$ given by
\eqref{EA-rho} and get
\begin{align}\label{rho-moc-decom}
  \partial_t\rho(x)- \partial_t\rho(y) & = -\partial_x(u\, \rho)(x) + \partial_x(u\, \rho)(y) \nonumber \\
  & = -\big( u\, \partial_x\rho (x) - u\,\partial_x\rho(y)\big) - \big(\rho(x) -\rho(y)\big)\partial_x u(x) 
  - \rho(y)\big(\partial_x u(x) -\partial_x u(y)\big) \nonumber \\
  & =:\, \mathrm{I} + \mathrm{II} + \mathrm{III}.
\end{align}

We first consider $\mathrm{III}$, due to that it is the term containing negative contribution which is crucial in achieving \eqref{eq:targ}.
From the relation $\partial_x u = \LL \rho + G =  \LL \rho +  F \rho$ (recalling $F=\frac{G}{\rho}$ satisfies equation \eqref{Feq}), we further get
\begin{equation}\label{IIIdecom}
\begin{split}
  \mathrm{III} &=\, - \rho(y) \big(\LL \rho(x) -\LL \rho(y)\big)
  - \rho(y) \big(\rho(x) -\rho(y) \big) F(x)
  - \rho(y)^2 \big(F(x) -F(y)\big)  \\
  & =:\,\mathrm{III}_1 + \mathrm{III}_2 + \mathrm{III}_3 .
\end{split}
\end{equation}

In order to estimate $\mathrm{III}_1$, we state the following lemma,
and postponed the proof in Section~\ref{sec:MOCes}.
\begin{lemma}\label{lem:MOCes1}
  Assume $\rho$ obeys the MOC $\omega(\xi)$, and $x, y$ satisfy the breakthrough scenario
  described in Lemma~\ref{lem:bd-scena}.
  Define $D (x,y):= \LL  \rho (x) - \LL \rho (y)$. Then $D(x,y)$ can be estimated as
\begin{equation}\label{Dest0}
  D(x,y) \geq D_1(x,y) - 2c_2 \omega(\xi)
\end{equation}
where
\begin{equation}\label{Dexp}
  D_1(x,y) := \, \mathrm{p.v.}\int_{|z|\leq a_0} \phi(z)\,\left( \omega(\xi)- \rho(x+z)+ \rho (y+ z)\right)\dd z,
\end{equation}
and it satisfies that
for any $\xi=|x-y|\in (0,\frac{a_0}{2}]$ (with $a_0>0$ the constant appearing in (A1)),
\begin{equation}\label{Dest}
\begin{split}
  D_1 (x,y)\geq &
  \, \frac{1}{c_1}\int_{0}^{\frac{\xi}{2}} \frac{2\omega (\xi)-\omega(\xi+2\eta)-\omega (\xi-2\eta)}{\eta^{1+\alpha}}\mathrm{d} \eta \\
  & +\frac{1}{c_1}\int_{\frac{\xi}{2}}^{a_0} \frac{2\omega (\xi)-\omega (2\eta+\xi)+\omega (2\eta-\xi) }{\eta^{1+\alpha}}\mathrm{d}\eta .
\end{split}
\end{equation}
Moreover, if we use the MOC defined in \eqref{MOC1},
we have that for any $\xi=|x-y|\in (0,\frac{a_0}{2}]$,
\begin{equation}\label{Dest2}
  D_1(x,y) \geq
  \begin{cases}
    \frac{\alpha}{32 c_1}\delta \lambda^{-1-\frac{\alpha}{2}} \xi^{1- \frac{\alpha}{2}}, &\quad \mathrm{for}\;\; 0<\xi \leq \lambda, \\
    \frac{2^\alpha-1}{2\alpha c_1} \omega(\xi) \xi^{-\alpha}, &\quad \mathrm{for}\;\; \lambda< \xi \leq \frac{a_0}{2}.
  \end{cases}
\end{equation}
\end{lemma}

\begin{remark}
  The $D_1$ term represents the dissipation phenomenon due to strong
  alignment in the short range. The extra term appears in the right
  hand side of \eqref{Dest0} takes care of the misalignment effect. We
  will verify that it can be controlled by the dissipation.
\end{remark}

Denote by $\rho_{\min,T}$ the minimum of $\rho(x,t)$ on domain
$\T\times[0,T]$. Owing to Lemma \ref{lem:lowbdd}, we have
\begin{equation}\label{eq:rho-min}
  \min_{(x,t)\in\T\times [0,T]} \rho(x,t) = \rho_{\min,T} \geq M_0e^{-c_3\bar{\rho}_0 T}.
\end{equation}
Then,
\begin{equation}\label{III1es}
  \mathrm{III}_1 \leq -\rho_{\min,T} D_1(x,y) +  2c_2 M_1 \omega(\xi) ,
\end{equation}
with $D_1(x,y)$ satisfying the estimate \eqref{Dest2} and $M_1$ the upper bound of $\|\rho(t_1)\|_{L^\infty}$ appearing in Lemma \ref{lem:uppbdd}.

For $\mathrm{III}_2$, recalling that $F=\frac{G}{\rho}$ has the $L^\infty$-estimate \eqref{Flinf-es}, we immediately get
\begin{equation}\label{III3es}
  \mathrm{III}_2 \leq M_1 \|F_0\|_{L^\infty} \omega(\xi).
\end{equation}
Also, $\partial_x F$ and $H :=\frac{\partial_x F}{\rho}$ satisfy the following equations
\begin{equation}\label{eq:H}
  \partial_t (\partial_x F) + \partial_x ( u\, (\partial_x F))=0,
  \quad\textrm{and}\quad \partial_t H  +  u\,\partial_x H=0.
\end{equation}
We directly deduce that
\begin{equation}\label{Hlinf-es}
  \sup_{t\in [0,T]}\|H(t)\|_{L^\infty(\T)} \leq \|H_0\|_{L^\infty(\T)} = \Big\|\frac{\partial_x F_0}{\rho_0}\Big\|_{L^\infty(\T)}.
\end{equation}
Thus by virtue of \eqref{eq:uppbdd} and \eqref{Hlinf-es}, we have
\begin{equation}\label{f-diff}
  |F(x)-F(y)|\leq\|\partial_xF\|_{L^\infty}\xi\leq\|H\|_{L^\infty}\|\rho\|_{L^\infty}\xi\leq
  \|H_0\|_{L^\infty}M_1\xi.
\end{equation}
Hence, the term $\mathrm{III}_3$ can be estimated as
\begin{equation}\label{III4es}
  \mathrm{III}_3 \leq \|H_0\|_{L^\infty} M_1^3 \xi.
\end{equation}

Gathering estimates \eqref{III1es}, \eqref{III3es} and \eqref{III4es} leads to
\begin{equation}\label{rIIIes}
  \mathrm{III} \leq -\rho_{\min,T} D_1(x,y)  + M_1\big(2 c_2
  +  \|F_0\|_{L^\infty} \big) \omega(\xi) + \|H_0\|_{L^\infty} M_1^3 \xi.
\end{equation}

Now, we turn to the estimate on $\mathrm{II}$. We state the following
lemma on the one-sided bounds of $\LL\rho(x)$ and $\LL\rho(y)$.
The idea follows from \cite[Section~4.2.2]{DKRT}, with an additional treatment on the
misalignment. The proof is placed in Section \ref{sec:MOCes}.
We only need to use the lower bound on $\LL\rho(x)$ here, but will use both
bounds later.
\begin{lemma}\label{lem:MOCes2}
  Assume $\rho$ obeys the MOC $\omega(\xi)$, and $x,y$ satisfy the
  breakthrough scenario described in Lemma~\ref{lem:bd-scena}. Then,
  we have the following one-sided bounds for every $\xi\in(0,\frac{r_0}{2}]$
  \begin{equation}\label{MOCes-LL}
    -\LL\rho(x),\,\, \LL\rho(y)\leq
    4c_1\int_\xi^{r_0}\frac{\omega(\xi+\eta)-\omega(\xi)}{\eta^{1+\alpha}}
    \dd \eta + 2c_3M_1.
  \end{equation}
  Moreover, if we use the MOC defined in \eqref{MOC1}, we have that
  \begin{equation}\label{MOCes-Lam-alp}
  -\LL\rho(x),\,\, \LL\rho(y)\leq\,2c_3M_1+
  \begin{cases}
   4 c_1\overline{C}_\alpha \delta \lambda^{-\frac\alpha2}\xi^{-\frac\alpha2},
   & \quad \textrm{for   }\;0<\xi\leq \lambda, \\
   \frac{12c_1}{\alpha^2} \gamma\xi^{-\alpha}, & \quad \textrm{for   }\;\lambda <\xi \leq \frac{r_0}{2},
  \end{cases}
\end{equation}
where $\overline{C}_\alpha$ is a positive constant that only depends on
$\alpha$ (see \eqref{barC-alp} for the explicit expression).
\end{lemma}

Thus by virtue of the relation $\partial_x u= \LL \rho +  F \rho$, scenario \eqref{eq:scena}, and using Lemma \ref{lem:MOCes2} and estimates \eqref{Flinf-es}, \eqref{eq:uppbdd},
we obtain
\begin{equation}\label{rII-es}
\begin{split}
  \mathrm{II} = &\, \omega(\xi) \big(- \LL \rho(x) - F(x) \rho(x)\big) \\
  \leq &\, 
  \omega(\xi) M_1 (2c_3+\|F_0\|_{L^\infty}) +
  \begin{cases}
    4 c_1 \overline{C}_\alpha \delta^2 \lambda^{-1-\frac{\alpha}{2}} \xi^{1-\frac{\alpha}{2}} ,
    & \; \textrm{for   }\;0<\xi\leq \lambda, \\
    \frac{12c_1}{\alpha^2} \gamma \omega(\xi) \xi^{-\alpha}, & \; \textrm{for   }\;\lambda<\xi\leq \frac{r_0}{2}.
  \end{cases}
\end{split}
\end{equation}

Next, we consider the contribution from the drift term $\mathrm{I}$.
The following lemma shows an estimate on the MOC on velocity $u$.
The proof is postponed in Section \ref{sec:MOCes}.
\begin{lemma}\label{lem:MOCes-u}
Assume $\rho$ obeys the MOC $\omega(\xi)$. Then, $u$ obeys the following MOC
\begin{equation}\label{MOCu}
\Omega(\xi)=\frac{52c_1}{\alpha}\int_0^\xi\frac{\omega(\eta)}{\eta^\alpha}\dd\eta
  + 8c_1\xi\int_\xi^{r_0+\xi}\frac{\omega(\eta)}{\eta^{1+\alpha}}\dd\eta + M_1(8c_3+\|F_0\|_{L^\infty})\xi,
\end{equation}
for any $\xi\in(0,\frac{r_0}{4}]$. Namely, for any $\tilde{x},
\tilde{y}\in\T$, with $\xi=|\tilde{x}-\tilde{y}|\leq\frac{r_0}{4}$,
\begin{equation}\label{u-MOC-es}
  |u(\tilde{x})-u(\tilde{y})|  \leq \Omega(\xi).
\end{equation}
Moreover, let $x, y$ satisfy the breakthrough scenario described in
Lemma~\ref{lem:bd-scena}. If we use the MOC defined in \eqref{MOC1},
with
\begin{equation}\label{gamma-cd}
  \gamma < \frac{3\alpha}{4}\delta,
\end{equation}
then, we have an enhanced estimate
\begin{equation}\label{u-enhanced}
  |u(x)-u(y)|\leq  4c_1^2D_1(x,y)\xi +M_1(8c_3+\|F_0\|_{L^\infty})\xi + \begin{cases}
   18 c_1\overline{C}_\alpha\delta\lambda^{-\frac\alpha2}\xi^{1-\frac\alpha2}, & \;\textrm{for   }\; 0<\xi \leq \lambda, \\
    \frac{26 c_1}{\alpha}  \omega(\xi) \xi^{1-\alpha}, & \;\textrm{for   }\; \lambda <\xi \leq \frac{r_0}{4}.
  \end{cases}
\end{equation}
\end{lemma}
\begin{remark}\label{rmk:Omega}
  Estimate \eqref{MOCu} was first introduced in \cite[Lemma]{KNV} on
  critical quasi-geostrophic equation.
  It was extended to the Euler-alignment system with $\alpha\in(0,1)$ in
  \cite{DKRT}. Here, we further generalize the estimate to
  $\alpha\in(0,2)$, and consider misalignment as well. The
  misalignment effect contributes to the last term in \eqref{MOCu}.

  When $\alpha\geq1$, the first term in \eqref{MOCu} can not
  be controlled by the dissipation. A modified MOC was introduced in
  \cite{KNV} for the case $\alpha=1$. Here, we propose an enhanced
  estimate \eqref{u-enhanced} on $|u(x)-u(y)|$, using $D_1(x,y)$ to
  replace the problematic first term in \eqref{MOCu}. The novel idea
  allows us to extend the result to the full range of
  $\alpha\in(0,2)$, without changing the MOC $\omega(\xi)$.
\end{remark}

By virtue of the relation
\[u(x)\partial_x \rho(x)= \lim_{h\rightarrow 0+} \frac{\rho(x + h
    u(x)) -\rho(x)}{h}\]
and using scenario \eqref{eq:scena},
we can obtain (see e.g. \cite{KNV})
\begin{equation}\label{rI-es0}
  |\mathrm{I}| =  |u(x)\partial_x\rho(x)- u(y)\partial_x\rho(y)| \leq |u(x)-u(y)|\omega'(\xi),
\end{equation}
which combined with Lemma \ref{lem:MOCes-u} and formula \eqref{MOC1} yields
\begin{align}\label{rI-es}
  |\mathrm{I}| \leq &\, 4c_1^2D_1(x,y)\omega'(\xi)\xi +M_1(8c_3+\|F_0\|_{L^\infty})\omega'(\xi)\xi \nonumber \\
  & + \begin{cases}
   18c_1\overline{C}_\alpha\delta^2\lambda^{-1-\frac\alpha2}\xi^{1-\frac\alpha2}, & \;\textrm{for   }\; 0<\xi \leq \lambda, \\
    \frac{26 c_1}{\alpha}  \gamma\omega(\xi) \xi^{-\alpha}, & \;\textrm{for
    }\; \lambda <\xi \leq \frac{r_0}{4}.
    \end{cases}
\end{align}

Hence, gathering \eqref{rho-moc-decom} and estimates \eqref{rIIIes}, \eqref{rII-es}, \eqref{rI-es}, and in light of \eqref{xi-scope}, we find that for every $0<\xi \leq \Xi$,
\begin{align}
  \partial_t\rho(x)- \partial_t\rho(y)
  \leq & -\Big(\rho_{\min,T} - 4c_1^2\omega'(\xi)\xi\Big) D_1(x,y) \label{Targ-es1}\\
  &+ \begin{cases}
   22c_1\overline{C}_\alpha\delta^2\lambda^{-1-\frac\alpha2}\xi^{1-\frac\alpha2}, & \;\textrm{for   }\; 0<\xi \leq \lambda \\
    \frac{64 c_1}{\alpha^2}  \gamma\omega(\xi) \xi^{-\alpha}, & \;\textrm{for
    }\; \lambda <\xi \leq \frac{r_0}{4}
    \end{cases}\nonumber\\
  & +  2 M_1 \Big(c_2 + c_3 + \|F_0\|_{L^\infty} \Big) \omega(\xi) +M_1\Big( 8c_3+\|F_0\|_{L^\infty}\Big) \omega'(\xi) \xi +  \|H_0\|_{L^\infty} M_1^3 \xi.
 \nonumber
\end{align}

Our goal now is to show the right hand side of estimate \eqref{Targ-es1} is negative, by appropriately choosing the constants $\delta,\gamma$ and $\lambda$ in
MOC $\omega(\xi)$ defined by \eqref{MOC1}. We divide the proof into two cases.

Case 1: $0<\xi \leq \lambda$.
In this case $\omega(\xi)\leq \delta \lambda^{-1}\xi $, and $\omega'(\xi)\xi \leq
\delta\lambda^{-1}\xi$ as well.
We first set $4c_1^2\omega'(\xi)\xi\leq 4c_1^2\delta \leq \frac{1}{4}
\rho_{\min,T}$, that is,
\begin{equation}\label{dkg-cd3}
  \delta \leq \frac{1}{16c_1^2}\rho_{\min,T}.
\end{equation}
So the first term in \eqref{Targ-es1} is bounded by
\[-\frac{3\rho_{\min,T}}{4}D_1(x,y),\quad\text{where}\,\,D_1(x,y)\geq\frac{\alpha}{32 c_1}\delta
  \lambda^{-1-\frac{\alpha}{2}} \xi^{1- \frac{\alpha}{2}}.\]
The second term in \eqref{Targ-es1} has the same scaling as $D_1$. It
could be made smaller than $\frac{1}{4}\rho_{\min,T}D_1$ by choosing
$\delta$ small,
\begin{equation}\label{dkg-cd5}
  22 c_1\overline{C}_\alpha\delta \leq  \frac{ \alpha}{128c_1}\rho_{\min,T},
\end{equation}
The third term is subcritical in scaling, and hence can be controlled
by $\frac{1}{4}\rho_{\min,T}D_1$ by choosing the scaling factor
$\lambda$ small. Indeed,
\[\delta\lambda^{-1}\xi\leq\delta\lambda^{-1+\frac\alpha2}\xi^{1-\frac\alpha2}
=\lambda^\alpha\left(\delta\lambda^{-1-\frac\alpha2}\xi^{1-\frac\alpha2}\right).\]
Therefore, we choose $\lambda$ as follows
\begin{equation}\label{dkg-cd4}
  \lambda \leq \delta \quad \textrm{and}\quad M_1\Big( 2 c_2 + 10 c_3+3\|F_0\|_{L^\infty}+M_1^2\|H_0\|_{L^\infty}\Big) \lambda^\alpha < \frac{ \alpha}{128c_1}\rho_{\min,T}.
\end{equation}

With the choices of $\delta$ and $\lambda$, we conclude
\begin{equation}\label{nega}
  \partial_t\rho(x)- \partial_t\rho(y)\leq-\frac{3\rho_{\min,T}}{4}
  +\frac{\rho_{\min,T}}{4}+\frac{\rho_{\min,T}}{4}=-\frac{\rho_{\min,T}}{4}<0.
\end{equation}

Case 2: $\lambda< \xi \leq \Xi$.
In this case $\omega'(\xi)\xi =\gamma$. We bound the first term in
\eqref{Targ-es1} with
\[-\frac{3\rho_{\min,T}}{4}D_1(x,y),\quad\text{where}\,\,D_1(x,y)\geq
  \frac{2^\alpha-1}{2\alpha c_1} \omega(\xi) \xi^{-\alpha},\]
by simply setting $\gamma$ small enough so that
$4c_1^2\omega'(\xi)\xi = 4c_1^2\gamma \leq\frac14\rho_{\min,T}$.
Note that we have already assumed $\gamma<\frac\delta2$. So, the
inequality is satisfied from the assumption \eqref{dkg-cd3}.

The second term in \eqref{Targ-es1} is scaling critical, and can be
easily made smaller than $\frac14\rho_{\min,T}D_1$ by choosing
$\gamma$ small
\begin{equation}\label{dkg-cd8}
   \frac{64c_1}{\alpha^2}\gamma < \frac{2^\alpha-1}{8\alpha c_1}\rho_{\min,T}.
\end{equation}

The third term in \eqref{Targ-es1} is subcritical in scaling, and
can be controlled by choosing the scaling factor $\lambda$ small. To
see this, observe
$\omega'(\xi)\xi=\gamma<\frac34\delta=\omega(\lambda)\leq\omega(\xi)$,
and $\xi^{-\alpha}\geq\Xi^{-\alpha} \geq \lambda^{-\alpha} e^{-\alpha \gamma^{-1} M_1}$ (from \eqref{xi-scope}). Hence, we only need
\begin{equation}\label{dkg-cd7}
    \lambda \leq \gamma e^{-\gamma^{-1} M_1},\quad \textrm{and}\quad  M_1\Big(2 c_2 + 10c_3 + 3\|F_0\|_{L^\infty}+M_1^2\|H_0\|_{L^\infty}\Big) e^{\alpha \gamma^{-1}M_1} \lambda^\alpha \leq  \frac{2^\alpha-1}{8\alpha c_1}\rho_{\min,T}.
\end{equation}

We end up with \eqref{nega} as well, finishing the whole proof.

We summarize our choice of the stationary MOC $\omega(\xi)$. Define
$\omega$ by \eqref{MOC1}. Pick the parameters in the following order:
(i) $\delta\in(0,1)$ satisfying \eqref{dkg-cd3} and \eqref{dkg-cd5}; (ii)
$\gamma\in(0,\frac{\delta}{2})$ satisfying \eqref{gamma-cd}
and \eqref{dkg-cd8}; (iii) $\lambda$ satisfying
\eqref{dkg-cd2}, \eqref{del-cond1}, \eqref{dkg-cd4} and \eqref{dkg-cd7}.

\begin{remark}\label{rmk:growth}
  Observe from \eqref{eq:rho-min} that $\rho_{\min, T}$ can decay
  exponentially in $T$. Then, our choices of parameters $\delta$ and
  $\gamma$ also decay exponentially in $T$. Then, from
  \eqref{del-cond1} and \eqref{dkg-cd7}, the bound on $\lambda$ is double exponentially in $T$.
  Thus, in view of \eqref{eq:rho-lip},
  $\|\partial_x\rho(\cdot,T)\|_{L^\infty}$ can grow double
  exponentially in $T$. Note that without the misalignment effect, it
  is known that $\|\partial_x\rho(\cdot,T)\|_{L^\infty}$ is bounded
  uniformly in all time. Our result indicates that the misalignment
  could destabilize the solution as time becomes large.
\end{remark}

\subsection{Uniform Lipschitz regularity of $ \partial_x \rho(t)$ on $[0,T]$}\label{subsec:der-rho-lip}
When $1<\alpha<2$, the boundedness of
$\|\partial_x^2\rho\|_{L^\infty(\T\times[0,T])}$ is required to ensure
global regularity. It suffices to show $\partial_x\rho(t)$ obeys the MOC
$\omega$ in \eqref{MOC1} for all $t\in[0,T]$.
Note that the parameters used in the MOC for $\partial_x\rho(t)$ can be
different from the MOC for $\rho(t)$. For instance, to ensure that
$\rho_0'$ obeys $\omega$, we need to pick $\lambda$ such that
\begin{equation}\label{dkg-asum1}
  \lambda \leq \frac{2\|\rho_0'\|_{L^\infty}}{\|\rho_0''\|_{L^\infty}} e^{-2\gamma^{-1}\|\rho_0'\|_{L^\infty}}.
\end{equation}
We shall continue use the notation $\omega(\xi)$ to denote the
MOC. But in this part, $\omega(\xi)$ is obeyed by $\partial_x\rho(t)$
rather than $\rho(t)$.

Let us denote $\rho'(x,t)=\partial_x\rho(x,t)$.
The construction of the MOC for $\rho'(t)$ is partly similar to the
argument for $\rho(t)$, with additional subtleties that need to be
taken care of. The proof of the preservation of MOC in time will
directly imply the desired bound on $\partial_x^2\rho$
\begin{equation}\label{eq:rho'-lip}
 \sup_{t\in[0,T]}\|\partial_x^2\rho(\cdot,t)\|_{L^\infty(\T)}\leq\omega'(0+)=\delta\lambda^{-1}.
\end{equation}

First, we state the only possible breakthrough scenario for the MOC on
$\rho'(t)$. The statement is similar to Lemma~\ref{lem:bd-scena}.
\begin{lemma}\label{lem:bd-scena2}
  Assume that $\rho(x,t)$ is a smooth function on $\T\times [0,T]$ and
$\rho_0'(x)$ obeys the MOC $\omega(\xi)$ given by \eqref{MOC1}.
Suppose that $t_1\in (0,T]$ is the first time that such an
$\omega(\xi)$ is lost by $\rho'$, then we have
\begin{equation}\label{eq:scena2}
    |\rho'(\tilde{x},t_1)-\rho'(\tilde{y},t_1)|\leq \omega(|\tilde{x}-\tilde{y}|),\quad\forall~ \tilde{x},\tilde{y}\in \T,
\end{equation}
and there exist two points $x\neq y\in \T$ satisfying
\begin{equation}\label{eq:scena3}
\begin{split}
  \rho'(x,t_1) - \rho'(y,t_1) = \omega(\xi), \quad \textrm{with   }\,\xi =|x-y|.
\end{split}
\end{equation}
\end{lemma}

Denote by $M_{2,T}$ the bound of $\rho'(t)$ on $[0,T]$ appearing in estimate \eqref{eq:rho-lip},
so that we write it as
\begin{equation}\label{eq:rho-lip2}
  \sup_{t\in [0,T]} \|\rho'(t)\|_{L^\infty(\T)} \leq M_{2,T}.
\end{equation}
Since $\rho'$ lies in $[-M_{2,T},M_{2,T}]$, the equality
\eqref{eq:scena3} implies $\omega(\xi)\leq 2M_{2,T}$.
Therefore, breakthrough could only happen in the region
\begin{equation}\label{eq:Xi1}
  0 < \xi \leq \Xi_1: = \omega^{-1}(2M_{2,T}) =\lambda e^{\gamma^{-1}(2M_{2,T}-\frac34\delta)}.
\end{equation}
We can pick a small enough $\lambda$
\begin{equation}\label{del-cond2}
  \lambda \leq \frac{r_0}{4} e^{- 2 \gamma^{-1} M_{2,T}},
\end{equation}
to guarantee that the breakthrough only happens in the short range,
with $\Xi_1\leq\frac{r_0}{4}$.

Next, we intend to prove that for the points $x\neq y\in\T$ satisfying
\eqref{eq:scena3} with $\xi=|x-y|$ in the range \eqref{eq:Xi1}, it holds
\begin{equation}\label{eq:targ2}
  \partial_t (\rho'(x,t)-\rho'(y,t))|_{t=t_1} <0.
\end{equation}

From the system \eqref{EAS-ref}, we get the dynamics of $\rho'(x,t)$ as
\begin{equation}\label{eq:rho'}
  \partial_t \rho' + u\, \partial_x \rho' + 2\rho'\, \partial_x u +\rho\, \partial_x^2 u = 0,
\end{equation}
with
\begin{equation}\label{eq:u-rho-rela}
  \partial_x u= \LL \rho + G,\quad\textrm{and}\quad \partial_x^2 u= \LL \rho' + \partial_x G.
\end{equation}
Then, we have
\begin{align}\label{rho'-moc-dec}
  \partial_t\rho'(x)- \partial_t\rho'(y)
  = & -\big( u(x)\, \partial_x\rho' (x) - u(y)\,\partial_x\rho'(y)\big) - 2\big(\rho'(x) -\rho'(y)\big)\partial_x u(x) \nonumber \\
  & - 2 \rho'(y)\big(\partial_x u(x) -\partial_x u(y)\big) - \big(\rho(x)\, \partial_x^2 u(x) -\rho(y)\, \partial_x^2 u(y)\big) \nonumber \\
  =: &\, \mathcal{I} + \mathcal{II} + \mathcal{III} +\mathcal{IV} .
\end{align}
Again, we suppress the $t_1$-dependence from now on for simplicity.

We start with the estimation on the term $\mathcal{IV}$, through a
similar treatment as on the terms $\mathrm{II}+\mathrm{III}$ in the MOC estimates
for $\rho(t)$.
A main difference is that $\rho(x)-\rho(y)$ does not
necessarily has a sign, in opposition to the case on MOC of $\rho(t)$,
where the quantity is positive due to \eqref{eq:scena}.
Instead, we will perform different decompositions depending on the
sign of $\rho(x)-\rho(y)$ as follows.
\begin{align*}
  \mathcal{IV} =&
  \begin{cases}- \big(\rho(x) -\rho(y)\big)\, \partial_x^2
    u(x) -\rho(y)\,\big(\partial_x^2 u(x)- \partial_x^2
    u(y)\big)&\text{if}~\rho(x)-\rho(y)\geq0\\
    - \big(\rho(x) -\rho(y)\big)\, \partial_x^2 u(y)
    -\rho(x)\,\big(\partial_x^2 u(x)- \partial_x^2
    u(y)\big)&\text{if}~\rho(x)-\rho(y)<0
  \end{cases}\\
  :=&\,\mathcal{IV}_1+\mathcal{IV}_2.
\end{align*}

The term $\mathcal{IV}_1$ can be estimated similarly as
$\mathrm{II}$. We have
\[\mathcal{IV}_1\leq  |\rho(x)-\rho(y)|\cdot\big(\max\{-\LL\rho'(x),\LL\rho'(y)\}+
  \|\partial_xG(\cdot)\|_{L^\infty}\big).\]
In particular, using \eqref{eq:rho-lip2}, $|\rho(x)-\rho(y)|$ can be estimated by
\begin{equation}\label{rhoLip}
  |\rho(x)-\rho(y)|\leq M_{2,T}\xi.
\end{equation}
Apply Lemma~\ref{lem:MOCes2} on $\rho'$ (instead of $\rho$) and get
\begin{equation}\label{MOCes-Lam-alp2}
  \max\{-\LL\rho'(x), \LL\rho'(y)\}\leq~4c_3M_{2,T}+
  \begin{cases}
   \frac{12c_1}{\alpha-1} \delta \lambda^{-\frac\alpha2} \xi^{-\frac\alpha2},
   & \quad \textrm{for   }\;0<\xi\leq \lambda, \\
   12c_1 \gamma \xi^{-\alpha}, & \quad \textrm{for   }\;\lambda <\xi \leq \frac{r_0}{2}.
  \end{cases}
\end{equation}
Here, we make use of the estimate $\omega(\xi) \leq 2 M_{2,T}$ (from \eqref{eq:Xi1}).

To estimate $\|\partial_xG(\cdot,t)\|_{L^\infty}$, we use the relation
\begin{equation}\label{par-G-eq}
  \partial_xG=\partial_x(\rho F)=\rho'F+\rho^2 H.
\end{equation}
Applying \eqref{Flinf-es}, \eqref{eq:uppbdd},\eqref{Hlinf-es} and
\eqref{eq:rho-lip2}, we get
\begin{equation}\label{par-G-es}
  \|\partial_x G\|_{L^\infty(\T\times [0,T])}
  \leq M_{2,T} \|F_0\|_{L^\infty} + M_1^2\|H_0\|_{L^\infty} .
\end{equation}

Putting together \eqref{rhoLip}, \eqref{MOCes-Lam-alp2} and
\eqref{par-G-es}, we end up with an estimate on $\mathcal{IV}_1$:
\begin{equation}\label{cIV-2}
  \mathcal{IV}_1\leq M_{2,T}^2(4c_3 + \|F_0\|_{L^\infty}) \xi +
  M_{2,T} M_1^2\|H_0\|_{L^\infty}\xi+
  \begin{cases}
    \frac{12 c_1}{\alpha-1}  M_{2,T} \,\delta\lambda^{-\frac{\alpha}{2}} \xi^{1-\frac{\alpha}{2}},
    & \; \textrm{for   }\;0<\xi\leq \lambda, \\
    12c_1 M_{2,T}\, \omega(\xi) \xi^{1-\alpha},
    & \; \textrm{for   }\;\lambda<\xi\leq \frac{r_0}{2},
  \end{cases}
\end{equation}
which has a similar structure as the estimate on $\mathrm{II}$ in
\eqref{rII-es}.
Note that in the last part, we use the fact
$\gamma\leq\frac{\delta}{2}\leq\omega(\lambda)\leq\omega(\xi)$ for
every $\xi> \lambda$.

Next, we estimate the term $\mathcal{IV}_2$, similarly as
$\mathrm{III}$. In particular,
\[\partial_x^2u(x)-\partial_x^2u(y)=\left(\LL\rho'(x)-\LL\rho'(y)\right)
  +\left(\partial_xG(x)-\partial_xG(y)\right).\]
For the first term (corresponding to $\mathrm{III}_1$), applying Lemma~\ref{lem:MOCes1} on $\rho'$, we obtain
\[\LL \rho' (x) - \LL \rho' (y)\geq D_1'(x,y)-2c_2\omega(\xi),\]
where $D_1'(x,y)$ is defined in \eqref{Dexp} with $\rho$ replaced by $\rho'$, satisfying
\begin{equation}\label{D'est2}
  D'_1(x,y) \geq
  \begin{cases}
    \frac{1}{32 c_1}\delta \lambda^{-1-\frac{\alpha}{2}} \xi^{1- \frac{\alpha}{2}}, &\quad \mathrm{for}\;\; 0<\xi \leq \lambda, \\
    \frac{1}{4c_1} \omega(\xi) \xi^{-\alpha}, &\quad \mathrm{for}\;\; \lambda<  \xi \leq \frac{a_0}{2}.
  \end{cases}
\end{equation}
For the second term, use the relation \eqref{par-G-eq} and get
\begin{equation}\label{GLipest}
  \partial_xG(x)-\partial_xG(y)=\left(\rho'(x)F(x)-\rho'(y)F(y)\right)
  +\left(\rho^2(x)H(x)-\rho^2(y)H(y)\right).
\end{equation}
The two parts can be estimated similarly as the terms $\mathrm{III}_2$
and $\mathrm{III}_3$ as follows.
For the first part, apply \eqref{Flinf-es}, \eqref{f-diff} and \eqref{eq:rho-lip2}
\begin{align*}
  |\rho'(x)F(x)-\rho'(y)F(y)|=&\,|(\rho'(x)-\rho'(y))F(y)+\rho'(x)(F(x)-F(y))|\\
  \leq&\, \|F_0\|_{L^\infty}\omega(\xi)+M_{2,T}M_1\|H_0\|_{L^\infty}\xi.
\end{align*}
For the second part, observe that $\partial_x H$ and $\frac{\partial_x H}{\rho}$ satisfy
\begin{equation}\label{eq:par-H}
  \partial_t (\partial_x H) + \partial_x ( u\, (\partial_x H))=0,\quad\textrm{and}\quad \partial_t\Big( \frac{\partial_x H}{\rho}\Big)  +  u\,\partial_x \Big( \frac{\partial_x H}{\rho} \Big)=0,
\end{equation}
which directly implies that
\begin{equation}\label{eq:par-H-es}
  \|\partial_x H\|_{L^\infty(\T\times [0,T])} \leq \|\rho\|_{L^\infty(\T\times [0,T])} \Big\|\frac{\partial_x H}{\rho}\Big\|_{L^\infty(\T\times [0,T])} \leq M_1 \Big\| \frac{\partial_x H_0}{\rho_0}\Big\|_{L^\infty}.
\end{equation}
Therefore,
\begin{align*}
  |\rho^2(x)H(x)-\rho^2(y)H(y)|=&\,
  |(\rho^2(x)-\rho^2(y))H(y)+\rho^2(x)(H(x)-H(y))|\\
  \leq&\,2M_1M_{2,T}\|H_0\|_{L^\infty}\xi+M_1^3 \Big\| \frac{\partial_x H_0}{\rho_0}\Big\|_{L^\infty}\xi.
\end{align*}

We summarize the estimate on $\mathcal{IV}_2$ as
\begin{equation}\label{cIV-3}
\begin{split}
  \mathcal{IV}_2 \leq  -\rho_{\min,T} D'_1(x,y)
  + M_1(2c_2+\|F_0\|_{L^\infty})\omega(\xi) + M_1^2\Big(3 \|H_0\|_{L^\infty} M_{2,T} + M_1^2 \Big\|\frac{\partial_x H_0}{\rho_0} \Big\|_{L^\infty}  \Big) \xi .
\end{split}
\end{equation}

Now, we consider the contribution from terms $\mathcal{II}$ and
$\mathcal{III}$ given by \eqref{rho'-moc-dec}. These two terms do not
appear in the estimates on the MOC of $\rho(t)$.
Yet, they play a crucial role in the estimate on the MOC of
$\rho'(t)$.
The following key lemma describes the bounds on $\partial_xu(x)$ and
$\partial_xu(x)-\partial_xu(y)$, which can be used to estimate
$\mathcal{II}$ and $\mathcal{III}$ respectively. The proof is placed
in Section \ref{sec:MOCes}).

\begin{lemma}\label{lem:MOCes5}
  Let $\alpha\in(1,2)$. Assume $\rho'$ obeys the MOC defined in
  \eqref{MOC1}. Then, for any $\tilde{x}\in\T$, we have
\begin{equation}\label{par-u-Linf-es}
  |\partial_x u(\tilde{x})| \leq
  \frac{4c_1}{(\alpha-1)^2(2-\alpha)}\delta\lambda^{-(\alpha-1)}+\frac{c_3}{2}M_{2,T}+\|F_0\|_{L^\infty}M_1.
\end{equation}
 Moreover, if $x, y$ satisfy the breakthrough scenario described in
 Lemma~\ref{lem:bd-scena2}, with $\xi=|x-y|\in(0,\frac{r_0}{4}]$,

 \begin{align}
   |\partial_x u(x) - \partial_x u(y)|\leq&\,
  4c_1^2D_1'(x,y)\xi\,+\,\begin{cases}
   \frac{54c_1}{\alpha-1} \delta\lambda^{-\frac\alpha2}\xi^{1-\frac\alpha2}, & \;\textrm{for   }\; 0<\xi \leq \lambda \\
    26 c_1  \omega(\xi) \xi^{1-\alpha}, & \;\textrm{for   }\; \lambda <\xi \leq \frac{r_0}{4}
  \end{cases}\nonumber\\
   & + \Big( (16c_3+\|F_0\|_{L^\infty})M_{2,T}+M_1^2\|H_0\|_{L^\infty}
     \Big) \xi.\label{par-u-diff-es}
 \end{align}
\end{lemma}

Apply scenario \eqref{eq:scena3} and estimate \eqref{eq:rho-lip2} to
Lemma~\ref{lem:MOCes5} and get
\begin{equation}\label{cII-es}
  |\mathcal{II}| \leq  \Big( \frac{8c_1}{(\alpha-1)^2(2-\alpha)}\delta\lambda^{-(\alpha-1)}+ c_3 M_{2,T} + 2\|F_0\|_{L^\infty}M_1\Big) \omega(\xi),
\end{equation}
and
 \begin{align}
   |\mathcal{III}|\leq&\,
  8c_1^2M_{2,T}D_1'(x,y)\xi\,+\,\begin{cases}
   \frac{108c_1}{\alpha-1} M_{2,T}\,\delta\lambda^{-\frac\alpha2}\xi^{1-\frac\alpha2}, & \;\textrm{for   }\; 0<\xi \leq \lambda \\
    52 c_1 M_{2,T} \omega(\xi) \xi^{1-\alpha}, & \;\textrm{for   }\; \lambda <\xi \leq \frac{r_0}{4}
  \end{cases}\nonumber\\
   &+2M_{2,T}\Big( (16c_3+\|F_0\|_{L^\infty})M_{2,T}+M_1^2\|H_0\|_{L^\infty}
     \Big) \xi.\label{cIII-es}
 \end{align}

Finally, for the drift term $\mathcal{I}$, thanks to the estimate
\eqref{par-u-Linf-es}, we argue similarly as \eqref{rI-es0} and directly calculate
\begin{equation}\label{cI-es}
  |\mathcal{I}| \leq \|\partial_x u\|_{L^\infty} \xi \omega'(\xi) \leq \left( \frac{4c_1}{(\alpha-1)^2(2-\alpha)}\delta\lambda^{-(\alpha-1)}+\frac{c_3}{2}M_{2,T}+\|F_0\|_{L^\infty}M_1\right) \omega'(\xi)\xi.
\end{equation}

Hence, gathering the splitting \eqref{rho'-moc-dec} and estimates \eqref{cIV-2}, \eqref{cIV-3}, \eqref{cII-es}, \eqref{cIII-es}, \eqref{cI-es},
we find that for every $0< \xi \leq \Xi_1$,
\begin{align}\label{Targ-es4}
  & \partial_t \rho'(x) -\partial_t \rho'(y)\, \leq \, -\Big( \rho_{\min,T} - 8c_1^2 M_{2,T}  \xi  \Big) D'_1(x,y)  \\
  & \,\,\,+  \left(\frac{4c_1}{(\alpha-1)^2(2-\alpha)}\delta\lambda^{-(\alpha-1)}+\frac{c_3}{2}M_{2,T}+\|F_0\|_{L^\infty}M_1\right) \big( \omega(\xi) + \omega'(\xi)\xi\big)\nonumber\\
  &\,\,\, +  M_1(2c_2+\|F_0\|_{L^\infty})\omega(\xi)+\tilde{C}_0\xi\,+ \,\begin{cases}
   \frac{120c_1}{\alpha-1} M_{2,T}\delta\lambda^{-\frac\alpha2}\xi^{1-\frac\alpha2}, & \;\textrm{for   }\; 0<\xi \leq \lambda, \\
    64 c_1 M_{2,T} \omega(\xi) \xi^{1-\alpha}, & \;\textrm{for   }\;
    \lambda <\xi \leq \frac{r_0}{4},
  \end{cases}\nonumber
\end{align}
where $\tilde{C}_0=\tilde{C}_0(\rho_0,u_0,T)$ is given by
\[\tilde{C}_0=M_{2,T}\Big(36 c_3 M_{2,T} + 3\|F_0\|_{L^\infty} M_{2,T}  + 6M_1^2\|H_0\|_{L^\infty}\Big)+ M_1^4\Big\|\frac{\partial_x H_0}{\rho_0}\Big\|_{L^\infty}.
\]

In order to show the right hand side of \eqref{Targ-es4} is negative,
we first set $8c_1^2 M_{2,T}  \xi\leq\frac14\rho_{\min,T}$.
Since $\xi\leq\Xi_1 =\lambda e^{\gamma^{-1}(2M_{2,T}-\frac34\delta)}$ (see \eqref{eq:Xi1}),
the bound can be guaranteed by choosing $\lambda$ sufficiently small
\begin{equation}\label{dkg'-cd3}
  \lambda\leq \frac{\rho_{\min,T}}{32c_1^2M_{2,T}}e^{-2\gamma^{-1}M_{2,T}}.
\end{equation}

It remains to show that the rest of the terms in the second and third lines of
\eqref{Targ-es4} are bounded by
$\frac{1}{2}\rho_{\min,T}D_1'(x,y)$, or sufficiently, from \eqref{D'est2}, bounded by
\begin{equation}\label{D1'bd}
  \begin{cases}
    \frac{\rho_{\min,T}}{64 c_1}\delta \lambda^{-1-\frac{\alpha}{2}} \xi^{1- \frac{\alpha}{2}}, &\quad \mathrm{for}\;\; 0<\xi \leq \lambda, \\
    \frac{\rho_{\min,T}}{8c_1} \omega(\xi) \xi^{-\alpha}, &\quad \mathrm{for}\;\; \lambda<  \xi \leq \Xi_1.
  \end{cases}
\end{equation}
Then, we conclude with
$\partial_t\rho'(x)-\partial_t\rho'(y)<0$ by \eqref{nega}, that finishes
the proof.

The bounds can be achieved by choosing $\lambda$ sufficiently
small, given $\delta$ and $\gamma$. To see this, we consider two cases.

Case 1: $0<\xi\leq \lambda$.
In this case $\omega(\xi) \leq \delta \lambda^{-1}\xi $, and
$\omega'(\xi)\xi \leq \delta\lambda^{-1}\xi$ as well. Comparing the
parameters in \eqref{Targ-es4} and \eqref{D1'bd}:
\begin{align*}
  \delta\lambda^{-(\alpha-1)}\big(\omega(\xi)+\omega'(\xi)\xi\big)
\leq&\, 2\delta^2\lambda^{-\alpha}\xi\leq
2\delta^2\lambda^{-\frac\alpha2}\xi^{1-\frac\alpha2}
\leq (2\delta\lambda) \cdot\delta \lambda^{-1-\frac{\alpha}{2}}\xi^{1- \frac{\alpha}{2}},\\
 \big(\omega(\xi)+\omega'(\xi)\xi\big)
\leq&\,  2\delta\lambda^{-1}\xi\leq
2\delta\lambda^{-1+\frac\alpha2}\xi^{1-\frac\alpha2}
\leq (2\lambda^\alpha) \cdot\delta \lambda^{-1-\frac{\alpha}{2}}\xi^{1- \frac{\alpha}{2}},\\
  \xi\leq&\,  \lambda^{\frac\alpha2}\xi^{1-\frac\alpha2}\leq
       \lambda^\alpha \cdot\delta \lambda^{-1-\frac{\alpha}{2}}\xi^{1-\frac{\alpha}{2}},\quad \text{for}\,\lambda<\delta,\\
  \delta\lambda^{-\frac\alpha2}\xi^{1-\frac\alpha2}\leq&\,
\lambda\cdot\delta\lambda^{-1-\frac\alpha2}\xi^{1-\frac\alpha2},
\end{align*}
we have
$\max\{2\delta\lambda, 2\lambda^\alpha,
\lambda^\alpha,\lambda\}\leq2\lambda$. Therefore, setting $\lambda$
small enough will indeed make the terms under control.

Case 2: $\lambda< \xi \leq \Xi_1$.
In this case we have
$\omega'(\xi)\xi =\gamma < \frac{3\delta}{4} = \omega(\lambda)
\leq \omega(\xi)$. Also, we recall $\xi\leq\Xi_1\leq C_\gamma \lambda$ with
the constant $C_\gamma=e^{2\gamma^{-1}M_{2,T}}$ (from \eqref{eq:Xi1}).
Comparing the parameters in \eqref{Targ-es4} and \eqref{D1'bd}:
\begin{align*}
  \delta\lambda^{-(\alpha-1)}\big(\omega(\xi)+\omega'(\xi)\xi\big)
  \leq&\, 2\delta^2\lambda^{-(\alpha-1)}\omega(\xi)\leq
  2\delta^2\lambda^{-(\alpha-1)}\big(C_\gamma\lambda\big)^{\alpha}\omega(\xi)\xi^{-\alpha}
\leq (2\delta^2C_\gamma^\alpha\lambda) \cdot\omega(\xi)\xi^{-\alpha},\\
 \big(\omega(\xi)+\omega'(\xi)\xi\big)
\leq&\,  2\omega(\xi)\leq
2\big(C_\gamma\lambda\big)^{\alpha}\omega(\xi)\xi^{-\alpha}
\leq (2C_\gamma^\alpha\lambda^\alpha) \cdot \omega(\xi)\xi^{-\alpha},\\
  \xi\leq&\,  \big(C_\gamma\lambda\big)^{1+\alpha}\frac{4}{3\delta}\omega(\xi)\xi^{-\alpha}\leq
      \big(2C_\gamma^{1+\alpha}\lambda^\alpha\big) \cdot \omega(\xi)\xi^{-\alpha},\quad \text{for}\,\lambda<\delta,\\
  \omega(\xi)\xi^{1-\alpha}\leq&\,
 \big(C_\gamma\lambda\big)\cdot\omega(\xi)\xi^{-\alpha},
\end{align*}
we have $\max\{2\delta^2C_\gamma^\alpha\lambda,
2C_\gamma^\alpha\lambda^\alpha, 2C_\gamma^{1+\alpha}\lambda^\alpha,
C_\gamma\lambda\}\leq 2C_\gamma^{1+\alpha}\lambda$. Therefore, setting
$\lambda$ small enough will make the terms under the desired control.

\begin{remark}
  As $T$ becomes large, the $\lambda$ could grow very fast. Indeed,
  from Remark~\ref{rmk:growth}, we know $M_{2,T}$ can grow double
  exponentially in $T$. With smallness assumption
  (e.g. \eqref{del-cond2} and \eqref{dkg'-cd3}) on $\lambda$, we see $\lambda^{-1}$ could
  grow triple exponentially in time. Thus, the bound on
  $\|\partial_x^2u(\cdot,t)\|_{L^\infty}$ in \eqref{eq:rho'-lip} is
  also triple exponential in time. Such possible fast growth does not
  happen without the presence of the misalignment.
\end{remark}

\section{Estimates concerning the modulus of continuity}\label{sec:MOCes}

In the section, we give the detailed proof of Lemmas \ref{lem:MOCes1},
\ref{lem:MOCes2}, \ref{lem:MOCes-u}, \ref{lem:MOCes5}, respectively in
order. All estimates are scaling critical. The idea of the proofs
follows from \cite{DKRT}. The main contribution is the inclusion of
the misalignment, and the generalization of the influence function $\phi$.

\begin{proof}[Proof of Lemma \ref{lem:MOCes1}]
  First, we decompose $D(x,y)$ into two parts
\begin{align*}
  D(x,y) & = \mathrm{p.v.}\, \int_{\R} \phi(z) \big( \omega(\xi) - \rho(x+z)+ \rho (y+ z)\big) \dd z \\
  & = D_1(x,y) + \int_{|z|\geq a_0} \phi(z)\,\big( \omega(\xi) - \rho(x+z)+ \rho (y+ z) \big) \dd z \\
  & \geq D_1(x,y) - 2\omega(\xi) \int_{|z|\geq a_0} |\phi(z)| \dd z \geq D_1(x,y) -  2 c_2 \omega(\xi) .
\end{align*}
 Here, $D_1$ is defined in \eqref{Dexp}, which characterizes the
 dissipation phenomenon in the short range. The second term represents
 the long range misalignment, and can be bounded by condition
 \eqref{phi-assum2}.
 This yields the estimate \eqref{Dest0}.

 \def\xih{{\frac\xi2}}

 The dissipation $D_1(x,y)$ has lower bound similar as in
 \cite[Lemma~4.5]{DKRT}, where $\phi(r)=r^{1+\alpha}$.
 To work with general influence functions, we
 adapt the argument in \cite[Lemma~2.3]{DKSV}, with a small variation
 to treat with influence functions that are compactly supported.
 Due to translation invariance and symmetry, we can let
 $x=\xih$ and $y=-\xih$ without loss of generality. In the following
 calculation, integrals make sense in principle values.
\begin{align*}
  D_1 =  &\, \bigg(\int_{-a_0}^{-\xih}+\int_{-\xih}^{-a_0}\bigg)\phi(z) \big(
                \omega(\xi) - \rho(x+z)+ \rho (y+ z)\big) \dd z\\
  = &\,   \int_0^{a_0-\xih} \phi\Big(\eta+\xih\Big)
      \big(\omega(\xi)+\rho(\eta)-\rho(-\eta)\big)\dd \eta
      +\int_0^{a_0+\xih} \phi\Big(\eta-\xih\Big)
      \big(\omega(\xi)-\rho(\eta)+\rho(-\eta)\big)\dd \eta\\
  =  &\,\int_0^{a_0-\xih} \left(\left[\phi\Big(\eta-\xih\Big)+
       \phi\Big(\eta+\xih\Big)\right]\omega(\xi)+
       \left[\phi\Big(\eta-\xih\Big)-\phi\Big(\eta+\xih\Big)\right]
       \big(-\rho(\eta)+\rho(-\eta)\big)\right)\,\dd\eta\\
  &+\int_{a_0-\xih}^{a_0+\xih}\phi\Big(\eta-\xih\Big)
    \big(\omega(\xi)-\rho(\eta)+\rho(-\eta)\big)\dd \eta.
\end{align*}

Due to the monotonicity assumption \eqref{phi-assum1.3} on $\phi$, it
is easy to check
\[\phi\Big(\eta-\xih\Big)-\phi\Big(\eta+\xih\Big)\geq0,\quad\forall~\eta\in\Big[0,a_0-\xih\Big].\]
Moreover, the breakthrough scenario \eqref{eq:scena0} implies
$|\rho(\eta)-\rho(-\eta)|\leq\omega(2\eta)$.
We can obtain a lower bound on $D_1$:
\begin{align*}
D_1\geq& \,\int_0^{a_0-\xih} \left(\left[\phi\Big(\eta-\xih\Big)+
       \phi\Big(\eta+\xih\Big)\right]\omega(\xi)-
       \left[\phi\Big(\eta-\xih\Big)-\phi\Big(\eta+\xih\Big)\right]
       \omega(2\eta)\right)\,\dd\eta\\
  &+\int_{a_0-\xih}^{a_0+\xih}\phi\Big(\eta-\xih\Big)
    \big(\omega(\xi)-\omega(2\eta)\big)\dd \eta\\
  = &
      \,\int_{-\xih}^{a_0-\xi}\phi(\eta)\big(\omega(\xi)-\omega(2\eta+\xi)\big)\dd\eta
      \,+\int_{\xih}^{a_0}\phi(\eta)\big(\omega(\xi)+\omega(2\eta-\xi)\big)\dd\eta\\
   &\,+\int_{a_0-\xi}^{a_0}\phi(\eta)\big(\omega(\xi)-\omega(2\eta+\xi)\big)\dd\eta\\
   = &\,
       \int_0^{\xih}\phi(\eta)\big(2\omega(\xi)-\omega(2\eta+\xi)-\omega(\xi-2\eta)\big)\dd\eta
       \,+\int_{\xih}^{a_0}\phi(\eta)\big(2\omega(\xi)+\omega(2\eta-\xi)-\omega(2\eta+\xi)\big)\dd\eta.
\end{align*}

Due to the concavity of $\omega(\xi)$, both terms $2\omega(\xi)-\omega(2\eta+\xi)-\omega(\xi-2\eta)$ and $2\omega(\xi)+\omega(2\eta-\xi)-\omega(2\eta+\xi)$
are positive. Thus assumption \eqref{phi-assum1} implies the wanted inequality \eqref{Dest}.

Next, we prove estimate \eqref{Dest2}, which is from direct
calculation.

Case 1: $0<\xi\leq \lambda$. We only keep the first term. By concavity
of $\omega(\xi)$,
\begin{equation}\label{omega-fact}
\begin{split}
  \omega(\xi+2\eta)+\omega(\xi-2\eta)-2\omega(\xi) & = 4\eta^2 \int_0^1 \int_{-1}^1 s \omega''(\xi + 2s \tau\,\eta )\,\dd \tau \dd s \\
  & \leq 4\eta^2 \int_0^1 \int_{-1}^0 s \omega''(\xi)\,\dd\tau \dd s  \leq 2\omega''(\xi) \eta^2.
\end{split}
\end{equation}
Then, we have
\begin{align}\label{D1est1}
  D_1(x,y) & \geq \frac{1}{c_1} \int_0^{\frac{\xi}{2}} \frac{-2\omega''(\xi) \eta^2}{\eta^{1+\alpha}} \dd \eta
  \geq \frac{\alpha(2+\alpha)}{8 c_1} \delta \lambda^{-1-\frac{\alpha}{2}} \xi^{\frac{\alpha}{2}-1}  \int_0^{\frac{\xi}{2}} \eta^{1-\alpha}\dd \eta  \nonumber \\
  & \geq \frac{ \alpha(2+\alpha) }{ 2^{2-\alpha}(2-\alpha)8c_1} \delta \lambda^{-1-\frac{\alpha}{2}} \xi^{1-\frac{\alpha}{2}}
  \geq \frac{ \alpha}{16(2-\alpha) c_1}\delta \lambda^{-1-\frac{\alpha}{2}} \xi^{1- \frac{\alpha}{2} } .
\end{align}

Case 2: $\lambda\leq \xi\leq \frac{a_0}{2}$. We only keep the second
term. Due to the concavity of $\omega$, we have for every $\eta\geq\frac\xi2$,
\[\omega(2\eta+\xi)-\omega(2\eta-\xi)\leq
  \omega(2\xi)=\omega(\xi)+\gamma\log2
  \leq\frac{3}{2}\omega(\xi),\]
where the last inequality holds since $\gamma<\frac\delta2$ and so
\[\gamma\log2<\frac38\delta=\frac12\omega(\delta)\leq\frac12\omega(\xi).\]
Thus, we find
\begin{equation}\label{D1est2}
\begin{split}
  D_1(x,y) \geq \frac{1}{2 c_1}\omega(\xi) \int_{\frac{\xi}{2}}^{a_0} \frac{1}{\eta^{1+\alpha}} \dd \eta
  \geq \frac{1}{2 c_1\alpha} \omega(\xi) \left[\Big(\frac\xi2\Big)^{-\alpha} - (2\xi)^{-\alpha}\right]
  \geq \frac{2^\alpha -1}{2 c_1 \alpha} \frac{\omega(\xi)}{\xi^\alpha} .
\end{split}
\end{equation}
Combining \eqref{D1est1} with \eqref{D1est2} leads to \eqref{Dest2}, as desired.
\end{proof}

\begin{proof}[Proof of Lemma \ref{lem:MOCes2}]
 The proof is similar to \cite[Lemma 4.5]{DKRT}, with suitable
 modifications that address the misalignment effect.
 We will only prove the lower bound on $\LL\rho(x)$. The upper bound on $\LL\rho(y)$
 can be obtained using the same argument.

 Without loss of generality, we assume that $\xi = x-y >0$.
 By using the periodicity property of $\rho$ and the scenario \eqref{eq:scena}, we see that
\begin{align}\label{lam-alp-esp2}
  \LL \rho(x) &= \,\mathrm{p.v.} \int_\R \phi(x-z)\big(\rho(x)-\rho(z) \big) \dd z
  =\, \mathrm{p.v.} \int_\T \phi^S(x-z)\, \big( \rho(x) - \rho(z)\big) \dd z \nonumber \\
  & = \,\mathrm{p.v.}\int_\T \phi^S(\eta)\,\big(\rho(x)-\rho(y) + \rho(y)- \rho(x-\eta) \big) \dd \eta \nonumber \\
  & = \, \mathrm{p.v.}\int_\T \phi^S(\eta)\,\big(\omega(\xi) + \rho(y)- \rho(y+\xi-\eta)\big) \dd \eta.
\end{align}
We have the following decomposition
\begin{align}\label{lam-alp-dec}
  \LL\rho(x) & = \left(\int_{-\frac{1}{2}}^{-\xi}  + \mathrm{p.v.}\int_{-\xi}^\xi + \int_\xi^{2\xi} + \int_{2\xi}^{\frac{1}{2}}\right) \Big(\phi^S(\eta)\big(\omega(\xi) + \rho(y)- \rho(y+\xi-\eta)\big) \dd \eta\Big) \nonumber \\
  & = A_{1,\phi} + A_{2,\phi} + A_{3,\phi} + A_{4,\phi} .
\end{align}
The terms $A_{2,\phi}$ and $A_{3,\phi}$ are nonnegative, which can be seen from scenario \eqref{eq:scena}, estimate \eqref{phi-s-assum1}
(with $2\xi \leq r_0$) and properties of $\omega$ (concavity and monotonicity):
\begin{equation}\label{A2alp-es}
\begin{split}
  A_{2,\phi} & = \,\mathrm{p.v.}\int_0^\xi \phi^S(\eta) \,\big(2\omega(\xi) + 2\rho(y) - \rho(y+\xi-\eta) - \rho(y+\xi +\eta)\big)  \dd \eta \\
  & \geq \,\mathrm{p.v.}\int_0^\xi \phi^S(\eta)\,\big(2\omega(\xi) -\omega(\xi-\eta) - \omega(\xi +\eta) \big) \dd \eta \geq 0,
\end{split}
\end{equation}
and
\begin{equation}\label{A3alp-es}
  A_{3,\phi} = \int_\xi^{2\xi}\phi^S(\eta)\,\big(\omega(\xi) + \rho(y)- \rho(y+\xi-\eta) \big) \dd \eta
  \geq  \int_\xi^{2\xi} \phi^S(\eta)\,\big(\omega(\xi) - \omega(\eta-\xi)\big)  \dd \eta\geq 0.
\end{equation}

Next, we obtain the upper bounds of $-A_{1,\phi}$ and $-A_{4,\phi}$.
\begin{align*}
  - A_{1,\phi}   = &~\int_{-\frac{1}{2}}^{-\xi} \phi^S(\eta) \big( \rho(y+\xi-\eta) - \rho(y) -\omega(\xi) \big) \dd \eta
  = \int_\xi^{\frac{1}{2}} \phi^S(\eta) \big(\rho(y+\xi +\eta) -\rho(y) -\omega(\xi) \big) \dd \eta \\
   \leq &~ \int_\xi^{r_0} |\phi^S(\eta)|\, \big(\omega(\xi + \eta)
    -\omega(\xi) \big) \dd \eta
    +\int_{r_0}^{\frac12} |\phi^S(\eta)| (2M_1)\dd\eta\\
  \leq&~ 2c_1\int_\xi^{r_0}\frac{\omega(\xi+\eta)-\omega(\xi)}{\eta^{1+\alpha}} \dd \eta +c_3M_1,
\end{align*}
where we make use of scenario \eqref{eq:scena0}, and also $\omega(\xi)\leq M_1$
due to \eqref{xi-scope}.
$-A_{4,\phi}$ can be estimated in the same way, with the same upper
bound as $-A_{1,\phi}$:
\[
  - A_{4,\phi}   = \int_{2\xi}^{\frac12} \phi^S(\eta) \big( \rho(y+\xi-\eta) - \rho(y) -\omega(\xi) \big) \dd \eta
  \leq~ 2c_1\int_{2\xi}^{r_0}\frac{\omega(\xi+\eta)-\omega(\xi)}{\eta^{1+\alpha}} \dd \eta +c_3M_1,
\]

Therefore, we conclude with \eqref{MOCes-LL}
\[-\LL\rho(x)\leq  4c_1\int_\xi^{r_0}\frac{\omega(\xi+\eta)-\omega(\xi)}{\eta^{1+\alpha}} \dd \eta + 2c_3M_1.\]

Next, we prove the estimate \eqref{MOCes-Lam-alp}.

Case 1: $0<\xi \leq \lambda$. The concavity of $\omega$ indicates
$\omega(\xi+\eta)-\omega(\xi)\leq\omega(\eta)$, and so
\begin{align}
  \int_\xi^{r_0}&\frac{\omega(\xi+\eta)-\omega(\xi)}{\eta^{1+\alpha}}
  \dd \eta \leq \int_\xi^{r_0}\frac{\omega(\eta)}{\eta^{1+\alpha}} \dd \eta
  \leq \delta\lambda^{-1}\int_\xi^\lambda\frac{1}{\eta^{\alpha}} \dd \eta
  +\int_\lambda^{r_0}\frac{\frac{3}{4}\delta+\gamma\log\frac{\eta}{\lambda}}{\eta^{1+\alpha}}
                 \dd \eta\nonumber\\
  &\leq\,\left(\frac{3}{4\alpha}+\frac{1}{2\alpha^2}\right)\delta\lambda^{-\alpha}+\begin{cases}
    \frac{1}{1-\alpha} \delta\lambda^{-\alpha},\quad &\textrm{for   }\; 0<\alpha <1, \\
     \delta \lambda^{-1} \log\frac{\lambda}{\xi},\quad &\textrm{for   }\; \alpha=1, \\
     \frac{1}{\alpha-1} \delta\lambda^{-1}\xi^{-(\alpha-1)},\quad & \textrm{for   }\; 1<\alpha<2,
   \end{cases}\nonumber\\
  &\leq\,\delta \overline{M}_\alpha(\xi,\lambda),\quad\text{with } \overline{M}_\alpha(\xi,\lambda):=\begin{cases}
    \frac{1}{\alpha^2(1-\alpha)} \lambda^{-\alpha},\quad &\textrm{for   }\; 0<\alpha <1, \\
    \lambda^{-1} \left(\log\frac{\lambda}{\xi}+\frac54\right),\quad &\textrm{for   }\; \alpha=1, \\
     \left(\frac{1}{\alpha-1}+\frac54\right)\lambda^{-1}\xi^{-(\alpha-1)},\quad & \textrm{for   }\; 1<\alpha<2,
   \end{cases}\label{omegaovereta}
\end{align}
where in the third inequality, we have used $\gamma<\frac\delta2$ and then
\[\int_\lambda^{r_0}\frac{\gamma\log\frac\eta\lambda}{\eta^{1+\alpha}}d\eta=\gamma\lambda^{-\alpha}\int_1^{r_0/\lambda}\frac{\log\zeta}{\zeta^{1+\alpha}}\dd\zeta\leq\frac{\gamma}{\alpha^2}\lambda^{-\alpha}\leq\frac{\delta}{2\alpha^2}\lambda^{-\alpha}.\]

The term $\overline{M}_\alpha(\xi,\lambda)$ is scaling critical. In
order to compare it with the dissipation, we state the following
inequality, where we only make use of the fact $\frac\xi\lambda\in(0,1]$
\begin{equation}\label{barC-alp}
  \overline{M}_\alpha(\xi,\lambda)\leq \overline{C}_\alpha \lambda^{-\frac{\alpha}{2}} \xi^{-\frac{\alpha}{2}},\quad
  \textrm{with}\quad \overline{C}_\alpha =
  \begin{cases}
    \frac{1}{\alpha^2(1-\alpha)} ,\quad &\textrm{for   }\; 0<\alpha <1, \\
    2,\quad &\textrm{for   }\; \alpha=1, \\
     \frac{1}{\alpha-1}+\frac54,\quad & \textrm{for   }\; 1<\alpha<2.
  \end{cases}
\end{equation}

Case 2: $\lambda<\xi\leq\frac{r_0}{2}$.
We use the explicit formula on $\omega$ and get
\[
  \int_\xi^{r_0}\frac{\omega(\xi+\eta)-\omega(\xi)}{\eta^{1+\alpha}}\dd\eta=
  \gamma\int_\xi^{r_0}\frac{\log(\xi+\eta)-\log(\xi)}{\eta^{1+\alpha}}\dd\eta
    \leq
    \gamma\xi^{-\alpha}\int_1^\infty\frac{\log(1+\zeta)}{\zeta^{1+\alpha}}\dd\zeta
    \leq \frac{\gamma (1+\alpha)}{\alpha^2}\xi^{-\alpha}.
\]

Collecting the above estimates yields the desired estimate \eqref{MOCes-Lam-alp}.
\end{proof}

\def\phiSt{\widetilde{\phi}^S}

\begin{proof}[Proof of Lemma \ref{lem:MOCes-u}]
  We denote $\tilde{x}, \tilde{y}\in\T$ to be  arbitrary points with distance
  $\xi=\tilde{x}-\tilde{y}\in(0,\frac{r_0}{4}]$.

  Recalling that $u$ has the expression formula \eqref{u-exp} and $I_0(t)$ is uniformly bounded (see estimate \eqref{I0t-bdd}), we have
\begin{align}\label{u-es-decom}
  |u(\tilde{x})-u(\tilde{y})| \leq |\psi(\tilde{x})-\psi(\tilde{y})| + |\LL \varphi(\tilde{x}) -\LL \varphi(\tilde{y})|  := U_1 + U_2 ,
\end{align}
where $\psi$ and $\varphi$ are mean-free periodic functions satisfying $G=\partial_x\psi$ and $\theta=\rho-\bar{\rho}_0=\partial_x \varphi$.
By virtue of the mean value theorem and estimates \eqref{Flinf-es}, \eqref{eq:uppbdd}, it is easy to see that
\begin{equation}\label{U1-es}
  U_1 \leq \|G(t_1)\|_{L^\infty} \xi \leq \|F(t_1)\|_{L^\infty} \|\rho(t_1)\|_{L^\infty} \xi \leq M_1 \|F_0\|_{L^\infty} \xi .
\end{equation}

Before estimating $U_2$, we first show the following expression formula of $\LL \varphi$ (one can see \cite[Eq. (4.47)]{DKRT} at the case
$\LL =\Lambda^\alpha$ with $\alpha\in (0,1)$, and it also holds for the whole range $\alpha\in (0,2)$):
\begin{align}\label{Lam-alp-vphi}
  \LL \varphi(\tilde{x}) & = \lim_{\epsilon\rightarrow 0}\int_{|z|\geq \epsilon} \phi(z) \big(\varphi(\tilde{x})-\varphi(\tilde{x}+z)\big) \dd z
  = \lim_{\epsilon\rightarrow 0}\int_{\epsilon \leq |z|\leq \frac{1}{2}} \phi^S(z) \big(\varphi(\tilde{x})-\varphi(\tilde{x}+z)\big) \dd z \nonumber \\
  & = - \lim_{\epsilon\rightarrow 0}\int_{\epsilon \leq |z|\leq
    \frac{1}{2}} \phiSt(z)\, \theta(\tilde{x}+z) \dd z
    =-\mathrm{p.v.}\int_\T \phiSt(z)\theta(\tilde{x}+z) \dd z,
\end{align}
with
\begin{align}\label{tild-phi-df}
  \phiSt(z) = \mathrm{sgn}(z)\int_{|z|}^{\frac{1}{2}} \phi^S(r) \dd r,\quad\forall~z\in\T\backslash\{0\},
\end{align}
where the second equality follows from integration by parts
together with the facts $-\partial_z \phiSt(z)=\phi^S(z)$
for every $z\neq0$, $\widetilde{\phi}^S(\pm\frac12)=0$, and for any $\alpha\in(0,2)$
\begin{align*}
  \lim_{\epsilon\rightarrow 0} |\phiSt(\epsilon)(2\varphi(\tilde{x})-\varphi(\tilde{x}+\epsilon)-\varphi(\tilde{x}-\epsilon))|
  & \leq \|\partial_x^2\varphi\|_{L^\infty} \lim_{\epsilon\rightarrow 0} \epsilon^2 \int_\epsilon^{\frac{1}{2}} |\phi^S(r)|\dd r \\
  & \leq \omega'(0+) \lim_{\epsilon\rightarrow 0} \epsilon^2 \Big( \int_\epsilon^{r_0 } \frac{2c_1}{r^{1+\alpha}} \dd r
  + \int_{r_0}^{\frac{1}{2}} c_3 \dd r\Big) =0.
\end{align*}
Here, we use $\partial_x^2\varphi=\partial_x\rho$, which is bounded by
$\omega'(0+)$ at time $t_1$, which is finite due to \eqref{ome-cond}.

From \eqref{Lam-alp-vphi} and the oddness of kernel
$\widetilde{\phi}^S(z)$, we can rewrite
\begin{equation}\label{Lam-alp-vphi2}
  \LL \varphi(\tilde{x})  =  -\mathrm{p.v.}\int_\T \phiSt(z)\rho(\tilde{x}+z)\dd z = \mathrm{p.v.}\int_\T \phiSt(z) \big( \rho(\tilde{x})-\rho(\tilde{x}+z)\big)\dd z.
\end{equation}

Now, we begin to estimate $U_2$. The idea follows from
\cite[Appendix]{KNV}, with modifications to adapt the periodic
influence function $\phi^S$ with misalignment.

\def\xm{x_*}
Denote $\xm=\frac{\tilde{x}+\tilde{y}}{2}$.
Decompose $\LL\varphi(\tilde{x})-\LL\varphi(\tilde{y})$ as follows
\begin{align*}
\LL \varphi(\tilde{x})-\LL\varphi(\tilde{y}) =&\left(\mathrm{p.v.}\int_{|z|\leq 2\xi}  \phiSt(z)\big(\rho(\tilde{x})-\rho(\tilde{x}+z)\big) \dd z
  - \mathrm{p.v.}\int_{|z|\leq 2\xi}  \phiSt(z)\big(\rho(\tilde{y})-\rho(\tilde{y}+z)\big) \dd z \right) \\
  & + \left(\int_{2\xi\leq |z|\leq \frac{1}{2}}  \phiSt(z)\big(\rho(\xm)-\rho(\tilde{x}+z)\big)\dd z
  - \int_{2\xi\leq |z|\leq \frac{1}{2}} \phiSt(z)\big(\rho(\xm)-\rho(\tilde{y}+z)\big) \dd z \right) \\
  := & \,U_{21} + U_{22}. 
\end{align*}

For $U_{21}$, we apply \eqref{eq:scena2} and get
\begin{align}
  |U_{21}|\leq&\, 4\int_0^{2\xi}|\phiSt(\eta)|\omega(\eta)\dd \eta
  \leq\frac{8c_1}{\alpha}\int_0^{2\xi}\frac{\omega(\eta)}{\eta^\alpha}\dd \eta
  +2c_3\int_0^{2\xi}\omega(\eta)\dd \eta\nonumber\\
  \leq&\, \frac{32c_1}{\alpha}\int_0^\xi\frac{\omega(\eta)}{\eta^\alpha}\dd\eta+4c_3M_1\xi,\label{est:21}
\end{align}
where in the second inequality, we estimate $\phiSt$ using
\eqref{tild-phi-df} and conditions \eqref{phi-s-assum1} and \eqref{phi-s-assum2}:
\begin{equation}\label{phiSt-es}
  |\phiSt(z)|\leq \int_{|z|}^{r_0}\frac{2c_1}{r^{1+\alpha}}\dd
  r+\int_{r_0}^{\frac12}c_3\dd r\leq
  \frac{2c_1}{\alpha}\frac{1}{|z|^\alpha}+\frac{c_3}{2},\quad \forall 0<|z|\leq r_0,
\end{equation}
and in the last inequality, we change variable and use
$\omega(2\eta)\leq 2\omega(\eta)$ due to the concavity of $\omega$
\begin{equation}\label{omegashift}
  \int_0^{2\xi}\frac{\omega(z)}{z^\alpha}\dd
  z=2^{1-\alpha}\int_0^\xi\frac{\omega(2\eta)}{\eta^\alpha}\dd\eta\leq
2^{2-\alpha}\int_0^\xi\frac{\omega(\eta)}{\eta^\alpha}d\eta.
\end{equation}

For $U_{22}$, we need to make use of the cancelation. Decompose the
term as follows
\begin{align*}
  U_{22}& =\,\int_{\frac52\xi\leq|z-\xm|\leq\frac12}\big(\phiSt(z-\tilde{x})-\phiSt(z-\tilde{y})\big)
           \big(\rho(\xm)-\rho(z)\big)\dd z\\
         &\mbox{}\quad\; +\int_{-3\xi}^{-2\xi}\phiSt(z) \big(\rho(\xm)-\rho(\tilde{x}+z)\big)\dd z
           - \int_{2\xi}^{3\xi}\phiSt(z)
           \big(\rho(\xm)-\rho(\tilde{y}+z)\big)\dd z \\
  & =:\,U_{22a}+U_{22b}+U_{22c}.
\end{align*}
In the first part, change variable and use the Newton-Leibniz formula
\begin{align*}
  U_{22a} &= \int_{\frac52\xi\leq| z |\leq\frac12}\left(\phiSt\Big( z -\frac\xi2\Big)-\phiSt\Big( z +\frac\xi2\Big)\right)
  \big(\rho(\xm)-\rho(x_*+ z )\big)\dd z  \\
  & = -\xi \int_0^1 \int_{\frac{5}{2}\xi \leq | z |\leq \frac{1}{2}} \phi^S\Big( z  -\frac{\xi}{2} + \tau \xi\Big)
  \big(\rho(\xm)-\rho(x_*+ z )\big)\dd z  \dd \tau.
\end{align*}
From conditions \eqref{phi-s-assum1} and \eqref{phi-s-assum2}, 
it yields
\begin{align}\label{est:22a}
 & |U_{22a}| \leq \xi \int_0^1 \int_{\frac{5}{2}\xi \leq | z |\leq \frac{1}{2}} \Big|\phi^S\Big( z  -\frac{\xi}{2} + \tau \xi\Big)\Big|
  \omega(| z |)\, \dd z \dd \tau \nonumber \\
 & \leq \xi \int_0^1 \int_{\frac{5}{2}\xi \leq | z |\leq \frac{1}{2}, | z  -\frac{\xi}{2} +\tau\xi|\leq r_0}
 \frac{2 c_1 \omega(| z |)}{| z -\frac{\xi}{2} +\tau\xi|^{1+\alpha}} \dd  z \dd \tau
 + c_3 \xi \int_0^1 \int_{\frac{5}{2}\xi \leq | z |\leq \frac{1}{2}, | z  -\frac{\xi}{2} +\tau\xi|\geq r_0}
 \omega(| z |) \dd  z \dd \tau  \nonumber \\
 & \leq 4 c_1 \xi \int_{\frac{5}{2}\xi \leq | z | \leq r_0 +\xi} \frac{\omega(| z |)}{| z |^{1+\alpha}} \dd  z  +  c_3 M_1 \xi
 \leq 8c_1 \xi\int_{\frac52\xi}^{r_0 + \xi } \frac{\omega(\eta)}{\eta^{1+\alpha}}\dd\eta + c_3M_1 \xi ,
\end{align}
where in the last line we have used $(|z|-\frac{\xi}{2})^{-(1+\alpha)} \leq (\frac{4}{5}|z|)^{-(1+\alpha)}\leq 2 |z|^{-(1+\alpha)}$ for every $|z|\geq \frac{5}{2}\xi$.
For the second part, change variable
\begin{align*}
  |U_{22b}| = \bigg|\int_{\frac32\xi}^{\frac52\xi}\phiSt(\eta+\frac\xi2) \big(\rho(\xm)-\rho(\xm-\eta)\big)\dd\eta\bigg|
  \leq \int_{\frac32\xi}^{\frac52\xi}\Big|\phiSt(\eta+\frac\xi2)\Big|\omega(\eta)\dd\eta
  \leq \omega\big(\frac{5}{2}\xi\big) \int_{2\xi}^{3\xi} |\phiSt(\eta)| \dd\eta,
\end{align*}
and then it can be treated by using \eqref{phiSt-es} and concavity of $\omega$:
\begin{equation}\label{est:22b}
  |U_{22b}|\leq \frac{5}{2} \omega(\xi) \Big( \frac{2c_1}{\alpha(2\xi)^\alpha} +  \frac{c_3}{2} \Big)\xi
  \leq \frac{5c_1}{\alpha}\omega(\xi) \xi^{1-\alpha} + \frac{5c_3}{4}M_1\xi.
\end{equation}
The third part $U_{22c}$ can be estimated by the same bound as $U_{22b}$.

Collecting the estimates \eqref{U1-es}, \eqref{est:21}, \eqref{est:22a} and
\eqref{est:22b}, we obtain a bound on $\Omega(\xi)$
\[|u(\tilde{x})-u(\tilde{y})|\leq\frac{32c_1}{\alpha}\int_0^\xi\frac{\omega(\eta)}{\eta^\alpha}\dd\eta + \frac{10 c_1}{\alpha}\omega(\xi) \xi^{1-\alpha}
  +8c_1\xi\int_\xi^{r_0+\xi}\frac{\omega(\eta)}{\eta^{1+\alpha}}\dd\eta
+M_1(\|F_0\|_{L^\infty}+8c_3)\xi,\]
which combined with estimate $\int_0^\xi \frac{\omega(\eta)}{\eta^\alpha}\dd \eta \geq \frac{\omega(\xi)}{\xi} \int_0^\xi \frac{1}{\eta^{\alpha-1}} \dd \eta
= \frac{1}{2-\alpha} \omega(\xi)\xi^{1-\alpha}$ concludes the proof of \eqref{u-MOC-es}.

Next, we provide an explicit estimate of $\Omega(\xi)$ when $\omega(\xi)$ is
chosen as \eqref{MOC1}. For $0<\alpha<1$, one can follow a similar
procedure as \cite[Lemma 4.4]{DKRT}. However, it does not work for
$\alpha\geq1$.
In particular, the first term in \eqref{u-MOC-es} can not be controlled
by the dissipation term in the case $\xi>\lambda$.

To overcome the difficulty, we introduce an enhanced estimate on
$U_2$, when $(\tilde{x},\tilde{y})=(x,y)$ which satisfies the
breakthrough scenario \eqref{eq:scena}.
For $U_{21}$, we make use of the cancelation and bound the term by the
dissipation $D_1(x,y)$ as follows
\begin{align*}
  |U_{21}| =&\,
  \Big|\int_{|z|\leq2\xi}\phiSt(z)\big(\omega(\xi)-\rho(x+z)+\rho(y+z)\big)\dd z\Big| \\
  \leq&\,\int_{|z|\leq2\xi}\left(\frac{2c_1}{|z|^\alpha}+\frac{c_3}{2}\right)
        \big(\omega(\xi)-\rho(x+z)+\rho(y+z)\big)\dd z\\
  \leq&\, 4c_1^2\xi\int_{|z|\leq2\xi}\phi(z)
        \big(\omega(\xi)-\rho(x+z)+\rho(y+z)\big)\dd
        z+\int_{|z|\leq2\xi}\frac{c_3}{2}\cdot(2M_1)\dd z\\
  \leq&\,4c_1^2D_1(x,y)\xi+4c_3M_1\xi
\end{align*}
where in the first inequality, we use \eqref{phiSt-es} and the fact
that $\omega(\xi)-\rho(x+z)+\rho(y+z)\geq0$,
in the second inequality, we use \eqref{phi-assum1} and then
\[\frac{1}{|z|^\alpha}\leq\frac{2\xi}{|z|^{1+\alpha}}\leq2c_1\xi\phi(z),
  \quad \forall |z|\leq2\xi,\]
and in the third inequality, we use the definition of $D_1(x,y)$
\eqref{Dexp}.
The estimation of $U_1$ and $U_{22}$ is the same as above.
Then, we end up with a better estimate on $u(x)-u(y)$:
\begin{align*}
  |u(x)-u(y)|\leq&\, 4c_1^2D_1(x,y)\xi + 8c_1 \xi\int_\xi^{r_0 + \xi } \frac{\omega(\eta)}{\eta^{1+\alpha}}\dd\eta
  + \frac{10c_1}{\alpha}\omega(\xi) \xi^{1-\alpha}  + M_1(8c_3+\|F_0\|_{L^\infty})\xi.
\end{align*}
Compared with \eqref{u-MOC-es}, the problematic term is replaced by a
new term involving $D_1(x,y)$, which is controllable by the dissipation.

Finally, let us calculate explicit bounds on the terms
$\xi\int_\xi^{r_0 + \xi} \frac{\omega(\eta)}{\eta^{1+\alpha}}\dd\eta$ and $\omega(\xi) \xi^{1-\alpha}$
when we choose the MOC in \eqref{MOC1}.

Case 1: $0<\xi\leq\lambda$. As a direct consequence of
\eqref{omegaovereta} and \eqref{barC-alp}, we have
\[\xi\int_\xi^{r_0 + \xi}\frac{\omega(\eta)}{\eta^{1+\alpha}}\dd\eta
\leq\overline{C}_\alpha
\delta\lambda^{-\frac\alpha2}\xi^{1-\frac\alpha2}.\]
From formula \eqref{MOC1} and the fact $\overline{C}_\alpha \geq \frac{1}{\alpha}$, it follows
\begin{equation*}
  \frac{1}{\alpha}\omega(\xi) \xi^{1-\alpha} \leq \frac{1}{\alpha}\delta \lambda^{-1} \xi^{2-\alpha} \leq \overline{C}_\alpha \delta \lambda^{-\frac{\alpha}{2}} \xi^{1-\frac{\alpha}{2}}.
\end{equation*}

Case 2: $\lambda<\xi\leq\frac{r_0}{4}$. Direct calculation leads to
\begin{align*}
  \xi\int_\xi^{r_0 + \xi}\frac{\omega(\eta)}{\eta^{1+\alpha}}\dd\eta
  =\xi\int_\xi^{r_0 + \xi}\frac{\frac34\delta+\gamma\log\frac\eta\lambda}{\eta^{1+\alpha}}\dd\eta
  \leq&\,\frac{3 \delta}{4\alpha}\xi^{1-\alpha}+
   \frac{\gamma}{\alpha^2}\xi^{1-\alpha}\left(\alpha\log\frac\xi\lambda+1\right)\\
  =&\,\frac{1}{\alpha}\xi^{1-\alpha}\omega(\xi)+ \frac{\gamma}{\alpha^2} \xi^{1-\alpha}
     \leq\frac{2}{\alpha} \omega(\xi) \xi^{1-\alpha},
\end{align*}
where in the last inequality, we apply \eqref{gamma-cd} and $\frac\gamma\alpha\leq\frac{3}{4}\delta=\omega(\lambda)<\omega(\xi)$.

Collecting all the estimates above, we conclude with
\eqref{u-enhanced}, as desired.
\end{proof}

\begin{proof}[Proof of Lemma \ref{lem:MOCes5}]
We first consider estimate \eqref{par-u-Linf-es}. From relation
$\partial_x u =\LL \rho + G$ and the estimate $\|G\|_{L^\infty}\leq
\|F_0\|_{L^\infty} M_1$ (see \eqref{G-Linf-es1}),
it suffices to bound $\LL \rho$.

Let $\tilde{x}\in\T$.
Through a similar argument as obtaining \eqref{Lam-alp-vphi}, we can verify
\begin{equation}\label{LLrho2}
  \LL\rho(\tilde{x}) = -\mathrm{p.v.}\int_\T \phiSt(z)\rho'(\tilde{x}+z) \dd z,
\end{equation}
where $\phiSt$ is defined in \eqref{tild-phi-df} satisfying estimate \eqref{phiSt-es}. We compute
\begin{align*}
  |\LL\rho(\tilde{x})|=&\,\bigg|\int_0^{r_0}\phiSt(\eta)(\rho'(\tilde{x}+\eta)-\rho'(\tilde{x}-\eta))\dd \eta
                 +\int_{r_0}^{\frac12}\phiSt(\eta)(\rho'(\tilde{x}+\eta)-\rho'(\tilde{x}-\eta))\dd \eta\bigg|\\
  \leq&\,\int_0^{r_0}\omega(2\eta)\left[\frac{2c_1}{\alpha \eta^\alpha}+\frac{c_3}{2}\right]\dd \eta
        +\int_{r_0}^{\frac12}\frac{c_3}{2}\cdot(2M_{2,T})\dd \eta\\
  \leq&\,\frac{2^\alpha c_1}{\alpha}\left[\int_0^{\lambda}
        \frac{\delta\lambda^{-1}}{\eta^{\alpha-1}}\dd\eta
        +\int_\lambda^{2r_0}\frac{\frac{3}{4}\delta+\gamma\log\frac{\eta}{\lambda}}{\eta^\alpha}\dd
        \eta\right]+\frac{c_3}{2}M_{2,T}\\
  \leq&\,\frac{2^\alpha
        c_1}{\alpha(2-\alpha)}\delta\lambda^{-(\alpha-1)}+\frac{2^\alpha
        c_1}{\alpha(\alpha-1)}\cdot\frac{3}{4}\delta\lambda^{-(\alpha-1)}+\frac{2^\alpha
        c_1}{\alpha(\alpha-1)^2}\gamma\lambda^{-(\alpha-1)}+\frac{c_3}{2}M_{2,T}\\
  \leq&\,\frac{4 c_1}{(\alpha-1)^2(2-\alpha)}\delta\lambda^{-(\alpha-1)}+\frac{c_3}{2}M_{2,T},
\end{align*}
which leads to the desired estimate \eqref{par-u-Linf-es}.

Next, we consider estimate \eqref{par-u-diff-es}.
Let $x,y\in\T$ be the points that satisfy the breakthrough scenario
\eqref{eq:scena3}.  Then,
\begin{equation}\label{par-u-diff-dec}
  \partial_x u(x) -\partial_x u(y)
  =\big(\LL \rho(x) -\LL \rho(y)\big)  + \big( G(x) - G(y) \big)
  =: \,\Pi_1 + \Pi_2 .
\end{equation}
For the term $\Pi_1$, since $\LL \rho(x)$ can be written as
\eqref{LLrho2}, we can directly apply the result in
Lemma~\ref{lem:MOCes-u}, and obtain
\[|\LL \rho(x) -\LL \rho(y)|\leq
  4c_1^2D_1'(x,y) + 8c_1 \xi\int_\xi^{r_0 + \xi } \frac{\omega(\eta)}{\eta^{1+\alpha}}\dd\eta
  + \frac{10c_1}{\alpha}\omega(\xi) \xi^{1-\alpha} + 16c_3M_{2,T}\xi, \]
by repeating the enhanced estimate on $U_2$, directly replacing $(\rho,
\LL\varphi, D_1,M_1)$ with $(\rho', \LL\rho, D_1', 2 M_{2,T})$ respectively.

For $\Pi_2$, thanks to estimate \eqref{par-G-es} and the mean value theorem, we immediately find
\begin{equation*}
  |\Pi_2| \leq  \|\partial_x G(t_1)\|_{L^\infty} \xi  \leq \Big( \|F_0\|_{L^\infty}M_{2,T}+M_1^2\|H_0\|_{L^\infty} \Big) \xi.
\end{equation*}
Hence, based on the above analysis, and using explicit estimates of $\xi\int_\xi^{r_0 + \xi} \frac{\omega(\eta)}{\eta^{1+\alpha}}\dd\eta$
and $\omega(\xi) \xi^{1-\alpha}$ as in Lemma \ref{lem:MOCes-u}, we can conclude estimate \eqref{par-u-diff-es}.
\end{proof}

\section{Appendix: commutator estimates}\label{sec:append}

We first present two Kato-Ponce type commutator estimates.
\begin{lemma}\label{lem:comm}
Let $x\in \R^d$ or $\T^d$, and $s\geq 0$. Then there exists a constant $C=C(s,d)>0$ so that
\begin{equation}\label{eq:comm-es}
  \|[\Lambda^s \nabla, f,g]\|_{L^2} \leq C \big( \|\nabla f\|_{L^\infty} \|g\|_{\dot H^s} + \|f\|_{\dot H^s} \|\nabla g\|_{L^\infty} \big),
\end{equation}
and 
\begin{equation}\label{eq:comm-es3}
  \|[\Lambda^s \nabla,f]g\|_{L^2} \leq C \big(\|\nabla_x f\|_{L^\infty} \|g\|_{\dot H^s} + \|f\|_{\dot H^{s+1}} \|g\|_{L^\infty} \big).
\end{equation}

\end{lemma}

\begin{proof}
We here only consider $x\in \R^d$, and the case of $\T^d$ can be similarly extended.
We first recall the following Kato-Ponce type commutator estimate proved in \cite[Corollary 1.4]{Li19}: for $s>-1$ suppose $A^s$ is a differential operator such that
its symbol $\widehat{A^s}(\zeta)$ is a homogeneous function of degree $s+1$ and
$\widehat{A^s}(\zeta)\in C^\infty(\mathbb{S}^{d-1})$, then for $1<p<\infty$ 
and for any $s_1,s_2\geq 0$ with $s_1+s_2 =s$, we have
\begin{align}\label{comm-es-Li}
  \Big\|A^s(f\,g) - \sum_{|\gamma|\leq s_1} \frac{1}{\gamma !} \partial^\gamma f A^{s,\gamma}g
  - \sum_{|\sigma|< s_2} \frac{1}{\sigma !} \partial^\sigma g\, A^{s,\sigma}f
  \Big \|_{L^p}
  \leq C \|\Lambda^{s_1} f\|_{\mathrm{BMO}} \|\Lambda^{s_2} g\|_{L^p},
\end{align}
where $C= C(s,s_1,s_2, p, d)$, $\gamma = (\gamma_1,\cdots,\gamma_d)\in \N^d$, $\partial^\gamma=\partial^\gamma_x = \partial_{x_1}^{\gamma_1}\cdots \partial_{x_d}^{\gamma_d}$,
$|\gamma| = \sum_{j=1}^d \gamma_j$, $\gamma ! =\gamma_1!\cdots \gamma_d!$, and the operators $A^{s,\gamma}$, $\Lambda^s$
are defined via the Fourier transform as
\begin{align*}
  \widehat{A^{s,\gamma} f}(\zeta) := i^{-|\gamma|} \partial^\gamma_\zeta\big(\widehat{A^s}(\zeta)\big)\, \hat{f}(\zeta), \quad
  \textrm{and}\quad \widehat{\Lambda^s f}(\zeta) := |\zeta|^s \hat{f}(\zeta).
\end{align*}
In order to prove \eqref{eq:comm-es}, we let $A^s = \Lambda^s \partial_{x_j}$ ($j=1,\cdots,d$), $s_1=1$, $s_2= s$, $p=2$, and it follows that
\begin{align*}
  \|[\Lambda^s\partial_{x_j}, f,g]\|_{L^2 } & = \|\Lambda^s\partial_{x_j} (f\,g) - f\, (\Lambda^s\partial_{x_j}g) - g\, (\Lambda^s\partial_{x_j} f)\|_{L^2} \\
  & \lesssim \sum_{|\gamma|=1} \|\partial^\gamma f\, A^{s,\gamma}g\|_{L^2} + \sum_{1\leq |\sigma| <s} \|\partial^\sigma g \,A^{s,\sigma} f\|_{L^2}
  + \|\Lambda f\|_{\mathrm{BMO}} \|\Lambda^s g\|_{L^2} \\
  & \lesssim \|\nabla f\|_{L^\infty} \|\Lambda^s g\|_{L^2} +  \|\nabla g\|_{L^\infty} \|\Lambda^s f\|_{L^2}
  +  \sum_{2\leq |\sigma| <s} \|\partial^\sigma g\|_{L^{\frac{2(s-1)}{|\sigma|-1}}} \|A^{s,\sigma} f\|_{L^{\frac{2(s-1)}{s-|\sigma|}}},
\end{align*}
where in the last line we also used the Calder\'on-Zygmund theorem. Note that $A^{s,\sigma}$ is a multiplier operator with symbol $\widehat{A^{s,\sigma}}(\zeta)$ a homogeneous function of order $s+1-|\sigma|$, the Calder\'on-Zygmund theory also implies that for every $2\leq |\sigma|<s$,
\begin{align*}
  \|A^{s,\sigma} f\|_{L^{\frac{2(s-1)}{s-|\sigma|}}} \lesssim \|\Lambda^{s+1-|\sigma|} f\|_{L^{\frac{2(s-1)}{s-|\sigma|}}} \lesssim \|\Lambda^{s-|\sigma|} \nabla f\|_{L^{\frac{2(s-1)}{s-|\sigma|}}},
\end{align*}
thus by using the following interpolation inequalities (e.g. see \cite[Pg. 28 and Lemma 2.10]{Li19}) that for every $2\leq |\sigma|< s$,
\begin{align*}
  \|\partial^\sigma g\|_{L^{\frac{2(s-1)}{|\sigma|-1}}} \lesssim \|\nabla g\|_{L^\infty}^{\frac{s-|\sigma|}{s-1}} \|\Lambda^s g\|_{L^2}^{\frac{|\sigma|-1}{s-1}},
  \quad \textrm{and}\quad \|\Lambda^{s-|\sigma|}\nabla f\|_{L^{\frac{2(s-1)}{s-|\sigma|}}} \lesssim \|\Lambda^{s-1} \nabla f\|_{L^2}^{\frac{s-|\sigma|}{s-1}}
  \|\nabla f\|_{L^\infty}^{\frac{|\sigma|-1}{s-1}},
\end{align*}
we infer that
\begin{align*}
  \sum_{2\leq |\sigma| <s} \|\partial^\sigma g\|_{L^{\frac{2(s-1)}{|\sigma|-1}}} \|A^{s,\sigma} f\|_{L^{\frac{2(s-1)}{s-|\sigma|}}}
  & \lesssim \big(\|\nabla g\|_{L^\infty} \|\Lambda^s f\|_{L^2} \big)^{\frac{s-|\sigma|}{s-1}}
  \big(\|\nabla f\|_{L^\infty} \|\Lambda^s g\|_{L^2} \big)^{\frac{|\sigma|-1}{s-1}} \\
  & \lesssim \|\nabla f\|_{L^\infty} \|\Lambda^s g\|_{L^2}  +  \|\nabla g\|_{L^\infty} \|\Lambda^s f\|_{L^2} .
\end{align*}
Hence gathering the above estimates leads to \eqref{eq:comm-es}, as desired.

Estimate \eqref{eq:comm-es3} is more or less classical, and it can also be proved by the same argument as above, thus we omit the details.
\end{proof}

The following commutator estimate involving with L\'evy operator $\LL$ plays an important role in our local well-posedness result.
\begin{lemma}\label{lem:com-es}
  Let $x\in \R$ or $\T$. Let $\LL$ be the L\'evy operator given by \eqref{Lop-exp} with kernel function $\phi(x)=\phi(-x)\in C^4(\R\setminus\{0\})$
satisfying assumptions (A1)(A2) with $\alpha\in (0,2)$, and let the operator $\sqrt{C'\mathrm{Id} +\LL}$ be given via Fourier transform as \eqref{def:sqrtL}. Then we have
\begin{equation}\label{eq:comm-es0}
  \|[\sqrt{C' \mathrm{Id} + \LL\,}, g]f\|_{L^2} \leq C \|f\|_{L^2}\|g\|_{C^{\frac{\alpha}{2}+\epsilon}},\quad \textrm{with\; $\epsilon>0$},
\end{equation}
with $C>0$ a constant depending on $\LL,s,\epsilon$.
\end{lemma}

\begin{remark}\label{rmk:comm-es}
  Note that estimate \eqref{eq:comm-es0} is a suitable generalization of the following commutator estimate (see \cite[Pg. 32]{DKRT})
\begin{equation}\label{eq:comm-es2}
  \|[\Lambda^{\frac{\alpha}{2}}, g]f\|_{L^2} \leq C \|f\|_{L^2}\|g\|_{C^{\frac{\alpha}{2}+\epsilon}},\quad \textrm{with $\epsilon>0$}.
\end{equation}
\end{remark}

We first recall some basic knowledge of paradifferential calculus. One can choose two nonnegative radial functions $\chi, \varphi\in C^\infty_c(\mathbb{R} )$ be
supported respectively in the ball $\{\zeta\in \mathbb{R} :|\zeta|\leq \frac{4}{3} \}$ and the annulus $\{\zeta\in
\mathbb{R} : \frac{3}{4}\leq |\zeta|\leq  \frac{8}{3} \}$ such that (see \cite{BCD11})
\begin{equation*}
  \chi(\zeta)+\sum_{k\in \mathbb{N}}\varphi(2^{-k}\zeta)=1, \quad \forall~ \zeta\in \mathbb{R} .
\end{equation*}
For every $ f\in S'(\R)$, we define the non-homogeneous Littlewood-Paley operators as follows
\begin{equation}\label{LPop}
  \Delta_{-1}f:=\chi(D)f; \quad \, \quad\Delta_{k}f:=\varphi(2^{-k}D)f,\;\;\;S_k f:=\sum_{-1\leq l\leq k-1} \Delta_l f,\;\;\;\forall~ k\in \mathbb{N}.
\end{equation}
Now for $s\in \mathbb{R}, (p,r)\in[1,+\infty]^2$, the inhomogeneous Besov space is defined as
\begin{equation}\label{Besov-spr}
  B^s_{p,r}:=\Big\{f\in\mathcal{S}'(\mathbb{R} );\|f\|_{B^s_{p,r}}:=\|\{2^{js}\|\Delta
  _k f\|_{L^p}\}_{k\geq -1}\|_{\ell^r }<\infty  \Big\}.
\end{equation}
In particular, $H^s = B^s_{2,2}$ for every $s\geq 0$.
Besides, Bony's decomposition yields
\begin{equation*}
  f\,g = T_f g + T_g f + R(f,g),
\end{equation*}
with
\begin{equation*}
  T_f g:= \sum_{k\in \N} S_{k-1}f \Delta_k g,\quad R(f,g)=\sum_{k\geq -1}\Delta_k f \widetilde{\Delta}_k g,\quad \widetilde{\Delta}_k := \Delta_{k-1} + \Delta_k + \Delta_{k+1}.
\end{equation*}

\begin{proof}[Proof of Lemma \ref{lem:com-es}]
We here prove estimate \eqref{eq:comm-es0} for $x\in \R$, and the periodic case can be easily adapted.
By using Bony's decomposition, we have the following splitting
\begin{equation}\label{sqL-decom}
\begin{split}
  \sqrt{C'\mathrm{Id} +\LL\,} (f\, g) & = \sqrt{C'\mathrm{Id} +\LL\,} T_f g + \sqrt{C'\mathrm{Id} +\LL\,} T_g f + \sqrt{C'\mathrm{Id} +\LL\,} R(f,g) : = J_1 + J_2 + J_3, \\
  (\sqrt{C'\mathrm{Id} +\LL\,} f)\, g & = T_{\sqrt{C'\mathrm{Id} +\LL} f} g  + T_g (\sqrt{C'\mathrm{Id} +\LL\,} f) +  R(\sqrt{C'\mathrm{Id} +\LL\,} f,g): = J_4 + J_5 + J_6.
\end{split}
\end{equation}
Through standard paraproduct calculus and Lemma \ref{lem:symb}, the terms $J_1$, $J_3$, $J_4$, $J_6$ can be treated as follows:
\begin{align*}
  \|J_1\|_{L^2}^2 & = \sum_{q\geq -1}\|\Delta_q \sqrt{C'\mathrm{Id} + \LL\,}T_f g\|_{L^2}^2 \lesssim \sum_{|k-q|\leq 4,k\in \N} (C + C 2^{q\alpha}) \|\Delta_q\big(S_{k-1}f \, \Delta_k g\big)\|_{L^2}^2 \\
  & \lesssim \sum_{k\in \N} (C + C 2^{k \alpha}) \|S_{k-1} f\|_{L^2}^2 \|\Delta_k g\|_{L^\infty}^2 \lesssim \|f\|_{L^2}^2 \|g\|_{B^{\alpha/2}_{\infty,2}}^2 \lesssim \|f\|_{L^2}^2 \|g\|_{C^{\frac{\alpha}{2}+\epsilon}}^2, \\
  \|J_3\|_{L^2}^2 & = \sum_{q\geq -1} \|\Delta_q \sqrt{C'\mathrm{Id}+ \LL\,} R(f,g)\|_{L^2}^2 \lesssim \sum_{q\geq -1} \sum_{k\geq q-2} (C + C 2^{q\alpha}) \| \Delta_q \big( \Delta_k f\, \widetilde{\Delta}_k g\big)\|_{L^2}^2 \\
  & \lesssim \|f\|_{L^2} \sum_{q\geq -1} \sum_{k\geq q-2} 2^{(q-k)\alpha} 2^{k \alpha} \|\widetilde{\Delta}_k g\|_{L^\infty}^2 \lesssim \|f\|_{L^2}^2 \|g\|_{C^{\frac{\alpha}{2}+\epsilon}}^2, \\
  \|J_4\|_{L^2}^2 &  = \sum_{q\geq -1}\|\Delta_q T_{\sqrt{C'\mathrm{Id} + \LL}f} g\|_{L^2}^2 \lesssim \sum_{|k-q|\leq 4, k\in \N} (C + C 2^{k\alpha}) \|S_{k-1}f\|_{L^2}^2 \, \|\Delta_k g\|_{L^\infty}^2 \lesssim \|f\|_{L^2}^2 \|g\|_{C^{\frac{\alpha}{2}+\epsilon}}^2, \\
  \|J_6\|_{L^2}^2 & = \sum_{q\geq -1} \|\Delta_q R(\sqrt{C'\mathrm{Id}+ \LL\,}f,g)\|_{L^2}^2 \lesssim \sum_{q\geq -1} \sum_{k\geq q-2} (C + C 2^{k\alpha}) \| \Delta_k f\|_{L^2}^2 \| \widetilde{\Delta}_k g \|_{L^2}^2 \\
  & \lesssim \|f\|_{L^2}^2 \sum_{k\geq -1} (k+2) 2^{k \alpha} \|\widetilde{\Delta}_k g\|_{L^\infty}^2 \lesssim  \|f\|_{L^2}^2 \|g\|_{C^{\frac{\alpha}{2}+\epsilon}}^2.
\end{align*}

Next we are devoted to the estimation of $J_2 - J_5$. For every $q\geq -1$, observe that
\begin{align}\label{J2-J5-dec}
  \Delta_q J_2 -\Delta_q J_5 & = \Delta_q \sqrt{C'\mathrm{Id} +\LL\,} T_g f - \Delta_q T_g(\sqrt{C'\mathrm{Id} +\LL\,}f) \nonumber \\
  & = \sum_{|k-q|\leq 4, k\in \N} \Delta_q \Big( \sqrt{C'\mathrm{Id} + \LL\,} \big(S_{k-1}g \Delta_k f\big) -
  S_{k-1} g \big(\sqrt{C' \mathrm{Id} +\LL\,} \Delta_k f \big)\Big) \nonumber \\
  & = : \sum_{|k-q|\leq 4, k\in \N} \Pi_{k,q}.
\end{align}
We first consider the case that $q\geq -1$ is large enough. Following the idea of \cite{KPV93},
and recalling that $A(\zeta)$ defined by \eqref{LKf} is the symbol of operator $\LL$,
we use the Fourier transform to write $\Pi_{k,q}(x)$ as follows
\begin{align}
  \Pi_{k,q}(x) 
  = & \iint \Big(\sqrt{C' + A(\zeta + \eta)} - \sqrt{C' + A(\zeta)}  \Big) \varphi_{2^q}(\zeta + \eta) \chi_{2^{k-2}}(\eta)
  \varphi_{2^k}(\zeta) \widehat{f}(\zeta) \widehat{g}(\eta) e^{i(\zeta+\eta)x} \dd \zeta \dd \eta \nonumber \\
  = & \iint m_{k,q}(\zeta, \eta)\,
  \varphi_{2^k}(\zeta) \widehat{f}(\zeta)\, \chi_{2^{k-2}}(\eta) |\eta| \widehat{g}(\eta)\, e^{i(\zeta+\eta)x}\, \dd \zeta \dd \eta, \label{Pi-kq-exp2}
\end{align}
where $(\varphi_r,\chi_r,\widetilde\varphi_r,\widetilde\chi_r)(\cdot) := (\varphi,\chi,\widetilde\varphi,\widetilde\chi)(\frac{\cdot}{r})$ for $r>0$,
\begin{align}\label{m-kq}
  m_{k,q}(\zeta,\eta) := \frac{\sqrt{C' + A(\zeta + \eta)} - \sqrt{C' + A(\zeta)} }{ |\eta|} \varphi_{2^q}(\zeta + \eta)
  \widetilde{\chi}_{2^{k-2}}(\eta) \widetilde{\varphi}_{2^k}( \zeta),
\end{align}
and $\widetilde{\varphi},\widetilde{\chi}\in C^\infty_c(\R)$ such that $0\leq \widetilde{\varphi},\widetilde{\chi}\leq 1$ and
\begin{align*}
  \widetilde{\varphi}\equiv 1\; \textrm{on}\; \big\{\frac{3}{4}\leq |\zeta|\leq \frac{8}{3}\big\},\;\;
  \mathrm{supp}\,\widetilde{\varphi}\subset \{\frac{2}{3}\leq |\zeta|\leq 3\},\;\;
  \widetilde{\chi}\equiv 1\; \textrm{on}\; \big\{|\zeta|\leq \frac{4}{3}\big\},\;\;
  \mathrm{supp}\,\widetilde{\chi}\subset \{|\zeta|\leq \frac{3}{2}\}.
\end{align*}
We also have
\begin{align*}
  \Pi_{k,q}(x) = \iint h_{k,q}(y,z)\, \Delta_k f(x-y)\, S_{k-1} \Lambda g(x-z) \dd y \dd z,
\end{align*}
with
\begin{align}\label{h-kq}
  h_{k,q}(y,z) = C_0 \iint m_{k,q}(\zeta,\eta) e^{i( y\zeta + z\eta)}\,\dd \zeta \dd \eta.
\end{align}

Note that the assumption that $q$ is sufficiently large is mainly used to ensure the spectrum $\zeta+\eta$ and $\zeta$ in $m_{k,q}(\zeta,\eta)$
satisfies $|\zeta+\eta|,|\zeta|\geq \max\{a_0^{-1},1\}$, thus we may assume that $q\geq q_0$ with $q_0 :=7 + [\log_2 \max\{a_0^{-1},1\}]$.
Concerning $h_{k,q}$ in this case we have the following key property (whose proof is postponed later).
\begin{lemma}\label{lem:m-h-prop}
Let $q\in\N$ be large enough so that $q\geq q_0$, and $k\in \N$ be satisfying $|k-q|\leq 4$. Then $h_{k,q}(y,z)$ given by \eqref{h-kq} satisfies
\begin{equation}\label{eq:claim}
  \iint_{\R^2} |h_{k,q}(y,z)|\,\dd y\dd z \leq C 2^{k(\frac{\alpha}{2} -1)},
\end{equation}
with $C>0$ a constant independent of $k,q$.
\end{lemma}

With Lemma \ref{lem:m-h-prop} at our disposal, we derive that
\begin{align}\label{J2-J5-hi-es}
  \sum_{q\geq q_0}\|\Delta_q J_2 -\Delta_q J_5\|_{L^2}^2 & \leq \sum_{q\geq q_0} \sum_{|k-q|\leq 4,k\in\N} \|\Pi_{k,q}\|_{L^2}^2 \nonumber \\
  & \leq \sum_{q\geq q_0} \sum_{|k-q|\leq 4,k\in\N} \|h_{k,q}\|_{L^1(\R^2)}^2 \|\Delta_k f\|_{L^2}^2 \|S_{k-1}\Lambda g\|_{L^\infty}^2 \nonumber \\
  & \leq \sum_{q\geq q_0} \sum_{|k-q|\leq 4,k\in\N} 2^{k(\frac{\alpha}{2}-1)} \|\Delta_k f\|_{L^2}^2 \|S_{k-1}\Lambda g\|_{L^\infty}^2 \nonumber \\
  & \leq C \Big(\sum_{k\in\N}\|\Delta_k f\|_{L^2}^2\Big) \| g \|_{B^{\frac{\alpha}{2}}_{\infty,1}}^2
  \leq C \|f\|_{L^2}^2 \|g\|_{C^{\frac{\alpha}{2}+\epsilon}}^2.
\end{align}

Next we consider the remaining case $q\leq q_0= 7 + [\log_2 \max\{a_0^{-1},1\}] $.
By using \eqref{A-est2} and Plancherel's theorem, we directly obtain 
\begin{align}\label{J2-J5-lo-es}
  & \sum_{-1\leq q\leq q_0} \|\Delta_q J_2  - \Delta_q J_5\|_{L^2}^2 \nonumber \\
  \leq & C \sum_{-1\leq q\leq q_0} \sum_{|k-q|\leq 4,k\in\N}
  \Big(\|\Delta_q \sqrt{C'\mathrm{Id} + \LL\,} \big( S_{k-1}g \Delta_k f\big) \|_{L^2}^2 +
  \|S_{k-1} g \big(\Delta_k\sqrt{C' \mathrm{Id} +\LL\,} f \big)\|_{L^2}^2\Big)   \nonumber \\
  \leq & C \sum_{-1\leq q \leq q_0} \sum_{|k-q|\leq 4,k\in\N} \|\Delta_k f\|_{L^2}^2 \|S_{k-1}g\|_{L^\infty}^2
  \leq  C \|f\|_{L^2}^2 \|g\|_{L^\infty}^2.
\end{align}
Hence estimates \eqref{J2-J5-hi-es} and \eqref{J2-J5-lo-es} leads to
\begin{equation}\label{J2-J5-L2es}
  \|J_2 -J_5\|_{L^2}^2 \leq \sum_{q\geq q_0}\|\Delta_q J_2 -\Delta_q J_5\|_{L^2}^2
  + \sum_{-1\leq q \leq q_0} \|\Delta_q J_2 - \Delta_q J_5\|_{L^2}^2 \leq C \|f\|_{L^2}^2 \|g\|_{C^{\frac{\alpha}{2}+\epsilon}}^2 .
\end{equation}
Gathering \eqref{J2-J5-L2es} and the above estimates on $J_i$ ($i=1,3,4,6$) with decomposition \eqref{sqL-decom} yields the desired estimate \eqref{eq:comm-es0}.
\end{proof}

It remains to prove Lemma \ref{lem:m-h-prop}.
\begin{proof}[Proof of Lemma \ref{lem:m-h-prop}]
  We first study the differentiability property of $m_{k,q}$. Notice that
\begin{align}\label{m-kq2}
  m_{k,q}(\zeta,\eta) = \int_0^1 \Big(\frac{\partial}{\partial \zeta}\sqrt{C' + A(\zeta + \tau \eta)}\Big) \dd \tau \, \mathrm{sgn}(\eta)
  \varphi_{2^q}(\zeta + \eta)  \widetilde{\chi}_{2^{k-2}}(\eta) \widetilde{\varphi}_{2^k}( \zeta),
\end{align}
with $\mathrm{sgn}(\eta)$ the usual sign function.
Thanks to estimate \eqref{nd-sqrAz-es} and the support property, the multiplier $m_{k,q}(\zeta,\eta)$ given by \eqref{m-kq2} satisfies
\begin{align}\label{m-kq-bdd}
  |m_{k,q}(\zeta,\eta)| \leq C \int_0^1 |\zeta + \tau \eta|^{\frac{\alpha }{2}-1} \dd \tau \,
  \varphi_{2^q}(\zeta + \eta)  \widetilde{\chi}_{2^{k-2}}(\eta) \widetilde{\varphi}_{2^k}( \zeta)
  \leq C 2^{k(\frac{\alpha}{2}-1)}\widetilde{\chi}_{2^{k-2}}(\eta) \widetilde{\varphi}_{2^k}( \zeta) ,
\end{align}
\begin{align}\label{pa-m-zeta-es}
  \big|\nabla_{\zeta,\eta} m_{k,q}(\zeta,\eta) \big| \lesssim &  \int_0^1 \Big| \frac{\partial^2}{\partial \zeta^2}\sqrt{C' + A(\zeta + \tau \eta)}\Big| \dd \tau
  \varphi_{2^q}(\zeta + \eta) \widetilde{\chi}_{2^{k-2}}(\eta) \widetilde{\varphi}_{2^k}( \zeta)  \nonumber \\
  & + 2^{-k} \int_0^1 |\zeta + \tau \eta|^{\frac{\alpha }{2}-1} \dd \tau \,\Big(\widetilde{\chi}_{2^{k-2}}(\eta) + |\widetilde{\chi}'(2^{-(k-2)}\eta)| \Big)
  \Big(\widetilde{\varphi}_{2^k}( \zeta) +   |\widetilde{\varphi}'(2^{-k}\zeta)| \Big) \nonumber \\
  \leq\,& C 2^{(\frac{\alpha}{2}-2)k} \Big(\widetilde{\chi}_{2^{k-2}}(\eta) + |\widetilde{\chi}'(2^{-(k-2)}\eta)| \Big) \Big(\widetilde{\varphi}_{2^k}( \zeta)  + |\widetilde{\varphi}'(2^{-k}\zeta)| \Big),
\end{align}
and for $l=2,3$,
\begin{align}\label{pa-m-et-es}
  \big|\nabla^l_{\zeta,\eta} m_{k,q}(\zeta,\eta)\big| \leq\, C 2^{(\frac{\alpha}{2}- l-1)k}   \bigg(\sum_{j=0}^l \Big|\frac{\dd^j \widetilde{\chi}}{\dd \eta^j}(2^{-(k-2)}\eta)\Big| \bigg)
  \bigg(\sum_{j=0}^l \Big|\frac{\dd^j \widetilde{\varphi}}{\dd \eta^j}(2^{-k}\zeta)\Big| \bigg).
\end{align}
where $\nabla_{\zeta,\eta}=(\partial_\zeta,\partial_\eta)$ is the vector-valued differential operator, and $C>0$ is a constant independent of $k,q$.

From estimate \eqref{m-kq-bdd}, it directly follows that
\begin{equation}\label{h-kq-bdd}
  |h_{k,q}(y,z)| \leq C 2^{k(\frac{\alpha}{2}-1)}\iint_{\R^2} \widetilde{\chi}_{2^{k-2}}(\eta) \widetilde{\varphi}_{2^k}( \zeta) \dd \zeta \dd \eta \leq C 2^{k(\frac{\alpha}{2} +1)}.
\end{equation}
Based on estimate \eqref{pa-m-et-es}, we can also derive the crucial piecewise decay estimate of $h_{k,q}(y,z)$.
Noting that
\begin{align*}
  -i(y\partial_\zeta + z \partial_\eta)  e^{i(y\zeta + z \eta)} = (y^2 + z^2) e^{i (y\zeta + z \eta)},
\end{align*}
we find that for every $(y,z)\neq (0,0)$,
\begin{align*}
  h_{k,q}(y,z) & = C_0 \iint_{\R^2} m_{k,q}(\zeta,\eta) \bigg(\Big(-\frac{iy}{y^2+z^2}\partial_\zeta
  -\frac{iz}{y^2 + z^2 }\partial_\eta\Big)^3 e^{i(\zeta y + \eta z)}\bigg)\,\dd \zeta \dd \eta \nonumber \\
  & = C_0 \iint_{\R^2} \bigg(\Big(\frac{iy}{y^2+z^2}\partial_\zeta
  + \frac{iz}{y^2 + z^2 }\partial_\eta\Big)^3 m_{k,q}(\zeta,\eta)\bigg)\,e^{i(\zeta y + \eta z)}\,\dd \zeta \dd \eta,
\end{align*}
which leads to that for all $(y,z)\neq (0,0)$
\begin{align}\label{h-kq-dec1}
  |h_{k,q}(y,z)| \leq & \frac{C }{(y^2 + z^2)^{\frac{3}{2} }}
  \iint_{\R^2} \big|\nabla_{\zeta,\eta}^3 m_{k,q}(\zeta,\eta) \big| \dd \zeta \dd \eta \leq  \frac{C }{(y^2 + z^2)^{\frac{3}{2}}}  2^{k (\frac{\alpha}{2} -2)}.
\end{align}

Now we prove the desired estimate \eqref{eq:claim} relied on estimates \eqref{h-kq-bdd} and \eqref{h-kq-dec1}.
Let $r>0$ be a number chosen later, and by using the change of variables, we have
\begin{align*}
  \iint_{\R^2} |h_{k,q}(y,z)| \dd y\dd z & \leq
  \iint_{\sqrt{y^2 +z^2}\leq r}  |h_{k,q}(y,z)| \dd y \dd z + \iint_{\sqrt{y^2 + z^2}\geq r} |h_{k,q}(y,z)| \dd y \dd z \nonumber \\
  & \leq \iint_{\sqrt{y^2 +z^2}\leq r}C 2^{k(\frac{\alpha}{2} +1)}  \dd y\dd z
  + \iint_{\sqrt{y^2 + z^2}\geq r} \frac{C }{(y^2 + z^2)^{\frac{3}{2}}}  2^{k (\frac{\alpha}{2} -2)} \dd y  \dd z \nonumber \\
  & \leq C 2^{k (\frac{\alpha}{2} + 1)} r^2 + C2^{k(\frac{\alpha}{2}- 2)} r^{-1}.
\end{align*}
Hence estimate \eqref{eq:claim} follows by choosing $r = 2^{-k}$.
\end{proof}

\textbf{Acknowledgements.}
QM is supported by  Beijing Institute of Technology Research Fund Program for Young Scholars.
CT is partially supported by NSF grant DMS 1853001.
LX is partially supported by NSFC grants (Nos. 11671039 and 11771043).

\end{document}